\documentclass[11pt, oneside]{amsart}

\usepackage{amsmath,amsthm,amssymb,hyperref, mdframed}
\usepackage{mathrsfs} 

\usepackage{tikz, tikz-cd}
\usepackage[all]{xy}
\usetikzlibrary{decorations.markings}
\usetikzlibrary{backgrounds,shapes}
\usetikzlibrary{patterns,hobby}
        \usepackage{pgfplots}
        \pgfplotsset{compat=1.6}
\usepackage{subfig}
\usepackage{caption}
\usepackage{color}
\usepackage{hyperref}
\hypersetup{
    colorlinks=true, 
    linktoc=all,     
    linkcolor={mauve},  
    citecolor={plum},
    filecolor={plum},
    urlcolor={plum}
}

\definecolor{smoked}{RGB}{216, 212, 204}
\definecolor{mauve}{RGB}{200, 55, 171}
\definecolor{apricot}{RGB}{250, 144, 4}
\definecolor{sky}{RGB}{66, 169, 244}
\definecolor{plum}{RGB}{76, 0, 102}

\definecolor{lightmauve}{RGB}{232, 173, 220}
\definecolor{lightapricot}{RGB}{253, 211, 155}
\definecolor{lightsky}{RGB}{178, 221, 251}
\definecolor{lightplum}{RGB}{184, 153, 192}

\usepackage{geometry} 
\usepackage[parfill]{parskip}

\newcommand{\Z}{{\mathbb Z}}
\newcommand{\C}{{\mathbb C}}
\newcommand{\R}{{\mathbb R}}

\newcommand{\Q}{{\mathbb Q}}
\newcommand{\CP}{{\mathbb{CP}}}

\DeclareMathOperator\Orb{Orb}

\renewcommand{\Im}{\mathrm{Im}}

\DeclareMathOperator{\PSL}{\mathrm{PSL}}
\DeclareMathOperator{\psl}{\PSL_2\R}
\DeclareMathOperator{\SL}{\mathrm{SL}}
\DeclareMathOperator{\SU}{\mathrm{SU}}
\DeclareMathOperator{\Tr}{\mathrm{Tr}}

\theoremstyle{plain}                    
\newtheorem{thm}{Theorem}[section]
\newtheorem{thma}{Theorem}

\newtheorem{lem}[thm]{Lemma}

\newtheorem{cor}[thm]{Corollary}

\newtheorem{fact}[thm]{Fact}	
\newtheorem{assumption}[thm]{Assumption}

\theoremstyle{definition}
\newtheorem{defn}[thm]{Definition}

\theoremstyle{remark}
\newtheorem{rmk}[thm]{Remark}
\newtheorem{claim}[thm]{Claim}

\numberwithin{equation}{section}

\newenvironment{proofclaim}
 {\proof}
 {\endproof}

\newcommand{\surface}{\Sigma}
\DeclareMathOperator{\Rep}{Rep}
\newcommand{\RepDT}{\Rep^\mathrm{DT}_{\alpha}(\surface)}
\newcommand{\RepDTarg}[1]{\Rep^\mathrm{DT}_{#1}}
\DeclareMathOperator{\IntRep}{R\overset{\circ}{e}p}
\newcommand{\IntRepDT}[1]{{\smash{\IntRep}\vphantom{\Rep}}_{#1, \alpha}^\mathrm{DT}}
\newcommand{\IntRepDTarg}[2]{{\smash{\IntRep}\vphantom{\Rep}}_{#1, #2}^\mathrm{DT}}
\DeclareMathOperator{\PMod}{PMod}
\newcommand{\mcg}{\PMod(\surface)}
\newcommand{\rational}{\mathcal{Q}}

\title[Orbit closures for Deroin--Tholozan representations]{Mapping class group orbit closures for Deroin--Tholozan representations}

\author{Yohann Bouilly}
\address[Y.~Bouilly]{Current address: Université de Strasbourg, 5/7 rue Ren\'e Descartes, 67084 Strasbourg, France.}
\email{bouilly@math.unistra.fr}

\author{Gianluca Faraco}
\address[G.~Faraco]{Current address: Dipartimento di Matematica e Applicazioni U5, Universita` degli Studi di Milano-Bicocca, Via Cozzi 55, 20125 Milano, Italy.}
\email{gianluca.faraco@unimib.it}
\email{gianluca.faraco.math@gmail.com}

\author{Arnaud Maret}
\address[A.~Maret]{Sorbonne Université and Université Paris Cité, CNRS, IMJ-PRG, F-75005 Paris, France.\newline
Current address: Université de Strasbourg, IRMA, 7 rue Descartes, 67000 Strasbourg, France.}
\email{maret.arnaud@unistra.fr}

\date{\today}

\begin{document}

\begin{abstract} 
We prove that infinite mapping class group orbits are dense in the character variety of Deroin--Tholozan representations. In other words, the action is minimal except for finite orbits. Our arguments rely on the symplectic structure of the character variety, emphasizing this geometric perspective over its algebraic properties.
\end{abstract}

\maketitle

\section{Introduction}

\subsection{Motivations and results}
The deformation space of representations of the fundamental group of an oriented surface $S$ into a group $G$ is known as the \emph{character variety} of $S$ and $G$. When $G$ is an algebraic group, character varieties inherit the structure of algebraic varieties. On the other hand, when $G$ is a Lie group whose Lie algebra supports an orthogonal structure which is invariant by the adjoint action of $G$, character varieties carry a natural (stratified) symplectic structure named after Goldman (Section~\ref{sec:goldamn-symplectic-form}). Sometimes, for instance when $G$ is $\SL_2\C$, both the algebraic and the symplectic structures coexist. 

One way to deform representations of surface groups is to pre-compose $\pi_1 S\to G$ by an automorphism of $\pi_1 S$, giving rise to an action of the \emph{mapping class group of $S$} on the character variety (Section~\ref{sec:DT-mcg-dynamics}). It preserves both the algebraic and the symplectic structures of the character variety. Our goal is to continue the line of research initiated several decades ago to understand orbit closures for the mapping class group action on character varieties. We will consider the case of Deroin--Tholozan representations---in short, \emph{DT representations}---which are special kinds of representations of the fundamental group of a punctured sphere into $\PSL_2\R$ (Section~\ref{sec:dt-representations-origin}). They may be thought of as rank-2 cousins of unitary representations in various ways. For instance, even though DT representations have Zariski dense image in the non-compact Lie group $\PSL_2\R$, their deformation spaces are compact character varieties which we call \emph{DT components}. DT representations are also \emph{totally elliptic} in the sense that they map every \emph{simple} closed curve to an elliptic element of $\PSL_2\R$. Our main result is the following.

\begin{thma}[Theorem~\ref{thm:infinite-orbits-are-dense}]\label{thm-intro:minimal-up-to-finite-orbits}
Infinite mapping class group orbits of conjugacy classes of DT representations are dense in their corresponding DT component. In other words, the mapping class group action on DT components is \emph{minimal up to finite orbits}.
\end{thma}

Goldman raised the question of finding necessary and sufficient conditions on surface group representations to ensure that their mapping class group orbit is dense in the corresponding (relative) character variety~\cite[Problem~2.7]{goldman-conjecture}. The combination of Theorem~\ref{thm-intro:minimal-up-to-finite-orbits} with the classification of finite orbits of DT representations from~\cite{arnaud-sam} fully answers Goldman's question over DT components. Finite mapping class group orbits, nevertheless, remain a rare phenomenon that occurs for a small number of punctures on the underlying sphere and only for specific peripheral angles (see eg.~\cite[Section~7]{arnaud-sam}). For instance, Bronstein--Maret proved in~\cite{arnaud-sam} that if the underlying sphere has seven punctures or more, then every orbit is infinite. In that case, Theorem~\ref{thm-intro:minimal-up-to-finite-orbits} implies that the mapping class group action is \emph{minimal}---every orbit is dense. 

Theorem~\ref{thm-intro:minimal-up-to-finite-orbits} was already established in the context of 4-punctured spheres by Previte--Xia~\cite[Theorem~3.3]{previte-xia-minimality} where the authors cover both the case of DT representations and representations into $\SU(2)$. Since a complete classification of finite mapping class group orbits was not available at the time, Previte--Xia only concluded minimality for certain rational peripheral monodromy parameters (see~\cite[Theorem~1.2]{previte-xia-minimality}). A variation of their argument can be found in the work of Cantat--Loray~\cite[proof of Theorem~C p.2962]{cantat-loray}. Cantat--Loray actually proved a more precise result: a bounded and infinite mapping class group orbit in a relative $\SL_2\C$ character variety of a $4$-punctured sphere can only be made of real representations (into $\SL_2\R$ or $\SU(2)$) and it is dense in the unique compact component of the real points of the character variety~\cite[Theorem~C]{cantat-loray}.

The results of Previte--Xia and Cantat--Loray treat relative $\SL_2\C$ character varieties as affine algebraic varieties given explicitly by a family of cubic surfaces in $\C^3$. We propose a different approach to prove Theorem~\ref{thm-intro:minimal-up-to-finite-orbits} that purely relies on the symplectic structure of DT components. Instead of parametrizing conjugacy classes of representations using algebraic coordinates given by trace functions (Section~\ref{sec:transervality-n=4}), our strategy is to use the action-angle coordinates on DT components developed in~\cite{action-angle} (Section~\ref{sec:parametrization-by-triangle-chains}). By doing so, we are able to precisely characterize when the Poisson bracket of two angle functions (Section~\ref{sec:DT-ham-dynamics}) vanish in terms of angle coordinates (Lemmas~\ref{lem:poisson-bracket-vanish-iff-gamma-takes-some-values} and~\ref{lem:poisson-bracket-vanish-iff-gamma-takes-some-values-general-case}). Somewhat surprisingly, our computations relate the zero locus of specific Poisson brackets with a real Lagrangian submanifold of DT components. For instance, in the context of 4-punctured spheres, if $b$ and $d$ are two simple closed curves in a torso configuration (see Figure~\ref{fig:fourpunctsphere}), and $\beta$ and $\delta$ are the associated angle functions, then the equation $\{\beta,\delta\}=0$ defines a Lagrangian submanifold corresponding to $\mathbb{RP}^1\subset \CP^1$ under the symplectomorphism of~\cite{action-angle} between DT components of 4-punctured spheres and $\CP^1$ (Remark~\ref{rem:correspondance-RP1-zero-locus-Poisson-bracket}).

\begin{figure}[!ht]
    \centering
    \begin{tikzpicture}[scale=.65]
  \draw[apricot] (0,2.4) arc(90:270:.75 and 2.4) node[near start, above left]{$b$};
  \draw[lightapricot] (0,2.4) arc(90:270:-.75 and 2.4);
  \draw[lightplum] (-2,0) arc(180:0:2 and .5);
  \draw[plum] (-2,0) arc(180:0:2 and -.5) node[near start, below]{$d$};
  
  \draw (-2,3) to[out=-30,in=210] (2,3);
  \draw (3,2) to[out=210,in=150] (3,-2);
  \draw (2,-3) to[out=150,in=30] (-2,-3);
  \draw (-3,-2) to[out=30,in=-30] (-3,2);

  \draw (-3,2) to[out=-30,in=-40] (-2,3);
  \draw (-3,2) to[out=140,in=140] (-2,3);
  \draw (2,3) to[out=30,in=30] (3,2);
  \draw (2,3) to[out=210,in=210] (3,2);
  \draw (3,-2) to[out=-30,in=-30] (2,-3);
  \draw (3,-2) to[out=140,in=140] (2,-3);
  \draw (-2,-3) to[out=30,in=30] (-3,-2);
  \draw (-2,-3) to[out=210,in=210] (-3,-2);  
\end{tikzpicture}
\hspace{1cm}
\begin{tikzpicture}[scale=.9]
  \shade[ball color = mauve, opacity=.2] (0,0) circle (2cm);
  \draw[thick, sky] (0,-2) arc (270:90:1 and 2);
  \draw[dashed, thick, sky] (0,2) arc (90:270:-1 and 2);
  \fill[apricot] (0,2) circle (0.07) node[above]{\tiny $d\beta=0$};
  \fill[apricot] (0,-2) circle (0.07) node[below]{\tiny $d\beta=0$};

  \draw[sky, anchor=west] (2,2) node{$\{\beta,\delta\}=0$};
   \draw[sky, ->] (2,2) to[out=180, in=40] (1.2,.5);
\end{tikzpicture}
    \caption{On the left: a $4$-punctured sphere with two curves $b$ and $d$ in torso configuration. On the right: a DT component (which is symplectically a sphere) and the Lagrangian submanifold cut out by $\{\beta,\delta\}=0$.}
    \label{fig:fourpunctsphere}
\end{figure}
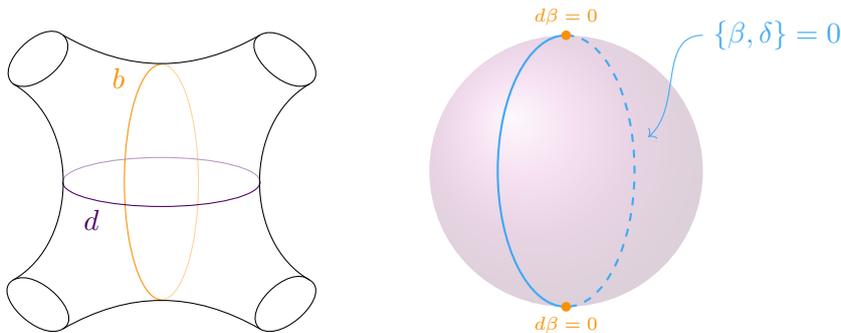

This result is used to identify a family of simple closed curves for which the differentials of the associated angle functions generate the cotangent space to the DT component at most points (Corollaries~\ref{cor:cotangent-space-generated-by-beta-delta-epsilon-n=4} and~\ref{cor:cotangent-space-generated-by-beta_i-delta_i-epsilon_i}).

\subsection{Related results}
Historically, mapping class group orbit closures of infinite orbits were first understood for representations into $\SU(2)$ and for surfaces of positive genus. The first milestone is due to Previte--Xia whose results claim that the mapping class group orbit of the conjugacy class of a representation with dense image in $\SU(2)$ is dense in the corresponding (relative) character variety~\cite{previte-xia-1, previte-xia-2}. Recently, Golsefidy--Tamam have identified inaccuracies in Previte--Xia's work when the genus of the surface is~1 or~2, and proposed a revised statement~\cite[Corollary~93]{golsefidy-tamam}. They also construct examples of representations with dense image in $\SU(2)$ for surfaces of genus 1 with 2 punctures and of genus 2 with 1 puncture, but whose corresponding mapping class group orbit is not dense, even though infinite~\cite[Section~8]{golsefidy-tamam}.

The complete study of orbit closures for representations of a genus-0 surface group into $\SU(2)$ was achieved by Golsefidy--Tamam in the same recent paper. They showed that if the mapping class group orbit of the conjugacy class of a representation with dense image in $\SU(2)$ is infinite, then it is dense in the associated relative character variety~\cite[Theorem~J]{golsefidy-tamam}.\footnote{It is worth mentioning that Golsefidy--Tamam also prove an analogous statement for character varieties of representations into $\SL_2\Z_p$, where $\Z_p$ is the ring of $p$-adic integers. The $p$-adic version of the theorem is somewhat weaker than its $\SU(2)$ counterpart as they only prove that the closure of an infinite orbit has nonempty interior~\cite[Theorem~K]{golsefidy-tamam}.} Their result generalizes the work of Previte--Xia on 4-punctured spheres~\cite{previte-xia-minimality} to spheres with an arbitrary number of punctures. Theorem~\ref{thm-intro:minimal-up-to-finite-orbits} is the analogue of Golsefidy--Tamam's result for DT representations instead of $\SU(2)$ representations.\footnote{Because DT representations only exist for punctured spheres, there is no analogue in positive genus.} Unlike in the positive genus case, it is important to assume the mapping class group orbit to be infinite because there exist examples of representations of the fundamental group of a 4-punctured sphere into $\SU(2)$ that have dense image but whose associated mapping class group orbit is finite.\footnote{They are the so-called \emph{Klein solution} and the \emph{237 elliptic solutions} discovered by Boalch~\cite{boalch-klein, boalch-237} and Kitaev~\cite{kitaev-237}, see also~\cite[Corollary~7.2 and the discussion before]{arnaud-sam}.} 

To completely understand mapping class group orbit closures in character varieties, it is necessary to identify all finite orbits. Substantial work has been produced on the topic by different authors over the past decades, so that we now have a complete understanding of finite mapping class group orbits of surface group representations into $\SL_2\C$. We refer the reader to~\cite[Section~1.3]{arnaud-sam} for a historical account. As for orbit closures, understanding finite orbits in the case of punctured spheres turns out to be the hardest nut to crack and also provides the richest zoology of finite orbits.

Moving away from character varieties of representations into $\SL_2\C$, Bouilly--Faraco answered Goldman's question on orbit closures~\cite[Problem~2.7]{goldman-conjecture} for representations of closed surface groups into compact abelian Lie groups. Note that when the target group is abelian, the character variety coincides with the space of representations (the conjugation action is trivial). They proved that the mapping class group orbit of a representation is dense if and only if the representation has dense image~\cite[Theorem~A]{yohann-gianluca}.

The study of minimality (or minimality up to finite orbits) is related to a weaker notion called \emph{almost minimality} which happens when almost every orbit is dense. Here, ``almost every orbit'' refers to the Liouville measure associated to the Goldman symplectic form on character varieties (Section~\ref{sec:goldamn-symplectic-form})---the so-called \emph{Goldman measure}. Almost minimality is a consequence of ergodicity. The mapping class group action on (relative) character varieties is known to be ergodic for many surface groups when the target Lie group is compact. This was first proved by Goldman for representations of any surface group into $\SU(2)$~\cite{goldman-ergodic}. Goldman’s results were later generalized by Pickrell and Xia to representations into any compact Lie group. Their generalization holds for surfaces of genus at least two with an arbitrary number of puntures and all peripheral monodromies, and for almost all peripheral monodromies in the case of a 1-punctured torus~\cite{pickrell-xia-1, pickrell-xia-2}. More recently, Goldman--Lawton--Xia~\cite{goldman-lawton-xia} established ergodicity for representations of the fundamental group of a 1-punctured torus into $\SU(3)$ for all peripheral monodromies---a particular case left open in Pickrell--Xia's paper. Ergodicity of the mapping class group action on DT components was established in~\cite{maret-ergodicity} for spheres with an arbitrary number of punctures, extending results of Goldman about the 4-punctured sphere~\cite{goldman-torus}.

There are plenty of other adventures to undertake around the topic of mapping class group dynamics on character varieties. Here are two examples of unknown territories to explore. Goldman's Conjecture~\cite[Conjecture~3.1]{goldman-conjecture} on the ergodicity of the mapping class group action on intermediate components of the $\PSL_2\R$ character variety of closed surface groups has only been proven in genus 2 by Marché--Wolff~\cite{marche-wolff} and remains open in higher genuses. Beyond ergodicty, there is the question of \emph{unique ergodicity}. It follows from the recent work of Cantat--Dupont--Martin-Baillon that an invariant ergodic measure for the mapping class group action on the relative $\SL_2\C$ character varieties of representations of the fundamental group of a 4-punctured sphere is either the Goldman measure or is supported on a finite orbit~\cite{stationary-measures}.\footnote{Their result is more general: the statement holds for \emph{stationary} ergodic measures, not only invariant ergodic measures.}  It would be interesting to determine whether a similar classification of invariant measures holds for character varieties on which we have a good understanding of orbit closures, such as DT components.

\subsection{Some ideas about the proof}
To prove Theorem~\ref{thm-intro:minimal-up-to-finite-orbits}, we need to prove that the closure of an infinite mapping class group orbit is the whole DT component. Since the Goldman measure is strictly positive, a closed set is the whole DT component if and only if it has full measure. By the ergodicity result of~\cite{maret-ergodicity} (Theorem~\ref{thm:ergodicity}), it is therefore enough to show that the closure of an infinite orbit always has nonempty interior to prove Theorem~\ref{thm-intro:minimal-up-to-finite-orbits}.
Constructing the desired open set is the heart of the argument and requires both an \emph{irrationality} and a \emph{transversality} result.

The irrationality statement will follow from Selberg's Lemma (Theorem~\ref{thm:selberg-lemma}). It implies that only finitely many rational multiples of $\pi$ can arise as rotation angles of elliptic elements in the image of a DT representation (Corollary~\ref{cor:finitely-many-rational-angles}). Involving Selberg's Lemma in this context is an idea that we borrowed from Cantat--Loray~\cite{cantat-loray}.

A \emph{transversality result} is a statement about generating the cotangent space at all (or at least at most) points in a DT component using differentials of angle functions (Section~\ref{sec:DT-ham-dynamics}) associated to an explicit family of simple closed curves on the underlying sphere. When transversality is achieved at a particular point, the Hamiltonian flows of the angle functions locally parametrize an open neighborhood of the point (Section~\ref{sec:local-parameterization}) (this explains the name ``transversality''). Being able to generate the cotangent space to a general character variety using trace functions associated to finitely many fundamental group elements is a consequence of Procesi's work~\cite{procesi}. The fact that those fundamental group elements can be taken to be \emph{simple} closed curves when the target Lie group is made of rank-2 matrices was observed by Goldman--Xia~\cite[Theorem~2.1]{goldman-xia}. Transversality results are usually obtained in an algebraic manner: first, identify sufficiently many trace functions to generate the coordinate ring of the character variety, then deduce that their differentials generate the cotangent space at every point. We propose a purely symplectic approach instead where we establish linear independence between differentials of angle functions by showing that some Poisson brackets do not vanish (Lemmas~\ref{lem:poisson-bracket-vanish-iff-gamma-takes-some-values} and~\ref{lem:poisson-bracket-vanish-iff-gamma-takes-some-values-general-case}).

More precisely, we will identify a \emph{minimal} set of simple closed curves (where minimal means that the number of curves is equal to the dimension of the DT component) whose corresponding angle functions have the following properties:
\begin{itemize}
    \item Their differentials generate the cotangent space to the DT component at every point in an open dense subset (Corollaries~\ref{cor:cotangent-space-generated-by-beta-delta-epsilon-n=4} and~\ref{cor:cotangent-space-generated-by-beta_i-delta_i-epsilon_i}, as well as Remark~\ref{rem:cotangent-space-generated-by-beta_i-delta_i-epsilon_i}).
    \item They essentially parametrize the DT component in the sense that prescribing a value for each angle function determines a finite set of points in the DT component (Corollaries~\ref{cor:finite-intersection-b-orbits-and-d-orbits-n=4} and~\ref{cor:finite-intersection-b-orbits-and-d-orbits}).
    \item The curves can be arranged as the vertices of a connected graph such that whenever two curves share an edge, the Poisson bracket of the corresponding angle functions almost never vanishes (Claims~\ref{claim:Poisson-bracket-d_i-d_i+1} and~\ref{claim:Poisson-bracket-d_i-d_i-1}).\footnote{This property is not a consequence of the first two in general. Consider, for instance, $\R^{2n}$ equipped with the standard symplectic form $\sum_i dx_i\wedge dy_i$. The coordinate functions $x_1,\ldots,x_n,y_1,\ldots,y_n$ satisfy the first two properties, but not the last one because the Poisson brackets $\{x_1,x_i\}$ and $\{x_1,y_i\}$ are identically zero whenever $i\geq 2$.}
\end{itemize}

The most challenging argument in the proof of Theorem~\ref{thm-intro:minimal-up-to-finite-orbits} consists in showing that an infinite orbit always contains at least one point---our \emph{preferred} point---at which transversality is achieved for the set of simple closed curves identified previously (Lemmas~\ref{lem:infinite-orbits-intersect-regular-fibers-at-infinitely-many-points}, \ref{lem:getting-rid-of-beta=pi}, and~\ref{lem:getting-rid-of-gamma_i=0-or-pi}) and which fulfills the expected irrationality condition (Section~\ref{sec:choosing-a-point}). Finding such a point is not immediate because, for instance, transversality fails to hold at more than finitely many points of the DT component. Once we have identified our preferred point, it automatically comes with an open neighborhood foliated by Hamiltonian flow lines by transversality. The irrationality property says that the value of at least one angle function at the point is an irrational multiple of $\pi$. Arguing using the close connection between Dehn twists and Hamiltonian flows of angle functions (Section~\ref{sec:DT-ham-dynamics}), we shall eventually conclude that each of the Hamiltonian flow lines foliating the open neighborhood of our preferred point is contained in the orbit closure, effectively giving us the open set we were looking for (Section~\ref{sec:open-set-U}).

In order to facilitate the arguments of Section~\ref{sec:proof}, we opt for a proof of Theorem~\ref{thm-intro:minimal-up-to-finite-orbits} by induction on the number of punctures on the underlying sphere. Previte--Xia's result from~\cite{previte-xia-minimality} for 4-punctured spheres (Theorem~\ref{thm:infinite-orbits-are-dense-n=4}) will serve as the base case. The parametrization of DT representations by triangle chains developed in~\cite{action-angle} clarifies the inductive scheme on the sphere's topology, as it aligns naturally with the ``gluing'' of representations. A similar inductive approach was also employed in~\cite{arnaud-sam}.

\subsection{Organization of the paper}
We start by providing a brief review of DT representations in Section~\ref{sec:background-DT-representations}, insisting both on the triangle chains parametrization and the subsequent action-angle coordinates (Section~\ref{sec:parametrization-by-triangle-chains}), as well as the Hamiltonian toric structure of DT components (Sections~\ref{sec:DT-ham-dynamics} and~\ref{sec:moment-map}).

Section~\ref{sec:transervality} contains all the transversality statements that we shall need to prove Theorem~\ref{thm-intro:minimal-up-to-finite-orbits} (Corollaries~\ref{cor:cotangent-space-generated-by-beta-delta-epsilon-n=4} and~\ref{cor:cotangent-space-generated-by-beta_i-delta_i-epsilon_i}). The main result characterizes the occurrences when the Poisson bracket of two angle functions vanishes (Lemmas~\ref{lem:poisson-bracket-vanish-iff-gamma-takes-some-values} and~\ref{lem:poisson-bracket-vanish-iff-gamma-takes-some-values-general-case}). We start with the case of the 4-punctured sphere (Section~\ref{sec:transervality-n=4}), before proceeding with the general case (Section~\ref{sec:general-case}).

We continue with some preliminaries in Section~\ref{sec:preliminaries}, including the digression on Selberg's Lemma (Section~\ref{sec:selberg}). We end Section~\ref{sec:preliminaries} by giving a proof of Theorem~\ref{thm-intro:minimal-up-to-finite-orbits} for 4-punctured spheres which an adaptation of Cantat--Loray's argument from~\cite{cantat-loray}. It is required as a base case of our inductive proof of Theorem~\ref{thm-intro:minimal-up-to-finite-orbits}, and we believe that it will help the reader to understand better the proof of the general case provided in Section~\ref{sec:proof}.

\smallskip

\textit{Notation:} The more lengthy proofs are organized into a sequence of claims, each accompanied by its own proof. To facilitate the reading, we use the diamond symbol~$\lozenge$ to indicate the end of an intermediate claim's proof, and keep the usual square symbol~$\square$ to mark the conclusion of the main statement's proof.

\smallskip

\subsection{Acknowledgments} We express our gratitude to Bill Goldman and Nicolas Tholozan for suggesting the question of minimality in the context of DT representations. This work also benefited from inspired conversations with Samuel Bronstein. Thanks to the anonymous referee for their careful read and the many meaningful suggestions. GF is a member of GNSAGA and expresses gratitude for their support. AM was partially supported by the European Research Council (ERC) under the European Union’s Horizon 2020 research and innovation program (Grant agreement No.~101096550).

\section{Background on DT representations}\label{sec:background-DT-representations}

\subsection{Overview}
We briefly introduce some aspects of DT representations that are relevant for this note. Similar recaps on DT representations can be found in~\cite{aaron-arnaud, arnaud-sam} which we partially duplicate here for the sake of self-containment. Besides recalling their fundamental properties (Sections~\ref{sec:dt-representations-origin} and~\ref{sec:goldamn-symplectic-form}) and their parametrization by triangle chains (Section~\ref{sec:parametrization-by-triangle-chains}), we insist on some dynamical aspects, covering both the mapping class group action (Section~\ref{sec:DT-mcg-dynamics}) and the toric structure (Sections~\ref{sec:DT-ham-dynamics} and~\ref{sec:moment-map}).

\subsection{Origin and main properties}\label{sec:dt-representations-origin}
For every integer $n\geq 3$, we fix an oriented topological sphere~$\surface$ along with a set $\mathcal{P}$ of $n$ punctures on $\surface$. The fundamental group $\pi_1\surface$ can be presented as
\begin{equation}\label{eq:geometric-presentation}
\pi_1\surface =\big\langle\, c_1,\dots,c_n\, |\, c_1\cdots c_n =1\,\big\rangle,
\end{equation}
by carefully choosing each $c_i$ as the homotopy class of a counterclockwise loop around a puncture of $\Sigma$---a \emph{peripheral loop}. Such a presentation of $\pi_1\surface$ is called \emph{geometric}. For an angle vector $\alpha\in(0,2\pi)^\mathcal{P}$, we introduce the \emph{$\alpha$-relative character variety} 
\[
\Rep_{\alpha}\big(\surface,\psl\big)
\]
as the space of conjugacy classes of representations $\rho\colon\pi_1\surface\rightarrow \psl$ which map peripheral loops around the punctures $p\in\mathcal{P}$ to elliptic elements of $\psl$ of rotation angles $\alpha_p$. The conjugacy class of a representation $\rho$ will be denoted by $[\rho]$. All $\alpha$-relative character varieties are smooth manifolds of dimension $2(n-3)$, sometimes with an isolated point, which happens exactly when $\sum_{p\in\mathcal{P}}\alpha_p$ is an integer multiple of $2\pi$. We shall adopt the following terminology

\begin{defn}\label{defn:rotangle}
    An elliptic element of $\psl$ has \emph{rotation angle} $\vartheta\in (0,2\pi)$ if it is conjugate to a matrix of the form
\[
\pm\begin{pmatrix}
    \cos(\vartheta/2) & \sin(\vartheta/2)\\
    -\sin(\vartheta/2) & \cos(\vartheta/2)
\end{pmatrix}.
\]
\end{defn}

Whenever the angle vector $\alpha$ satisfies
\begin{equation}\label{eq:angle-condition}
\sum_{p\in\mathcal P}\alpha_p> 2\pi(n-1),
\end{equation}
Deroin--Tholozan proved that $\Rep_{\alpha}(\surface_n,\,\psl)$ contains a smooth compact component diffeomorphic to $\CP^{n-3}$~\cite{deroin-tholozan}. An analogous compact component exists when instead $\sum_{p\in\mathcal P}\alpha_p< 2\pi$; the two are images of each other by the non-trivial outer automorphism of $\psl$. These compact components had already been identified by Benedetto--Goldman when $\Sigma$ is a 4-punctures spheres~\cite{benedetto-goldman}. According to Mondello, the compact components discovered by Deroin--Tholozan---which we shall call \emph{DT components}---are the unique compact components inside their respective $\alpha$-relative character varieties~\cite[Corollary~4.17]{mondello}. We will denote DT components by
\[
\RepDT\subset \Rep_{\alpha}(\surface,\psl)
\]
and refer to the representations whose conjugacy class lies in $\RepDT$ as \emph{DT representations}. 

DT representations have the fundamental property of being \emph{totally elliptic}, meaning that every \emph{simple} closed curve is mapped to an elliptic element inside $\psl$, see \cite[Lemma~3.2]{deroin-tholozan}. Their image is always a Zariski dense subgroup of $\psl$ which is rarely discrete and therefore often dense. Nevertheless, each DT representation has a geometric interpretation as the holonomy of a certain kind of hyperbolic cone metric on $\Sigma$ as explained in~\cite[Section~4]{deroin-tholozan} and further detailed by Fenyes--Maret in~\cite{aaron-arnaud}.

\subsection{Symplectic and Poisson structures}\label{sec:goldamn-symplectic-form}
The diffeomorphisms of~\cite{deroin-tholozan} between DT components and complex projective spaces are actually isomorphisms of symplectic manifolds. The symplectic structure on $\RepDT$ is given by the so-called Goldman symplectic form~$\omega_{\mathcal{G}}$~\cite{goldman-symplectic}. The symplectic form on $\CP^{n-3}$ is the Fubini--Study symplectic form with total volume equal to
\begin{equation}\label{eq:volume-and-lambda}
\frac{(\lambda\,\pi)^{n-3}}{(n-3)!},\,\, \text{where}\,\, \lambda=\sum_{p\in\mathcal{P}}\alpha_p-2\pi(n-1).
\end{equation}
Note that $\lambda$ is a positive constant by~\eqref{eq:angle-condition}. 

Associated to the Goldman symplectic form is the Poisson bracket 
\[
\{-,-\}\colon C^\infty\big(\RepDT\big)\times C^\infty\big(\RepDT\big)\to C^\infty\big(\RepDT\big)
\]
defined by $\{f,g\}=\omega_\mathcal{G}(X_f,X_g)$ for any pair of smooth functions $f$, $g$, and where $X_f$ and $X_g$ denote the Hamiltonian vector fields of $f$ and $g$. Equivalently, $\{f,g\}=df(X_g)=-dg(X_f)$. Goldman famously computed the Poisson brackets of certain kinds of smooth functions on character varieties constructed from conjugacy invariant functions of the target Lie group~\cite{goldman-invariant}. In Section~\ref{sec:transervality}, we will conduct independent computations to characterize the points inside DT components at which the Poisson bracket of two angle functions vanishes.

\subsection{Mapping class group dynamics}\label{sec:DT-mcg-dynamics}
The \emph{pure mapping class group} of $\surface$ is the group of isotopy classes of orientation-preserving homeomorphisms of $\surface$ that fix each puncture individually. We will denote it by $\PMod(\surface)$. It is naturally isomorphic to a subgroup of the group of outer automorphisms of $\pi_1\surface$ by the Dehn--Nielsen--Baer Theorem, as explained in~\cite[Theorem~8.8]{mcg-primer}. It therefore naturally acts on any relative character variety $\Rep_{\alpha}(\surface,\psl)$ by pre-composition and preserves each DT component. The mapping class group action also preserves the Goldman symplectic form and the associated probability Liouville measure~$\nu_{\mathcal{G}}$---the \emph{Goldman measure}. It turns out that the pure mapping class group of $\Sigma$ acts ergodically on every DT component. For 4-punctured spheres, ergodicity follows from Goldman's work~\cite{goldman-torus}; the general case was proven in~\cite{maret-ergodicity}.
\begin{thm}\label{thm:ergodicity}
The action of $\PMod(\surface)$ on $\RepDT$ is ergodic with respect to the Goldman measure for every angle vector $\alpha$ satisfying the angle condition~\eqref{eq:angle-condition}.
\end{thm}
A consequence of ergodicity is that almost every mapping class group orbit is dense. This property is sometimes called \emph{almost minimality}. The purpose of this paper is to go one step further and prove that every infinite orbit is dense.
\begin{rmk}
Interestingly, the mapping class group action on the non-compact components of $\Rep_{\alpha}(\surface,\psl)$ when $\surface$ is a 4-punctured sphere is partially properly discontinuous, but can have some ergodic regions too, as related by Palesi in~\cite{palesi}.
\end{rmk}

\subsection{Hamiltonian dynamics}\label{sec:DT-ham-dynamics}
There is another interesting action on $\RepDT$ coming from Hamiltonian dynamics. To any oriented simple closed curve $a$ on $\surface$ corresponds a smooth function 
\begin{equation}\label{eq:angle-function}
\vartheta_a\colon \RepDT\to (0,2\pi)
\end{equation}
called the \emph{angle function} of $a$. It is defined as follows: by total ellipticity of DT representations (Section~\ref{sec:dt-representations-origin}), the image of $a$ under any DT representation $\rho$ is elliptic. We define $\vartheta_a([\rho])$ to be the rotation angle of $\rho(a)$ (Definition \ref{defn:rotangle}). The Hamiltonian vector field of $\vartheta_a$ will be denoted by $X_a$. It is defined by the relation $\omega_\mathcal{G}(X_a,-)=d\vartheta_a$. The Hamiltonian flow of $\vartheta_a$ will in turn be denoted by 
\[
\Phi_a\colon \R\times \RepDT\to\RepDT.
\]
Deroin--Tholozan proved that $\Phi_a^t=\Phi_a(t,-)$ is $\pi$-periodic in the variable $t$~\cite[Proposition~3.3]{deroin-tholozan}. The orbits of $\Phi_a$ are therefore either singular points of $\vartheta_a$, or embedded circles of length $\pi$. In the sequel, we shall refer to these orbits as $a$-\textit{orbits}.

For any pair of disjoint simple closed curves, say $a_1$ and $a_2$, the Hamiltonian flows $\Phi_{a_1}$ and $\Phi_{a_2}$ just defined commute by a general result of Goldman~\cite{goldman-invariant}. As a consequence, given a pants decomposition of $\surface$---that means a collection of $n-3$ free homotopy classes of disjoint simple closed curves---the $n-3$ associated flows define a Hamiltonian torus action of $(\R/\pi\Z)^{n-3}$ on $\RepDT$. Such an action is maximal in the sense that $n-3$ is equal to half the dimension of $\RepDT$ and it was originally used by Deroin--Tholozan to relate $\RepDT$ and $\CP^{n-3}$ via Delzant's classification of symplectic toric manifolds, see~\cite{deroin-tholozan} for details.

The mapping class group action and the Hamiltonian dynamics of angles functions are related by the following classical identity, traditionally attributed to Goldman~\cite{goldman-ergodic}. To an un-oriented simple closed curve $a$ corresponds a element $\tau_a\in \mcg$ called the \emph{Dehn twist} along the curve $a$. A precise definition can be found in~\cite[Chapter~3]{mcg-primer}. As a diffeomorphism of $\RepDT$, the Dehn twist $\tau_a$ and the Hamiltonian flow $\Phi_a$ are related by
\begin{equation}\label{eq:goldman-trick}
\tau_a=\Phi_a^{t=\vartheta_a/2}.
\end{equation}
In other words, $\tau_a\!\cdot\![\rho]$ is obtained by rotating $[\rho]$ along its $a$-orbit by an angle $\vartheta_a([\rho])/2$. We can therefore think of the Dehn twist $\tau_a$ as a ``discretization'' of the Hamiltonian flow~$\Phi_a$. This interpretation of Dehn twists leads to the following fact.
\begin{fact}\label{fact:Dehn-twists-irrational-rotations}
When $\vartheta_a([\rho])$ is an irrational multiple of $\pi$, the $\tau_a$-orbit of $[\rho]$ is dense inside its $a$-orbit.
\end{fact}

\subsection{Parametrization by triangle chains}\label{sec:parametrization-by-triangle-chains}
Using total ellipticity of DT representations, it is possible to build a combinatorial model for $\RepDT$ in terms of polygonal objects in the hyperbolic plane called \emph{triangle chains}. We briefly recall Maret's construction of triangle chains from~\cite{action-angle} and refer the reader the original paper for more details. 

To associate a triangle chain to a DT representation, one first picks a pants decomposition, say $\mathcal{B}$, of $\surface$. We always work with so-called \emph{chained pants decompositions}, meaning that every pair of pants contains at least one of the punctures of $\surface$. The next step is to find a geometric presentation of $\pi_1\surface$ which is compatible with $\mathcal{B}$ in the following sense. Recall that a geometric presentation of $\pi_1\surface$ has generators $c_1,\ldots, c_n$ that satisfy $c_1\cdots c_n=1$ (Section~\ref{sec:dt-representations-origin}), where each $c_i$ is the homotopy class of a peripheral loop. Such a presentation is said to be \emph{compatible} with $\mathcal{B}$, if the $n-3$ pants curves lift to the fundamental group elements $b_i=(c_1\cdots c_{i+1})^{-1}$ for $i=1,\ldots,n-3$. The pants decomposition of $\surface$ given by $b_1,\ldots, b_{n-3}$ is called the \emph{standard} pants decomposition associated to the geometric generators $(c_1,\ldots, c_n)$. It is always possible to find a geometric presentation of $\pi_1\surface$ which is compatible with a given chained pants decomposition, as explained in~\cite[Appendix~B]{action-angle}. 

\begin{center}
\vspace{2mm}
\begin{tikzpicture}[scale=1.3, decoration={
    markings,
    mark=at position 0.6 with {\arrow{>}}}]
  \draw[postaction={decorate}] (0,-.5) arc(-90:-270: .25 and .5) node[midway, left]{$c_1$};
  \draw[black!40] (0,.5) arc(90:-90: .25 and .5);
  \draw[apricot, postaction={decorate}] (2,.5) arc(90:270: .25 and .5) node[midway, left]{$b_1$};
  \draw[lightapricot] (2,.5) arc(90:-90: .25 and .5);
  \draw[apricot, postaction={decorate}] (4,.5) arc(90:270: .25 and .5) node[midway, left]{$b_2$};
  \draw[lightapricot] (4,.5) arc(90:-90: .25 and .5);
  \draw[apricot, postaction={decorate}] (6,.5) arc(90:270: .25 and .5) node[midway, left]{$b_3$};
  \draw[lightapricot] (6,.5) arc(90:-90: .25 and .5);
  \draw[postaction={decorate}] (8,.5) arc(90:270: .25 and .5) node[midway, left]{$c_6$};
  \draw (8,.5) arc(90:-90: .25 and .5);
  
  \draw (.5,1) arc(180:0: .5 and .25) node[midway, above]{$c_2$};
  \draw[postaction={decorate}] (.5,1) arc(-180:0: .5 and .25);
  \draw (2.5,1) arc(180:0: .5 and .25)node[midway, above]{$c_3$};
  \draw[postaction={decorate}] (2.5,1) arc(-180:0: .5 and .25);
  \draw (4.5,1) arc(180:0: .5 and .25)node[midway, above]{$c_4$};
  \draw[postaction={decorate}] (4.5,1) arc(-180:0: .5 and .25);
  \draw (6.5,1) arc(180:0: .5 and .25)node[midway, above]{$c_5$};
  \draw[postaction={decorate}] (6.5,1) arc(-180:0: .5 and .25);
   
  \draw (0,.5) to[out=0,in=-90] (.5,1);
  \draw (1.5,1) to[out=-90,in=180] (2,.5);
  \draw (0,-.5) to[out=0,in=180] (2,-.5);
  
  \draw (2,.5) to[out=0,in=-90] (2.5,1);
  \draw (3.5,1) to[out=-90,in=180] (4,.5);
  \draw (2,-.5) to[out=0,in=180] (4,-.5);
  
  \draw (4,.5) to[out=0,in=-90] (4.5,1);
  \draw (5.5,1) to[out=-90,in=180] (6,.5);
  \draw (4,-.5) to[out=0,in=180] (6,-.5);
  
  \draw (6,.5) to[out=0,in=-90] (6.5,1);
  \draw (7.5,1) to[out=-90,in=180] (8,.5);
  \draw (6,-.5) to[out=0,in=180] (8,-.5);
\end{tikzpicture}
\vspace{2mm}
\end{center}
Note that a geometric presentation of $\pi_1\surface$ induces a bijection $\mathcal{P}\to \{1,\ldots,n\}$ which gives a labeling of the punctures. In practice, we will use this labeling to index variables, such as the angle vector $\alpha$, on $\{1,\ldots,n\}$ rather than on $\mathcal{P}$.

Now that we have found a compatible geometric presentation to our chosen pants decomposition $\mathcal{B}$, we can explain how to associate a triangle chain to a DT representation~$\rho$ which we will call the \emph{$\mathcal{B}$-triangle chain} of $\rho$. It is made of $n-2$ hyperbolic triangles as follows:
\begin{itemize}
    \item Draw the fixed points $C_1,\ldots,C_n$ of $\rho(c_1),\ldots,\rho(c_n)$ respectively---the \emph{exterior vertices}---and the fixed points $B_1,\ldots,B_{n-3}$ of $\rho(b_1),\ldots,\rho(b_{n-3})$ respectively---the \emph{shared vertices}. Here we are using that $\rho$ is totally elliptic in order to say that $\rho(b_1),\ldots,\rho(b_{n-3})$ are elliptic.
    \item Draw a geodesic segment between two of these points if the corresponding curves on $\surface$ belong to the same pair of pants. We end up with a chain of $n-2$ triangles whose vertices are $(C_1,C_2,B_1), (B_1,C_3,B_2),\ldots, (B_{n-3},C_{n-1},C_n)$.
\end{itemize}
\begin{center}
\begin{tikzpicture}[font=\sffamily, scale=1.2]
    
\node[anchor=south west,inner sep=0] at (0,0) {\includegraphics[width=10.8cm]{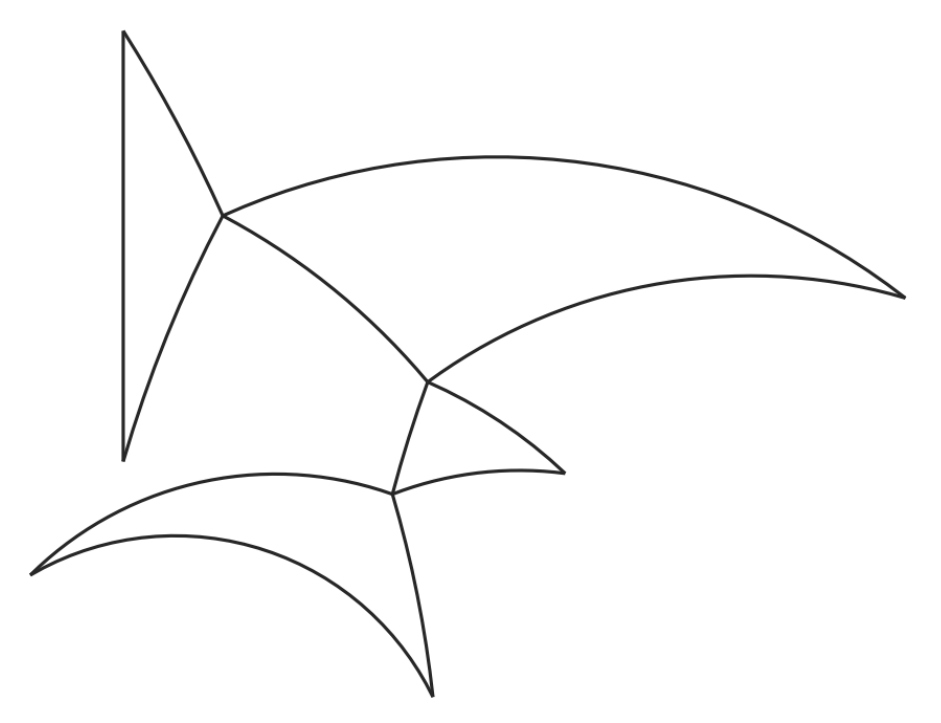}};

\begin{scope}
\fill (1.19,2.53) circle (0.07) node[left]{$C_1$};
\fill (1.19,6.62) circle (0.07) node[left]{$C_2$};
\fill (8.65,4.08) circle (0.07) node[below right]{$C_3$};
\fill (5.4,2.4) circle (0.07) node[below right]{$C_4$};
\fill (4.14,0.28) circle (0.07) node[below right]{$C_5$};
\fill (.31,1.43) circle (0.07) node[left]{$C_6$};
\end{scope}

\begin{scope}[apricot]
\fill (2.15,4.86) circle (0.07);
\fill (2.25,5.2) node{$B_1$};
\fill (4.1,3.27) circle (0.07) node[left]{$B_2$};
\fill (3.76,2.2) circle (0.07);
\fill (4.05,2.05) node{$B_3$};
\end{scope}
\end{tikzpicture}
\end{center}

Chains of triangles enjoy the following pleasant properties, see~\cite[Section \S3.2]{action-angle} for more details and proofs:
\begin{itemize}
    \item The interior angle at each exterior vertex $C_i$ is $\pi-\alpha_i/2$.
    \item The two interior angles on each side of a shared vertex $B_i$ add up to $\pi$.
    \item The triangles are clockwise oriented for the order of points given above.
\end{itemize}
It is possible that triangles overlap or that up to all but one triangle are degenerate to a single point. We shall say that the triangle chain is \emph{singular} whenever this happens and \emph{regular} otherwise. It is however not possible that all the triangles are degenerate since the total area of the triangles in the chain is $\lambda/2$, where $\lambda$ is the positive constant introduced in~\eqref{eq:volume-and-lambda}.

The realization of DT representations as triangle chains can be used to produce action-angle coordinates for $\RepDT$, as explained in~\cite{action-angle}. The action coordinates are given by the angle functions (introduced in~\eqref{eq:angle-function}) associated to the pants curves of $\mathcal{B}$. For simplicity, we shall write $\beta_i=\vartheta_{b_i}$. It is another important feature of triangle chains that the interior angles on both sides of the shared vertex $B_i$ are equal to $\pi-\beta_i/2$ and $\beta_i/2$, as shown on the picture below. The angle coordinates $\gamma_1,\ldots,\gamma_{n-3}$ are the angles ``between" consecutive triangles in the chain. It is worth mentioning that not all the angles $\gamma_1,\ldots,\gamma_{n-3}$ are well-defined when some triangles in the chain are degenerate.
\begin{center}
\begin{tikzpicture}[scale=1.2, font=\sffamily,decoration={markings, mark=at position 1 with {\arrow{>}}}]
\node[anchor=south west,inner sep=0] at (0,0) {\includegraphics[width=10.8cm]{fig/fig-triangles-black}};

\begin{scope}
\fill (1.19,2.53) circle (0.07) node[left]{$C_1$};
\fill (1.19,6.62) circle (0.07) node[left]{$C_2$};
\fill (8.65,4.08) circle (0.07) node[below right]{$C_3$};
\fill (5.4,2.4) circle (0.07) node[below right]{$C_4$};
\fill (4.14,0.28) circle (0.07) node[below right]{$C_5$};
\fill (.31,1.43) circle (0.07) node[left]{$C_6$};
\end{scope}

\begin{scope}[apricot]
\fill (2.15,4.86) circle (0.07);
\fill (2.25,5.2) node{$B_1$};
\fill (4.1,3.27) circle (0.07) node[left]{$B_2$};
\fill (3.76,2.2) circle (0.07);
\fill (4.05,2.05) node{$B_3$};
\end{scope}

\draw[thick, postaction={decorate}, sky] (2.8,5.1) arc (10:115:.7) node[near end, above right]{$\gamma_1$};
\draw[thick, postaction={decorate}, sky] (4.7,3) arc (-20:32:.7) node[near end, below right]{$\gamma_2$};
\draw[thick, postaction={decorate}, sky] (3.95,1.6) arc (-80:15:.6) node[near end, below right]{$\gamma_3$};

\draw[thick, apricot] (1.9,5.3) arc (109:250:.5);
\draw[apricot] (2.4,4.2) node{\tiny $\pi-\beta_1/2$};
\draw[thick, apricot] (4.5,3.55) arc (25:129:.5) node[midway, above]{\small $\pi-\beta_2/2$};
\draw[thick, apricot] (4.25,2.35) arc (5:78:.4) node[at end, left]{\small $\pi-\beta_3/2$};
\end{tikzpicture}
\end{center}

The action coordinates $\beta_1,\ldots, \beta_{n-3}$ along with the angle coordinates $\gamma_1,\ldots, \gamma_{n-3}$ completely parametrize $\RepDT$. Since these coordinates were constructed from the pants decomposition $\mathcal{B}$, the actions of the Dehn twists $\tau_{b_1},\ldots, \tau_{b_{n-3}}$ can be expressed in a simple way. If $[\rho]\in\RepDT$ is not fixed by $\tau_{b_i}$ and it has coordinates $(\beta_1,\ldots,\beta_{n-3},\gamma_1,\ldots,\gamma_{n-3})$, then the coordinates of $\tau_{b_i}\!\cdot\![\rho]$ are the same as those of $[\rho]$ except for
\begin{equation}\label{eq:action-of-tau_b_i}
\gamma_i(\tau_{b_i}\!\cdot\![\rho])=\gamma_i+\beta_i.
\end{equation}
For more details, the reader is invited to consult~\cite{action-angle}.

\subsection{Moment map}\label{sec:moment-map}
The functions $\beta_1,\ldots,\beta_{n-3}$ are essentially the components of a moment map for the Hamiltonian torus action described in Section~\ref{sec:DT-ham-dynamics}. A convenient choice of a moment map $\mu\colon \RepDT\to \R^{n-3}$ is the following. We define $\mu_i([\rho])$ as $1/\lambda$ times the non-negative area of the $(i+1)^{th}$ triangle in the $\mathcal{B}$-triangle chain of $[\rho]$, where $\lambda$ is the constant introduced in~\eqref{eq:volume-and-lambda}. In other words, $\mu=(\mu_1,\ldots,\mu_{n-3})$ with:
\begin{align*}
\mu_1&=\frac{1}{2\lambda}(\alpha_3-\beta_1 +\beta_2-2\pi)\\
\mu_2&= \frac{1}{2\lambda}(\alpha_4-\beta_2 +\beta_3-2\pi)\\
&\,\,\,\vdots\\
\mu_{n-3}&=\frac{1}{2\lambda}(\alpha_{n-1}+\alpha_n-\beta_{n-3}-2\pi).
\end{align*}
For convenience, we also introduce
\[
\mu_0=\frac{1}{2\lambda}(\alpha_1+\alpha_2 +\beta_1-4\pi).
\]
The image of $\mu$ inside $\R^{n-3}$ is isometric to the standard polytope of $\R^{n-3}$ of side length~$\lambda$. It is called the $\emph{moment polytope}$ of $\mu$ and was first described by Deroin--Tholozan in~\cite{deroin-tholozan}. The preimage of the interior of the moment polytope is an open and dense part of $\RepDT$ that foliates into Lagrangian tori via $\mu$. We denote it by 
\[
\IntRepDT{\mathcal{B}}(\surface) \subset \RepDT
\]
because it depends on the initial choice of pants decomposition $\mathcal{B}$. It consists of all $[\rho]$ whose $\mathcal{B}$-triangle chain is regular. The fibers of the moment map that lie inside $\IntRepDT{\mathcal{B}}(\surface)$ will be called \emph{regular fibers}. The preimage of the boundary of the moment polytope is made of smaller dimensional isotropic tori called \emph{irregular fibers}. The points in the irregular fibers are characterized by at least one triangle being degenerate in the $\mathcal{B}$-triangle chain. 

It was already observed in~\cite[Lemma~5.1]{arnaud-sam} that the inequalities $\mu_i\geq 0$ and $\alpha_i<2\pi$ imply that
\begin{equation}\label{eq:beta_i-increasing-sequence}
\beta_1<\beta_2<\cdots<\beta_{n-3}.
\end{equation}

In the following, it will be useful to switch points of view between the components of the moment map $\mu$ and the action coordinates $\beta_1,\ldots,\beta_{n-3}$. We shall also require a criterion to identify the fixed points of a given Dehn twist in terms of triangle chains. The following fact is an immediate consequence of the definition of the moment map $\mu$.

\begin{fact}\label{fact:characterization-fixed-point-Dehn-twist}
For any $[\rho]\in\RepDT$ and any $i=1,\ldots, n-3$, the following equivalences hold:
\begin{align}
(d\beta_i)_{[\rho]}=0 \quad &\Leftrightarrow\quad \mu_0([\rho])=\cdots=\mu_{i-1}([\rho])=0\quad\text{or}\quad \mu_i([\rho])=\cdots=\mu_{n-3}([\rho])=0 \label{eq:differential-beta-i-vanishes-iff-consecutive-areas-are-zero} \\
&\Leftrightarrow \quad \tau_{b_i}\!\cdot\! [\rho]=[\rho] \nonumber \\
&\Leftrightarrow \quad \text{the $b_{i}$-orbit of $[\rho]$ is a point} \nonumber 
\end{align}
In other words, the critical points of $\beta_i$ coincide with the fixed points of the Dehn twist $\tau_{b_i}$ and they consist of all the points whose $\mathcal{B}$-triangle chain is of the following shape: all the triangles to either the left or the right of the shared vertex $B_i$ are degenerate. 
\end{fact}

\section{Infinitesimal transversality via Poisson brackets}\label{sec:transervality}
\subsection{Overview}
The goal of this section is to describe an explicit family of simple closed curves on $\surface$ whose angle functions, once differentiated, generate the cotangent space to $\RepDT$, at least at every point of $\IntRepDT{\mathcal{B}}(\surface)$. By adopting a geometric perspective, this happens when enough Hamiltonian flow lines meet transversely at every point (Section~\ref{sec:local-parameterization}). We begin with the case of a 4-punctured sphere (Section~\ref{sec:transervality-n=4}) before generalizing the arguments to an arbitrary number of punctures (Section~\ref{sec:general-case}) for two reasons. First, it brings some clarity on the technical arguments involved, and, second, the transversality result for $n=4$ (Corollary~\ref{cor:cotangent-space-generated-by-beta-delta-epsilon-n=4}) is slightly stronger than its generalization (Corollary~\ref{cor:cotangent-space-generated-by-beta_i-delta_i-epsilon_i}).

\subsection{Local transversality}\label{sec:local-parameterization}
We start by explaining how to parametrize a neighborhood of a point by using Hamiltonian flow lines. Let $M$ denote a $2m$-dimensional symplectic manifold and $f_1,\ldots,f_{2m}\colon M\to \R$ be Hamiltonian functions such that the cotangent space to $M$ at some point $x$ has basis $(df_1)_x,\ldots, (df_{2m})_x$. The Inverse Function Theorem implies the existence of an open neighborhood $U\subset M$ of $x$ which can be parametrized by the cube $(-\varepsilon,\varepsilon)^{2m}\subset \R^{2m}$ via the diffeomorphism 
\begin{align*}
\Psi\colon (-\varepsilon,\varepsilon)^{2m}&\to U\subset M\\
(t_1,\ldots,t_{2m})&\mapsto \Phi_{f_1}^{t_1}\circ\cdots\circ \Phi_{f_{2m}}^{t_{2m}}(x).
\end{align*}

\begin{lem}\label{lem:permuting-Hamiltonian-flows}
For every permutation $(f_1',\ldots, f_{2m}')$ of $(f_1,\ldots, f_{2m})$, there exists $\varepsilon'>0$ such that the map  
\begin{align*}
\Psi'\colon (-\varepsilon',\varepsilon')^{2m}&\to  M\\
(t_1,\ldots,t_{2m})&\mapsto \Phi_{f_1'}^{t_1}\circ\cdots\circ \Phi_{f_{2m}'}^{t_{2m}}(x)
\end{align*}
is a well-defined diffeomorphism onto an open subset $U'$ of $U$.
\end{lem}
\begin{proof}
The Inverse Function Theorem implies the existence of $\varepsilon'>0$ such that $\Psi'$ is a well-defined diffeomorphism onto an open subset $U'$ of $M$. What we have to prove is that we can shrink $\varepsilon'$ if necessary to guarantee that $U'\subset U$. This follows from a standard topological argument. The decreasing sequence of compact subsets $\Psi'([-1/n, 1/n]^{2m})\subset U'$ converges to the singleton $\{x\}$ and must therefore eventually be contained inside $U$.
\end{proof}

In particular, when we are dealing with two Hamiltonian functions $f$ and $g$ such that the Poisson bracket $\{f,g\}(x)\neq 0$, then $f$ is a diffeomorphism from a small open interval around $x$ along its $g$-orbit to an interval in $\R$. The same conclusion holds if we permute the roles of $f$ and $g$. This gives us the following fact.

\begin{fact}\label{fact:densely-many-irrational-points}
If $\{f,g\}(x)\neq 0$, then there are densely many points in an open interval around~$x$ contained in its $g$-orbit whose value under $f$ is an irrational multiple of $\pi$.
\end{fact}

Fact~\ref{fact:densely-many-irrational-points} will often be used in combination with Fact~\ref{fact:Dehn-twists-irrational-rotations} by taking $f$ and $g$ to be two angle functions $\vartheta_a$ and $\vartheta_b$, as defined in Section~\ref{sec:DT-ham-dynamics}. Briefly, if $\{\vartheta_a,\vartheta_b\}([\rho])\neq 0$ for some $[\rho]\in\RepDT$, then Fact~\ref{fact:densely-many-irrational-points} implies the existence of densely many $[\rho']$ in a small interval~$I$ around $[\rho]$ contained in its $b$-orbit, such that $\vartheta_a([\rho'])$ is an irrational multiple of~$\pi$. In particular, by Fact~\ref{fact:Dehn-twists-irrational-rotations}, the orbit of each $[\rho']$ under the Dehn twist $\tau_a$ is dense inside its $a$-orbit. So, if we knew, for instance, that the closure of the mapping class group orbit of~$[\rho]$ contains the $b$-orbit of~$[\rho]$, then it would actually contain a two-dimensional neighborhood of $[\rho]$ foliated by the $a$-orbits of every point in $I$. This ``dimension augmentation'' process will be detailed further in Section~\ref{sec:open-set-U}.

\subsection{The 4-punctured sphere}\label{sec:transervality-n=4} 
Let us consider first the case of a sphere with four punctures. We work with an arbitrary geometric presentation of $\pi_1\surface$ with four generators $c_1,c_2,c_3,c_4$ and its standard pants decomposition $\mathcal{B}$ given by the fundamental group element $b=(c_1c_2)^{-1}$. The family of curves we want to consider is the following:
\[
b=(c_1c_2)^{-1},\quad d=(c_{2}c_{3})^{-1},\quad e=(c_1c_3)^{-1}.
\]
\begin{center}
\begin{tikzpicture}[scale=.6]
  \draw[apricot] (0,2.4) arc(90:270:.75 and 2.4) node[near start, above left]{$b$};
  \draw[lightapricot] (0,2.4) arc(90:270:-.75 and 2.4);
  \draw[lightplum] (-2,0) arc(180:0:2 and .5);
  \draw[plum] (-2,0) arc(180:0:2 and -.5) node[near start, below]{$d$};
  
  \draw (-2,3) to[out=-30,in=210] (2,3);
  \draw (3,2) to[out=210,in=150] (3,-2);
  \draw (2,-3) to[out=150,in=30] (-2,-3);
  \draw (-3,-2) to[out=30,in=-30] (-3,2);

  \draw (-3,2) to[out=-30,in=-40] (-2,3);
  \draw (-3,2) to[out=140,in=140] (-2,3);
  \draw (2,3) to[out=30,in=30] (3,2);
  \draw (2,3) to[out=210,in=210] (3,2);
  \draw (3,-2) to[out=-30,in=-30] (2,-3);
  \draw (3,-2) to[out=140,in=140] (2,-3);
  \draw (-2,-3) to[out=30,in=30] (-3,-2);
  \draw (-2,-3) to[out=210,in=210] (-3,-2); 

  \draw (-3.1,3.1) node{$c_2$};
  \draw (3.1,3.1) node{$c_3$};
  \draw (3.1,-3.1) node{$c_4$};
  \draw (-3.1,-3.1) node{$c_1$};
\end{tikzpicture}
\hspace{1cm}
\begin{tikzpicture}[scale=.6]
  \draw[lightmauve] (1,2.55) to[out=10, in=80] (2.15,1);
  \draw[mauve] (1,2.55) to[out=190, in=70] (-2.15,-1);
  \draw[lightmauve] (-2.15,-1) to[out=-110, in=190] (-1,-2.55);
  \draw[mauve] (-1,-2.55) to[out=10, in=-105] (2.15,1);
  \draw[mauve] (-.5,1.5) node{$e$};
  
  \draw (-2,3) to[out=-30,in=210] (2,3);
  \draw (3,2) to[out=210,in=150] (3,-2);
  \draw (2,-3) to[out=150,in=30] (-2,-3);
  \draw (-3,-2) to[out=30,in=-30] (-3,2);

  \draw (-3,2) to[out=-30,in=-40] (-2,3);
  \draw (-3,2) to[out=140,in=140] (-2,3);
  \draw (2,3) to[out=30,in=30] (3,2);
  \draw (2,3) to[out=210,in=210] (3,2);
  \draw (3,-2) to[out=-30,in=-30] (2,-3);
  \draw (3,-2) to[out=140,in=140] (2,-3);
  \draw (-2,-3) to[out=30,in=30] (-3,-2);
  \draw (-2,-3) to[out=210,in=210] (-3,-2); 

  \draw (-3.1,3.1) node{$c_2$};
  \draw (3.1,3.1) node{$c_3$};
  \draw (3.1,-3.1) node{$c_4$};
  \draw (-3.1,-3.1) node{$c_1$};
\end{tikzpicture}
\end{center}
It is known since Fricke, Klein and Vogt that the coordinate ring of the relative $\SL_2\C$ character variety of $\surface$ with boundary traces $2\cos(\alpha_i/2)$ can be generated using the \emph{trace functions} $\Tr_b$, $\Tr_d$, and $\Tr_e$ associated to the curves $b$, $d$, and $e$ for any angle vector $\alpha$ (see for instance~\cite{cantat-loray}). Recall that the trace function associated to a fundamental group element $\gamma\in\pi_1\surface$ is the function $\Tr_\gamma$ defined on the character variety by $\Tr_\gamma([\rho])=\Tr(\rho(\gamma))$. Being able to identify generators of the coordinate ring has geometrical implications: the differentials $d\Tr_b$, $d\Tr_d$, and $d\Tr_e$ generate the cotangent space to the relative character variety at every point. 

We show how to circumvent the algebraic geometry and reprove this statement by computing Poisson brackets. We shall use the action-angle coordinates $(\beta,\gamma)$ from Section~\ref{sec:parametrization-by-triangle-chains}. Since DT representations take values in $\psl$ (and not $\SL_2\R$), it will be more convenient to work with angles functions (already introduced in~\eqref{eq:angle-function}) rather than trace functions. We abbreviate $\delta=\vartheta_{d}$ and $\varepsilon=\vartheta_{e}$. We start by characterizing the points at which the Poisson brackets $\{\beta,\delta\}$ and $\{\beta,\varepsilon\}$ vanish.
\begin{lem}\label{lem:poisson-bracket-vanish-iff-gamma-takes-some-values}
When $\surface$ is a 4-punctured sphere, a point $[\rho]\in\RepDT$ with coordinates $\overline\beta=\beta([\rho])$ and $\overline\gamma=\gamma([\rho])$ satisfies
\begin{align*}
    \{\beta,\delta\}([\rho])=0\quad &\Longleftrightarrow \quad (d\beta)_{[\rho]}=0\quad \text{or} \quad \overline\gamma\in \{0,\pi\}\\
    \{\beta,\varepsilon\}([\rho])=0\quad &\Longleftrightarrow \quad (d\beta)_{[\rho]}=0\quad \text{or} \quad \overline\gamma\in \{\overline\beta/2,\overline\beta/2-\pi\}.
\end{align*}
In particular, since $\overline{\beta}\in (0,2\pi)$, we have $\{\beta,\delta\}([\rho])=\{\beta,\varepsilon\}([\rho])=0$ if and only if $(d\beta)_{[\rho]}=0$.
\end{lem}
\begin{proof}
In the first place, we notice that if $(d\beta)_{[\rho]}=0$ then $\{\beta,\delta\}([\rho])=0$. Thus let us assume that $(d\beta)_{[\rho]}\neq 0$. This means that $[\rho]$ lies in the regular fibers of the moment map (Section~\ref{sec:moment-map}) or, equivalently, that the $\mathcal{B}$-triangle chain of $[\rho]$ consists of two non-degenerate triangles. Since $\{\beta,\delta\}=-d\delta(X_b)$ by definition of the Poisson bracket (Section~\ref{sec:goldamn-symplectic-form}), we have that $\{\beta,\delta\}([\rho])=0$ if and only if $[\rho]$ is a singular point of the function $\delta$ restricted to the $b$-orbit of $[\rho]$ and similarly for $\varepsilon$.

We would like to compare the three triangle chains of $[\rho]$ associated to the three pants decompositions $\mathcal{B}$, $\mathcal{D}$, and $\mathcal{E}$ determined, respectively, by the fundamental group elements $b$, $d$, and $e$. To do so, we first need to find a geometric presentation of $\pi_1\surface$ compatible with each pants decomposition. Because of the definition of $b$, the geometric generators $(c_1,c_2,c_3,c_4)$ are compatible with the pants decomposition given by $b=(c_1c_2)^{-1}$. We also note that $(c_2,c_3,c_4,c_1)$ are geometric generators compatible with $d=(c_2c_3)^{-1}$ and $(c_1,c_3,c_4,(c_3c_4)^{-1}c_2(c_3c_4))$ are geometric generators compatible with $e=(c_1c_3)^{-1}$. We are assuming that the $\mathcal{B}$-triangle chain of $[\rho]$ is regular, so it is made of two non-degenerate triangles $(C_1,C_2,B)$ and $(B,C_3,C_4)$ joined at the shared vertex $B$. They have respective interior angles $(\pi-\alpha_1/2,\pi-\alpha_2/2,\pi-\beta/2)$ and $(\beta/2,\pi-\alpha_3/2,\pi-\alpha_4/2)$. 

Let $D$ and $E$ denote the fixed points of $\rho(d)$ and $\rho(e)$ and let $\overline\delta=\delta([\rho])$ and $\overline\varepsilon=\varepsilon([\rho])$. The first triangle in the $\mathcal{D}$-triangle chain of $[\rho]$ has vertices $(C_2,C_3,D)$. Such a triangle is degenerate if and only if $C_2=C_3$, which can only happen when $\overline{\gamma}=0$. In that case, $(d\delta)_{[\rho]}=0$ by Fact~\ref{fact:characterization-fixed-point-Dehn-twist} and $\{\beta,\delta\}([\rho])=0$. Otherwise, the triangle $(C_2,C_3,D)$ is non-degenerate and has interior angles $(\pi-\alpha_2/2,\pi-\alpha_3/2,\pi-\overline{\delta}/2)$. The first triangle in the $\mathcal{E}$-triangle chain of $[\rho]$ has vertices $(C_1,C_3,E)$. It is degenerate exactly when $C_1=C_3$, which implies $\overline\gamma=\overline\beta/2-\pi$ and $\{\beta,\varepsilon\}([\rho])=0$ as above. Otherwise, it is non-degenerate and has interior angles $(\pi-\alpha_1/2,\pi-\alpha_3/2,\pi-\overline{\varepsilon}/2)$. 
\begin{center}
\begin{tikzpicture}[font=\sffamily, scale=.9]    
\node[anchor=south west,inner sep=0] at (0,0) {\includegraphics[width=5.4cm]{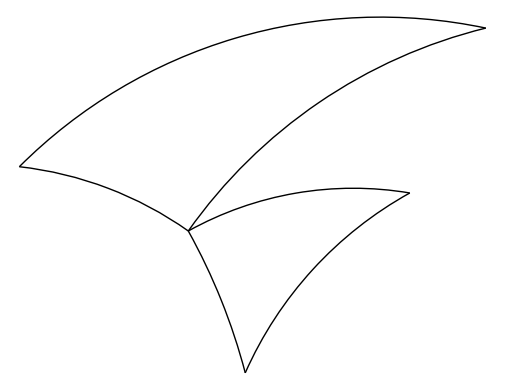}};

\begin{scope}
\fill (2.93,.25) circle (0.07) node[left]{$C_4$};
\fill (0.23,2.65) circle (0.07) node[left]{$C_1$};
\fill (4.88,2.35) circle (0.07) node[right]{$C_3$};
\fill (5.75,4.3) circle (0.07) node[right]{$C_2$};
\fill (2.25,1.9) circle (0.07) node[below left]{$B$};
\end{scope}

\begin{scope}
\draw[mauve] (4.88,2.35) to[bend left=15]  (5.75,4.3);
\fill[mauve] (3.7,3.1) circle (0.07) node[below]{$D$};
\draw[mauve] (4.88,2.35) to[bend right=15]  (3.7,3.1);
\draw[mauve] (5.75,4.3) to[bend right=15]  (3.7,3.1);
\end{scope}

\draw[mauve, thick] (5.5,4) arc (240:215:.6) node[near start, below right]{\small $\pi-\alpha_{2}/2$};
\draw[sky, thick] (2.9,2.2) arc (20:50:.6) node[near start, above right]{\small $\overline\gamma$};
\draw[mauve, thick] (5,2.9) arc (80:135:.6) node[at start, right]{\small $\pi-\alpha_{3}/2$};
\draw[mauve, thick] (4.1,2.95) arc (-20:38:.5) node[at end, above left]{\small $\pi-\overline\delta/2$};
\end{tikzpicture}
\begin{tikzpicture}[font=\sffamily, scale=.9]    
\node[anchor=south west,inner sep=0] at (0,0) {\includegraphics[width=5.4cm]{fig/fig-triangle-chain}};

\begin{scope}
\fill (2.93,.25) circle (0.07) node[left]{$C_4$};
\fill (0.23,2.65) circle (0.07) node[left]{$C_1$};
\fill (4.88,2.35) circle (0.07) node[right]{$C_3$};
\fill (5.75,4.3) circle (0.07) node[right]{$C_2$};
\fill (2.25,1.9) circle (0.07) node[below left]{$B$};
\end{scope}

\begin{scope}
\draw[plum] (4.88,2.35) to[bend right=20]  (0.23,2.65);
\fill[plum] (1.4,1) circle (0.07) node[below]{$E$};
\draw[plum] (4.88,2.35) to[bend right=15]  (1.4,1);
\draw[plum] (0.23,2.65) to[bend left=10]  (1.4,1);
\end{scope}

\draw[apricot, thick] (2.5,2.25) arc (40:155:.4) node[near end, above]{\small $\pi-\overline\beta/2$};
\draw[plum, thick] (0.72,2.8) arc (20:-40:.6) node[at end, left]{\small $\pi-\alpha_{1}/2$};
\draw[sky, thick] (2.9,2.2) arc (20:50:.6) node[near start, above right]{\small $\overline\gamma$};
\draw[plum, thick] (4.1,2.25) arc (200:160:.6) node[at end, above right]{\small $\pi-\alpha_{3}/2$};
\draw[plum, thick] (1.8,1.3) arc (30:115:.5) node[at end, below left]{\small $\pi-\overline\varepsilon/2$};
\end{tikzpicture}
\end{center}
The hyperbolic law of cosines applied to the triangle $(C_2,C_3,D)$ gives
\begin{equation}\label{eq:trig-expression-for-delta-n=4}
-\cos(\overline\delta/2)=-\cos(\alpha_2/2)\cos(\alpha_3/2)+\sin(\alpha_2/2)\sin(\alpha_3/2)\cosh(d(C_2,C_3)).
\end{equation}
If we apply it again to the triangle $(C_2,C_3,B)$, we obtain
\begin{equation}\label{eq:trig-expression-for-gamma-n=4}
\cos(\overline\gamma)=\frac{\cosh(d(C_2,B))\cosh(d(C_3,B))-\cosh(d(C_2,C_3))}{\sinh(d(C_2,B))\sinh(d(C_3,B))}.
\end{equation}
We can eliminate the term $\cosh(d(C_2,C_3))$ in both equations ~\eqref{eq:trig-expression-for-delta-n=4} and~\eqref{eq:trig-expression-for-gamma-n=4} to obtain
\begin{equation}
    \cos(\overline\delta/2) = \cos(\overline\gamma)\cdot k_1 + k_2, \label{eq:delta-in-terms-of-gamma}
\end{equation}
where
\begin{align*}
k_1 &= \sin(\alpha_2/2)\sin(\alpha_3/2)\sinh(d(C_2,B))\sinh(d(C_3,B)),\\
k_2 &= \cos(\alpha_2/2)\cos(\alpha_3/2)-\sin(\alpha_2/2)\sin(\alpha_3/2)\cosh(d(C_2,B))\cosh(d(C_3,B)).
\end{align*}

Since $\alpha_2, \alpha_3$ are fixed parameters and $d(C_2,B), d(C_3,B)$ are constant along $b$-orbits, we conclude that both $k_1$ and $k_2$ are constant along $b$-orbits. Note that $k_1$ vanishes if and only if $C_2=B$ or $C_3=B$, which is equivalent to $(d\beta)_{[\rho]}=0$ by Fact~\ref{fact:characterization-fixed-point-Dehn-twist}. Differentiating~\eqref{eq:delta-in-terms-of-gamma} along the $b$-orbit of $[\rho]$, we obtain
\begin{equation}\label{eq:poisson-bracket-beta-delta}
    \{\beta,\delta\}([\rho])=-(d\delta)_{[\rho]}(X_b)=\frac{-2k_1\sin(\overline{\gamma})}{\sin(\overline\delta/2)}\, (d\gamma)_{[\rho]}(X_b).
\end{equation}
Recall that $(d\gamma)_{[\rho]}(X_b)\neq 0$ because $\gamma$ is the angle coordinate paired with $\beta$. Moreover, since $\delta$ takes values in $(0,2\pi)$, $\sin(\overline{\delta}/2)$ never vanishes. As a consequence, we conclude from~\eqref{eq:poisson-bracket-beta-delta} that $\{\beta,\delta\}([\rho])=0$ if and only if $\sin(\overline\gamma)=0$; that is $\overline\gamma\in\{0,\pi\}$.

Analogously, the hyperbolic law of cosines applied to the triangle $(C_1,C_3,E)$ gives
\begin{equation}\label{eq:trig-expression-for-epsilon-n=4}
-\cos(\overline\varepsilon/2)=-\cos(\alpha_1/2)\cos(\alpha_3/2)+\sin(\alpha_1/2)\sin(\alpha_3/2)\cosh(d(C_1,C_3))
\end{equation}
and, when applied to the triangle $(C_1,C_3,B)$, it gives
\begin{equation}\label{eq:trig-expression-for-gamma-n=4-(case-epsilon)}
-\cos(\overline\beta/2-\overline\gamma)=\frac{\cosh(d(C_1,B))\cosh(d(C_3,B))-\cosh(d(C_1,C_3))}{\sinh(d(C_1,B))\sinh(d(C_3,B))}.
\end{equation}
As above, if we combine equations~\eqref{eq:trig-expression-for-epsilon-n=4} and~\eqref{eq:trig-expression-for-gamma-n=4-(case-epsilon)} we obtain
\begin{equation}\label{eq:epsilon-in-terms-of-gamma}
\cos(\overline\varepsilon/2)=\cos(\overline\beta/2-\overline\gamma)\cdot k_1' + k_2',
\end{equation}
where $k_1'$ and $k_2'$ are constant along $b$-orbits and $k_1'$ vanishes if and only if $(d\varepsilon)_{[\rho]}=0$. Because $\beta$ is also constant along $b$-orbits, after differentiating~\eqref{eq:epsilon-in-terms-of-gamma} along the $b$-orbit of $[\rho]$, we obtain
\begin{equation}\label{eq:poisson-bracket-beta-epsilon}
    \{\beta,\varepsilon\}([\rho])=-(d\varepsilon)_{[\rho]}(X_b)=\frac{2k_1'\sin(\overline{\beta}/2-\overline{\gamma})}{\sin(\overline\varepsilon/2)}\, (d\gamma)_{[\rho]}(X_b).
\end{equation}
For the same reasons as before, we conclude from~\eqref{eq:poisson-bracket-beta-epsilon} that $\{\beta,\varepsilon\}([\rho])=0$ if and only if $\sin(\overline\beta/2-\overline\gamma)=0$, which is equivalent to $\overline\gamma\in\{\overline\beta/2,\overline\beta/2-\pi\}$.
\end{proof}

\begin{rmk}\label{rem:poisson-bracket-vanish-iff-gamma-takes-some-values}
Lemma~\ref{lem:poisson-bracket-vanish-iff-gamma-takes-some-values} is expressed in terms of the action-angle coordinates $(\beta,\gamma)$. It is also possible to formulate an analogous statement in terms of the action-angle coordinates associated to the pants decomposition $\mathcal{D}$. The conclusion would then be that 
\[
\{\delta,\beta\}([\rho])=\{\delta,\varepsilon\}([\rho])=0 \quad\Longleftrightarrow\quad (d\delta)_{[\rho]}=0.
\]
\end{rmk}

\begin{rmk}\label{rem:correspondance-RP1-zero-locus-Poisson-bracket}
It is quite interesting to observe that the locus of points in $\RepDT$ where $\{\beta,\delta\}=0$ identifies with $\mathbb{RP}^1\subset \CP^1$ (the blue curve on the picture below) under the symplectomorphism from~\cite[Section~4]{action-angle} and the locus where $\{\beta,\varepsilon\}=0$ is its image by a ``half Dehn twist'' along the curve $b$.
\begin{center}
\begin{tikzpicture}
  \shade[ball color = mauve, opacity=.2] (0,0) circle (2cm);
  \draw[thick, sky] (0,-2) arc (270:90:1 and 2) node[midway, left]{\tiny $\gamma=0$};
  \draw[dashed, thick, sky] (0,2) arc (90:270:-1 and 2) node[midway, left]{\tiny $\gamma=\pi$};
  \fill[apricot] (0,2) circle (0.07) node[above]{\tiny $d\beta=0$};
  \fill[apricot] (0,-2) circle (0.07) node[below]{\tiny $d\beta=0$};

  \draw[apricot, thick] (3.2,-2.2) to (3.2,2.2);
  \draw[apricot] (2.5, 0) node{$\xrightarrow{\beta}$};
\end{tikzpicture}
\end{center}
\end{rmk}

\begin{cor}\label{cor:cotangent-space-generated-by-beta-delta-epsilon-n=4}
When $\surface$ is a 4-punctured sphere, the cotangent space to $\RepDT$ is generated by $d\beta$, $d\delta$, and $d\varepsilon$ at every point:
\[
T^\ast \RepDT=\big\langle d\beta, d\delta, d\varepsilon \big\rangle.
\]
\end{cor}
\begin{proof}
We first prove that the three Poisson brackets $\{\beta,\delta\}$, $\{\beta,\varepsilon\}$, and $\{\delta,\varepsilon\}$ can never vanish simultaneously. If there were a point $[\rho]\in\RepDT$ at which all three brackets vanished, then $(d\beta)_{[\rho]}=(d\delta)_{[\rho]}=0$ as a consequence of Lemma~\ref{lem:poisson-bracket-vanish-iff-gamma-takes-some-values} and Remark~\ref{rem:poisson-bracket-vanish-iff-gamma-takes-some-values}.

In terms of the $\mathcal{B}$-triangle chain of $[\rho]$, $(d\beta)_{[\rho]}$ vanishes if and only if $C_1=C_2$ or $C_3=C_4$ by Fact~\ref{fact:characterization-fixed-point-Dehn-twist}. Similarly, in terms of the $\mathcal{D}$-triangle chain of $[\rho]$, $(d\delta)_{[\rho]}$ vanishes if and only if $C_2=C_3$ or $C_4=C_1$ for the same reason. Thus, whenever both $(d\beta)_{[\rho]}$ and $(d\delta)_{[\rho]}$ vanish, three out of the four vertices $\{C_1, C_2, C_3, C_4\}$ must coincide. In particular, both triangles in the $\mathcal B$-triangle chain of $[\rho]$ would be degenerate and this is impossible. As a consequence, we conclude that $d\beta$ and $d\delta$ cannot both vanish at the same point in $\RepDT$. 

So, for an arbitrary point $[\rho]\in\RepDT$, one of the real numbers $\{\beta,\delta\}([\rho])$, $\{\beta,\varepsilon\}([\rho])$, or $\{\delta,\varepsilon\}([\rho])$ is non-zero. Let us assume first that $\{\beta,\delta\}([\rho])\neq 0$. Then, by definition of the Poisson bracket, we get that $\omega_{\mathcal G}(X_b([\rho]),X_d([\rho]))\neq 0$. This implies that the tangent space of $\RepDT$ at $[\rho]$ is generated by $X_b([\rho])$ and $X_d([\rho])$ because it is a 2-dimensional vector space. Equivalently, the cotangent space at $[\rho]$ is generated by $(d\beta)_{[\rho]}$ and $(d\delta)_{[\rho]}$. In the same fashion, if $\{\beta,\varepsilon\}([\rho])\neq 0$, then we would similarly obtain that the cotangent space at $[\rho]$ is generated by $(d\beta)_{[\rho]}$ and $(d\varepsilon)_{[\rho]}$ and if $\{\delta,\varepsilon\}([\rho])\neq 0$, then it is generated by $(d\delta)_{[\rho]}$ and $(d\varepsilon)_{[\rho]}$.
\end{proof}

\begin{rmk}
If we study the proof of Corollary~\ref{cor:cotangent-space-generated-by-beta-delta-epsilon-n=4} carefully, we shall learn which pair of differentials $\{d\beta,d\delta,d\varepsilon\}$ form a basis of the cotangent at every point. We would then obtain the following three possibilities:
\begin{align*}
    (d\beta)_{[\rho]}\neq 0 \text{ and } \gamma([\rho])\notin\{0,\pi\} &\quad\Longrightarrow\quad T^\ast_{[\rho]} \RepDT=\big\langle (d\beta)_{[\rho]}, (d\delta)_{[\rho]} \big\rangle,\\
    (d\beta)_{[\rho]}\neq 0 \text{ and } \gamma([\rho])\in\{0,\pi\} &\quad\Longrightarrow\quad T^\ast_{[\rho]} \RepDT=\big\langle (d\beta)_{[\rho]}, (d\varepsilon)_{[\rho]} \big\rangle,\\
    (d\beta)_{[\rho]}= 0 & \quad \Longrightarrow\quad T^\ast_{[\rho]} \RepDT=\big\langle (d\delta)_{[\rho]}, (d\varepsilon)_{[\rho]} \big\rangle.
\end{align*}
\end{rmk}

We can infer yet another consequence from the trigonometric computations in the proof of Lemma~\ref{lem:poisson-bracket-vanish-iff-gamma-takes-some-values}. It says that each pair of angle functions $(\beta,\delta)$ and $(\beta,\varepsilon)$ essentially parametrizes $\RepDT$ (up to finite redundancy).

\begin{cor}\label{cor:finite-intersection-b-orbits-and-d-orbits-n=4}
When $\surface$ is a 4-punctured sphere, a $b$-orbit intersects a $d$-orbit or an $e$-orbit in at most two points. In other words, if $\zeta\in\{\delta,\varepsilon\}$, then the map 
\[
(\beta,\zeta)\colon\RepDT\to (0,2\pi)^2
\]
is at most two-to-one.
\end{cor}
\begin{proof}
Let us fix a $b$-orbit, or in other words, let us fix a level set of the function $\beta$. If the $b$-orbit is a single point, the conclusion is immediate. Assume now that the $b$-orbit is a circle. In that case, it is parametrized by the angle coordinate $\gamma$. If we look again at the system of trigonometric equations~\eqref{eq:delta-in-terms-of-gamma}, we see that once the value of $\delta$ is prescribed, there are at most two values of $\gamma$ that solves the system (one being $2\pi$ minus the other). This shows that a $b$-orbit intersects a $d$-orbit in at most two points. If we instead combine~\eqref{eq:trig-expression-for-epsilon-n=4} and~\eqref{eq:trig-expression-for-gamma-n=4-(case-epsilon)}, we can write a similar argument to show that a $b$-orbit intersects an $e$-orbit in at most two points.
\end{proof}

\begin{rmk}
An alternative argument (of algebraic flavor) to prove Corollary~\ref{cor:finite-intersection-b-orbits-and-d-orbits-n=4} can be found in~\cite[Fact~4.3]{arnaud-sam}.
\end{rmk}

\subsection{General case}\label{sec:general-case}
We proceed with the general case where the number of punctures $n$ on $\surface$ is arbitrarily large. We fix a geometric presentation of $\pi_1\surface$ with generators $c_1,\ldots,c_n$ and consider its standard pants decomposition $\mathcal{B}$. When $\surface$ has five punctures or more, we shall only provide a family of simple closed curves whose angle functions generate the cotangent space at every point of $\IntRepDT{\mathcal{B}}(\surface)$. Recall from Section~\ref{sec:moment-map} that $\IntRepDT{\mathcal{B}}(\surface)$ is the open and dense subset of $\RepDT$ that consists of all regular fibers of the moment map associated to the pants decomposition $\mathcal{B}$. The simple closed curves that we want to consider are indexed by $i=1,\ldots,n-3$:
\[
b_i=(c_1\cdots c_{i+1})^{-1},\quad d_i=(c_{i+1}c_{i+2})^{-1},\quad e_i=c_{i+2}^{-1}b_{i-1}=(c_1\cdots c_{i}c_{i+2})^{-1}.
\]
\begin{center}
\vspace{2mm}
\begin{tikzpicture}[scale=1.1, decoration={
    markings,
    mark=at position 0.6 with {\arrow{>}}}]
  \draw[postaction={decorate}] (0,-.5) arc(-90:-270: .25 and .5) node[midway, left]{$c_1$};
   \draw[black!40] (0,.5) arc(90:-90: .25 and .5);
   \draw[apricot] (2,.5) arc(90:270: .25 and .5) node[near end, above left]{$b_1$};
   \draw[lightapricot] (2,.5) arc(90:-90: .25 and .5);
   \draw[apricot] (4,.5) arc(90:270: .25 and .5) node[near end, above left]{$b_2$};
   \draw[lightapricot] (4,.5) arc(90:-90: .25 and .5);
   \draw[apricot] (6,.5) arc(90:270: .25 and .5) node[near end, above left]{$b_3$};
   \draw[lightapricot] (6,.5) arc(90:-90: .25 and .5);
  \draw[apricot] (8,.5) arc(90:270: .25 and .5);
   \draw[apricot] (8,.5) arc(90:-90: .25 and .5);
   \draw (8.7, 0) node{$\ldots$};
  
   \draw (.5,1) arc(180:0: .5 and .25) node[midway, above]{$c_2$};
  \draw[postaction={decorate}] (.5,1) arc(-180:0: .5 and .25);
   \draw (2.5,1) arc(180:0: .5 and .25)node[midway, above]{$c_3$};
   \draw[postaction={decorate}] (2.5,1) arc(-180:0: .5 and .25);
   \draw (4.5,1) arc(180:0: .5 and .25)node[midway, above]{$c_4$};
   \draw[postaction={decorate}] (4.5,1) arc(-180:0: .5 and .25);
   \draw (6.5,1) arc(180:0: .5 and .25)node[midway, above]{$c_5$};
   \draw[postaction={decorate}] (6.5,1) arc(-180:0: .5 and .25);

   \draw[plum] (3.58,.73) arc(0:180:1.58 and -.5) node[near end, below]{$d_1$};
   \draw[lightplum] (3.58,.73) arc(-50:-130:2.46 and 1.25);
   \draw[plum] (5.58,.73) arc(0:180:1.58 and -.5) node[near start, below left]{$d_2$};
   \draw[lightplum] (5.58,.73) arc(-50:-130:2.46 and 1.25);
   \draw[plum] (7.58,.73) arc(0:180:1.58 and -.5) node[near start, below]{$d_3$};
   \draw[lightplum] (7.58,.73) arc(-50:-130:2.46 and 1.25);
   
   \draw (0,.5) to[out=0,in=-90] (.5,1);
   \draw (1.5,1) to[out=-90,in=180] (2,.5);
   \draw (0,-.5) to[out=0,in=180] (2,-.5);
  
   \draw (2,.5) to[out=0,in=-90] (2.5,1);
   \draw (3.5,1) to[out=-90,in=180] (4,.5);
   \draw (2,-.5) to[out=0,in=180] (4,-.5);
  
   \draw (4,.5) to[out=0,in=-90] (4.5,1);
   \draw (5.5,1) to[out=-90,in=180] (6,.5);
   \draw (4,-.5) to[out=0,in=180] (6,-.5);
  
   \draw (6,.5) to[out=0,in=-90] (6.5,1);
   \draw (7.5,1) to[out=-90,in=180] (8,.5);
   \draw (6,-.5) to[out=0,in=180] (8,-.5);
 \end{tikzpicture}
 \vspace{2mm}
 \end{center}
 \begin{center}
 \vspace{2mm}
 \begin{tikzpicture}[scale=1.1, decoration={
     markings,
     mark=at position 0.6 with {\arrow{>}}}]
   \draw[postaction={decorate}] (0,-.5) arc(-90:-270: .25 and .5) node[midway, left]{$c_1$};
   \draw[black!40] (0,.5) arc(90:-90: .25 and .5);
   \draw[black] (8,.5) arc(90:270: .25 and .5);
   \draw[black] (8,.5) arc(90:-90: .25 and .5);
   \draw (8.7, 0) node{$\ldots$};
  
   \draw[lightmauve] (3.57,.73) arc(170:10: -.58 and .25);
   \draw[lightmauve] (.2,.545) edge[out=-10,in=160,-] (.8,-.5);
   \draw[mauve] (.2,.545) arc(-150:-10: 1.2 and .6);
   \draw[mauve] (3.57,.73) edge[out=270,in=20,-] node[near end, above right]{$e_1$} (.8,-.5);
   \draw[lightmauve] (5.57,.73) arc(170:10: -.58 and .25);
   \draw[lightmauve] (2.2,.545) edge[out=-10,in=160,-] (2.8,-.5);
   \draw[mauve] (2.2,.545) arc(-150:-10: 1.2 and .6);
   \draw[mauve] (5.57,.73) edge[out=270,in=20,-] node[near end, above right]{$e_2$} (2.8,-.5);
   \draw[lightmauve] (7.57,.73) arc(170:10: -.58 and .25);
   \draw[lightmauve] (4.2,.545) edge[out=-10,in=160,-] (4.8,-.5);
   \draw[mauve] (4.2,.545) arc(-150:-10: 1.2 and .6);
   \draw[mauve] (7.57,.73) edge[out=270,in=20,-] node[midway, above]{$e_3$} (4.8,-.5);

   \draw (.5,1) arc(180:0: .5 and .25) node[midway, above]{$c_2$};
   \draw[postaction={decorate}] (.5,1) arc(-180:0: .5 and .25);
   \draw (2.5,1) arc(180:0: .5 and .25)node[midway, above]{$c_3$};
   \draw[postaction={decorate}] (2.5,1) arc(-180:0: .5 and .25);
   \draw (4.5,1) arc(180:0: .5 and .25)node[midway, above]{$c_4$};
   \draw[postaction={decorate}] (4.5,1) arc(-180:0: .5 and .25);
   \draw (6.5,1) arc(180:0: .5 and .25)node[midway, above]{$c_5$};
   \draw[postaction={decorate}] (6.5,1) arc(-180:0: .5 and .25);
   
   \draw (0,.5) to[out=0,in=-90] (.5,1);
   \draw (1.5,1) to[out=-90,in=180] (2,.5);
   \draw (0,-.5) to[out=0,in=180] (2,-.5);
  
   \draw (2,.5) to[out=0,in=-90] (2.5,1);
   \draw (3.5,1) to[out=-90,in=180] (4,.5);
   \draw (2,-.5) to[out=0,in=180] (4,-.5);
  
   \draw (4,.5) to[out=0,in=-90] (4.5,1);
   \draw (5.5,1) to[out=-90,in=180] (6,.5);
   \draw (4,-.5) to[out=0,in=180] (6,-.5);
  
   \draw (6,.5) to[out=0,in=-90] (6.5,1);
   \draw (7.5,1) to[out=-90,in=180] (8,.5);
   \draw (6,-.5) to[out=0,in=180] (8,-.5);
 \end{tikzpicture}
 \vspace{2mm}
 \end{center}
The fundamental group elements $b_1,\ldots, b_{n-3}$ are the ones defining the standard pants decomposition~$\mathcal{B}$. The associated action-angle coordinates introduced in Section~\ref{sec:parametrization-by-triangle-chains} are $(\beta_1,\ldots,\beta_{n-3},\gamma_1,\ldots,\gamma_{n-3})$. The curves $d_i$ and $e_i$ intersect $b_i$ twice and are disjoint from all other curves $b_j$ for $j\neq i$. It will be convenient to abbreviate the corresponding angle functions by $\delta_i=\vartheta_{d_i}$ and $\varepsilon_i=\vartheta_{e_i}$.

We start by characterizing the points of $\IntRepDT{\mathcal{B}}(\surface)$ at which the Poisson brackets $\{\beta_i,\delta_i\}$ and $\{\beta_i,\varepsilon_i\}$ vanish. We shall express such a condition in terms of the action-angle coordinates associated to~$\mathcal{B}$.
\begin{lem}\label{lem:poisson-bracket-vanish-iff-gamma-takes-some-values-general-case}
If $[\rho]\in\IntRepDT{\mathcal{B}}(\surface)$ has coordinates $(\overline\beta_1,\ldots,\overline\beta_{n-3},\overline\gamma_1,\ldots,\overline\gamma_{n-3})$, then for every $i=1,\ldots,n-3$ it holds that
\begin{align*}
    \{\beta_i,\delta_i\}([\rho])=0\quad &\Longleftrightarrow \quad \overline\gamma_i\in \{0,\pi\}\\
   \{\beta_i,\varepsilon_i\}([\rho])=0\quad &\Longleftrightarrow \quad \overline\gamma_i\in \{\overline\beta_i/2,\overline\beta_i/2-\pi\}.
\end{align*}
In particular, since $\overline{\beta}_i\in (0,2\pi)$, $\{\beta_i,\delta_i\}$ and $\{\beta_i,\varepsilon_i\}$ cannot simultaneously vanish on $\IntRepDT{\mathcal{B}}(\surface)$.
\end{lem}
\begin{proof}
We shall work with the $\mathcal{B}$-triangle chain of $[\rho]$ and, since we are assuming that $[\rho]$ lies in $\IntRepDT{\mathcal{B}}(\surface)$, all the triangles in the chain of $[\rho]$ are non-degenerate. Unlike the situation in Lemma~\ref{lem:poisson-bracket-vanish-iff-gamma-takes-some-values}, the curves $d_i$ and $e_i$ are not pants decompositions on their own. We would like to complete each of them into pants decompositions $\mathcal{D}_i$ and $\mathcal{E}_i$ of $\surface$. We can take the pants decomposition $\mathcal{D}_i$  to be the standard pants decomposition associated to the geometric generators 
\[
(c_{i+1},c_{i+2},\ldots,c_n,c_1,\ldots,c_{i}).
\]
The first triangle in the $\mathcal{D}_i$-triangle chain of $[\rho]$ has vertices $(C_{i+1},C_{i+2},D_i)$ and interior angles $(\pi-\alpha_{i+1}/2, \pi-\alpha_{i+2}/2, \pi-\overline\delta_i/2)$, where $D_i$ is the fixed point of $\rho(d_i)$ and $\overline\delta_i=\delta_i([\rho])$. It can be degenerate and this happens exactly when $C_{i+1}=C_{i+2}$, which is only possible if $\overline{\gamma}_i=0$. In that case, $(d\delta_i)_{[\rho]}=0$ by Fact~\ref{fact:characterization-fixed-point-Dehn-twist} and thus $\{\beta_i,\delta_i\}([\rho])=0$.

Similarly, the pants decomposition $\mathcal{E}_i$ will be taken to be the standard pants decomposition associated to the geometric generators 
\[
(c_1,\ldots,c_i,c_{i+2},c_{i+3},\ldots, c_{n}, (c_{i+2}\cdots c_{n})^{-1}c_{i+1}(c_{i+2}\cdots c_{n})).
\]
The $i^{th}$ triangle in the $\mathcal{E}_i$-triangle chain of $[\rho]$ has vertices $(B_{i-1}, C_{i+2}, E_i)$, where $E_i$ is the fixed point of $\rho(e_i)$ (we are adopting the convention that $B_0=C_1$). It is degenerate if and only if $B_{i-1}=C_{i+2}$, which implies $\overline{\gamma}_i=\overline{\beta}_i/2-\pi$. In that case, $\{\beta_i,\varepsilon_i\}([\rho])=0$ because $(d\varepsilon_i)_{[\rho]}=0$ by Fact~\ref{fact:characterization-fixed-point-Dehn-twist}. Otherwise, the triangle $(B_{i-1}, C_{i+2}, E_i)$ is non-degenerate and the interior angle at $B_{i-1}$ coincides with the interior angle at $B_{i-1}$ in the triangle $(B_{i-1}, C_{i+1},B_i)$ from the $\mathcal{B}$-triangle chain of $[\rho]$. This is because the $\mathcal{B}$-triangle chain and $\mathcal{E}_i$-triangle chain of $[\rho]$ have the same first $i-1$ triangles. It is thus equal to $\overline\beta_{i-1}/2$ (with the convention that $\overline\beta_0=2\pi-\alpha_1$). The other two angles are $\pi-\alpha_{i+2}/2$ and $\pi-\overline\varepsilon_i/2$, with $\overline\varepsilon_i=\varepsilon_i([\rho])$.
\begin{center}
\begin{tikzpicture}[font=\sffamily, scale=.9]    
\node[anchor=south west,inner sep=0] at (0,0) {\includegraphics[width=5.4cm]{fig/fig-triangle-chain}};

\begin{scope}
\fill (2.93,.25) circle (0.07) node[left]{$B_{i+1}$};
\fill (0.23,2.65) circle (0.07) node[left]{$B_{i-1}$};
\fill (4.88,2.35) circle (0.07) node[right]{$C_{i+2}$};
\fill (5.75,4.3) circle (0.07) node[right]{$C_{i+1}$};
\fill (2.25,1.9) circle (0.07) node[below left]{$B_i$};
\end{scope}

\begin{scope}
\draw[mauve] (4.88,2.35) to[bend left=15]  (5.75,4.3);
\fill[mauve] (3.7,3.1) circle (0.07) node[below]{$D_i$};
\draw[mauve] (4.88,2.35) to[bend right=15]  (3.7,3.1);
\draw[mauve] (5.75,4.3) to[bend right=15]  (3.7,3.1);
\end{scope}

\draw[mauve, thick] (5.5,4) arc (240:215:.6) node[near start, below right]{\small $\pi-\alpha_{i+1}/2$};
\draw[sky, thick] (2.9,2.2) arc (20:50:.6);
\draw[sky] (3.15,2.52) node{\small $\overline\gamma_i$};
\draw[mauve, thick] (5,2.9) arc (80:135:.6) node[at start, right]{\small $\pi-\alpha_{i+2}/2$};
\draw[mauve, thick] (4.1,2.95) arc (-20:38:.5) node[at end, above left]{\small $\pi-\overline\delta_i/2$};
\end{tikzpicture}
\begin{tikzpicture}[font=\sffamily, scale=.9]    
\node[anchor=south west,inner sep=0] at (0,0) {\includegraphics[width=5.4cm]{fig/fig-triangle-chain}};

\begin{scope}
\fill (2.93,.25) circle (0.07) node[left]{$B_{i+1}$};
\fill (0.23,2.65) circle (0.07) node[left]{$B_{i-1}$};
\fill (4.88,2.35) circle (0.07) node[right]{$C_{i+2}$};
\fill (5.75,4.3) circle (0.07) node[right]{$C_{i+1}$};
\fill (2.25,1.9) circle (0.07);
\draw (2.3,1.7) node[left]{$B_i$};
\end{scope}

\begin{scope}
\draw[plum] (4.88,2.35) to[bend right=20]  (0.23,2.65);
\fill[plum] (1.4,1) circle (0.07) node[below]{$E_i$};
\draw[plum] (4.88,2.35) to[bend right=15]  (1.4,1);
\draw[plum] (0.23,2.65) to[bend left=10]  (1.4,1);
\end{scope}

\draw[apricot, thick] (2.5,2.25) arc (40:155:.4) node[near end, above]{\small $\pi-\overline\beta_i/2$};
\draw[plum, thick] (0.72,2.8) arc (20:-40:.6) node[near end, below left]{\small $\beta_{i-1}/2$};
\draw[sky, thick] (2.9,2.2) arc (20:50:.6);
\draw[sky] (3.15,2.52) node{\small $\overline\gamma_i$};
\draw[plum, thick] (4.1,2.25) arc (200:160:.6) node[near end, above right]{\small $\pi-\alpha_{i+2}/2$};
\draw[plum, thick] (1.8,1.3) arc (30:115:.5) node[at end, below left]{\small $\pi-\overline\varepsilon_i/2$};
\end{tikzpicture}
\end{center}
Once we observed that $\beta_{i-1}$ is constant along $b_i$-orbits, we can apply the same trigonometric computations as in the proof of Lemma~\ref{lem:poisson-bracket-vanish-iff-gamma-takes-some-values} to the triangles shown in the pictures above to obtain relations of the kind
\begin{equation}\label{eq:delta_i-and-varepsilon_i-in-terms-of-gamma_i}
\cos(\overline\delta_i/2)=\cos(\overline{\gamma}_i)\cdot k_1+k_2, \quad \cos(\overline\varepsilon_i/2)=\cos(\overline{\beta}_i/2-\overline\gamma_i)\cdot k_1'+k_2',
\end{equation}
where $k_1,k_2,k_1',k_2'$ are constant along $b_i$-orbits. The conclusion follows as in Lemma~\ref{lem:poisson-bracket-vanish-iff-gamma-takes-some-values}. 
\end{proof}

\begin{rmk}
Lemma~\ref{lem:poisson-bracket-vanish-iff-gamma-takes-some-values-general-case} is a more precise version of~\cite[Lemma~3.3]{maret-ergodicity}. It not only says that the Poisson bracket $\{\beta_i,\delta_i\}$ vanishes at most two points along each $b_i$-orbit, it also gives the precise value of the angle coordinates $\gamma_i$ at which this happens.
\end{rmk}

\begin{rmk}
The observation made in Remark~\ref{rem:correspondance-RP1-zero-locus-Poisson-bracket} extends as follows. The set of points in $\IntRepDT{\mathcal B}$ where all the Poisson brackets $\{\beta_i,\delta_i\}$ vanish for every $i=1,\ldots,n-3$ is mapped inside the real Lagrangian submanifold $\mathbb{RP}^{n-3}\subset \CP^{n-3}$ by the symplectomorphism of~\cite[Section~4]{action-angle}.
\end{rmk}

\begin{cor}\label{cor:cotangent-space-generated-by-beta_i-delta_i-epsilon_i}
The cotangent space to $\RepDT$ at every point of $\IntRepDT{\mathcal{B}}(\surface)$ is generated by the differentials $d\beta_i$, $d\delta_i$, and $d\varepsilon_i$:
\[
T^\ast \IntRepDT{\mathcal{B}}(\surface)=\big\langle d\beta_i, d\delta_i, d\varepsilon_i \,\vert\, i=1,\ldots,n-3\big\rangle.
\]
\end{cor}
\begin{proof}
Lemma~\ref{lem:poisson-bracket-vanish-iff-gamma-takes-some-values-general-case} says that, for every $i=1,\ldots,n-3$, the Poisson brackets $\{\beta_i,\delta_i\}$ and $\{\beta_i,\varepsilon_i\}$ cannot both vanish on $\IntRepDT{\mathcal{B}}(\surface)$. Let us first assume that $\{\beta_i,\delta_i\}([\rho])\neq 0$. Recall that, since the curve $d_i$ is disjoint from the curves $b_j$ for $j\neq i$, it follows that $\{\beta_j,\delta_i\}([\rho])= 0$ for every $j\neq i$. The functions $\beta_1,\ldots,\beta_{n-3},\gamma_1,\ldots,\gamma_{n-3}$ form a system of action-angle coordinates for $\IntRepDT{\mathcal{B}}(\surface)$, implying that the cotangent space to $\IntRepDT{\mathcal{B}}(\surface)$ has everywhere a basis given by $d\beta_i, d\gamma_i$ for $i=1,\ldots,n-3$. When we express $(d\delta_i)_{[\rho]}$ is this basis, the coefficient in front of $(d\gamma_j)_{[\rho]}$ is equal to $\{\beta_j,\delta_i\}([\rho])$ (up to maybe a sign). We conclude that $(d\delta_i)_{[\rho]}$ has zero coordinate for $(d\gamma_j)_{[\rho]}$ when $j\neq i$ and non-zero coordinate for $(d\gamma_i)_{[\rho]}$. In other words, we can write $(d\gamma_i)_{[\rho]}$ as a linear combination of $(d\delta_i)_{[\rho]}$ and $(d\beta_1)_{[\rho]},\ldots, (d\beta_{n-3})_{[\rho]}$. The same conclusion for $(d\varepsilon_i)_{[\rho]}$ holds if we are in the case $\{\beta_i,\varepsilon_i\}([\rho])\neq 0$. This shows that the cotangent space to $\IntRepDT{\mathcal{B}}(\surface)$ at $[\rho]$ is generated by $(d\beta_i)_{[\rho]}$, $(d\delta_i)_{[\rho]}$, and $(d\varepsilon_i)_{[\rho]}$ for $i=1,\ldots,n-3$.
\end{proof}

\begin{rmk}\label{rem:cotangent-space-generated-by-beta_i-delta_i-epsilon_i}
Our proof of Corollary~\ref{cor:cotangent-space-generated-by-beta_i-delta_i-epsilon_i} actually shows the following stronger statement: a basis of the cotangent space to $\IntRepDT{\mathcal{B}}(\surface)$ at $[\rho]$ is given by $(d\beta_1)_{[\rho]},\ldots, (d\beta_{n-3})_{[\rho]}$ and carefully picking one of $(d\delta_i)_{[\rho]}$ and $(d\varepsilon_i)_{[\rho]}$ for every $i=1,\ldots,n-3$. Moreover, we can pick $(d\delta_i)_{[\rho]}$ over $(d\varepsilon_i)_{[\rho]}$ as long as $\{\beta_i,\delta_i\}([\rho])\neq 0$, which happens exactly when $\overline\gamma_i\notin \{0,\pi\}$ by Lemma~\ref{lem:poisson-bracket-vanish-iff-gamma-takes-some-values-general-case} (here $\overline{\gamma}_i$ denotes the $i^{th}$ angle coordinate of $[\rho]$). In other words,
\[
[\rho]\in\IntRepDT{\mathcal{B}}(\surface)\text{ and }\overline\gamma_i\notin\{0,\pi\}\,\forall i\quad \Longrightarrow \quad T^\ast_{[\rho]} \IntRepDT{\mathcal{B}}(\surface)=\big\langle d\beta_i, d\delta_i \,\vert\, i=1,\ldots,n-3\big\rangle.
\]
\end{rmk}

\begin{rmk}
The conclusion of Corollary~\ref{cor:cotangent-space-generated-by-beta_i-delta_i-epsilon_i} only holds at points of $\IntRepDT{\mathcal{B}}(\surface)$, and in general not on the whole DT component if $n\geq 5$. The reason is that there are points $[\rho]\in\RepDT$ where the three differentials $(d\beta_i)_{[\rho]}, (d\delta_i)_{[\rho]}, (d\varepsilon_i)_{[\rho]}$ vanish simultaneously for some $i$ (consider for instance points $[\rho]$ for which both moment map components $\mu_{i-1}$ and $\mu_i$ vanish). In order to generate the cotangent space at every point of $\RepDT$ using the differential of angle functions associated to a specific family of simple closed curves one would necessarily have to consider more curves than just $b_i$, $d_i$, and $e_i$.
\end{rmk}

We can formulate a generalization of Corollary~\ref{cor:finite-intersection-b-orbits-and-d-orbits-n=4} for an arbitrary number of punctures. It says that the angle functions $\beta_1,\ldots,\beta_{n-3}$ and $\delta_1,\ldots,\delta_{n-3}$ essentially parametrizes $\IntRepDT{\mathcal{B}}(\surface)$ (up to finite redundancy).

\begin{cor}\label{cor:finite-intersection-b-orbits-and-d-orbits}
For every choice of $\zeta_i\in\{\delta_i,\varepsilon_i\}$, the preimage of every point under the map
\[
(\beta_1,\ldots,\beta_{n-3},\zeta_1,\ldots,\zeta_{n-3})\colon \IntRepDT{\mathcal{B}}(\surface)\to (0,2\pi)^{2(n-3)}
\]
consists of at most $2^{n-3}$ points.
\end{cor}
\begin{proof}
It is enough to show that the prescription of a value for $\beta_i$, as well as a value for $\delta_i$ or $\varepsilon_i$ forces the angle coordinate $\gamma_i$ to take one among two possible values for every $i=1\ldots,n-3$. This, however, follows directly from~\eqref{eq:delta_i-and-varepsilon_i-in-terms-of-gamma_i}.
\end{proof}

\section{Density of infinite orbits: preliminaries}\label{sec:preliminaries}

\subsection{Overview}
This section is a preparation to the proof that infinite orbits are dense in DT components (Theorem~\ref{thm:infinite-orbits-are-dense}). We introduce one of the key tools: Selberg's Lemma (Section~\ref{sec:selberg}). We also include a proof that closures of infinite orbits are open for DT representations of 4-punctured spheres (Section~\ref{sec:infinite-orbits-are-dense-n=4}).

\subsection{Order of rational rotations}\label{sec:selberg}
The image of a DT representation may contain elliptic elements of $\psl$ whose angle of rotation are rational multiples of $\pi$. We call them \emph{rational rotations}. We would like to understand which rational rotations can simultaneously be in the image of a DT representation. It turns out that only finitely many rational rotation angles are possible. This is a consequence of Selberg's Lemma.
\begin{thm}[Selberg's Lemma]\label{thm:selberg-lemma}
A finitely generated linear group over a field of characteristic zero is virtually torsion free.    
\end{thm}
Applied to our context, Selberg's Lemma says that the image of a DT representation $\rho\colon\pi_1\surface\to\psl$ is virtually torsion free. This means that there is an index $N_\rho$ subgroup of $\rho(\pi_1\surface)$ which has no torsion. In particular, the order of any torsion element in $\rho(\pi_1\surface)$ must divide $N_\rho$. This means that the possible values for the rotation angles of rational rotations in $\rho(\pi_1\surface)$ are of the form $2\pi a/b$ where $0<a<b$ are coprime integers and $b$ divides $N_\rho$. This describes a finite family of angles.
\begin{cor}\label{cor:finitely-many-rational-angles}
The set of rotation angles of all the rational rotations in the image of a DT representation is finite.
\end{cor} 
We adopt the following notation. For a DT representation $\rho$, we shall denote by 
\[
\rational\subset \pi\Q
\]
the finite set of rotation angles of rational rotations in the image of $\rho$. It is clear that the set $\rational$ is invariant under conjugation of $\rho$. It is also constant along mapping class group orbits since all orbit points have the same image up to conjugation. We can therefore think of $\rational$ as an invariant of the mapping class group orbit of $[\rho]$.

\subsection{The 4-punctured sphere}\label{sec:infinite-orbits-are-dense-n=4}
We recall Cantat--Loray's argument from~\cite[p.2962]{cantat-loray} to show that infinite orbits in a DT components of a 4-punctured sphere are dense, reproving a result originally due to Previte--Xia~\cite{previte-xia-minimality}. Our presentation differs from the original in the sense that we use the transversality results from Section~\ref{sec:transervality-n=4} which we obtained via symplectic geometry, rather than using the algebraic properties of trace coordinates.

\begin{thm}[\cite{previte-xia-minimality, cantat-loray}]\label{thm:infinite-orbits-are-dense-n=4}
The closure of every infinite mapping class group orbit contained in a DT component of a 4-punctured sphere is open. Consequently, the closure of the orbit is the whole DT component.
\end{thm}

The proof of Theorem~\ref{thm:infinite-orbits-are-dense-n=4} unfolds over the next sections (Sections~\ref{sec:plan-of-proof-n=4}---\ref{sec:limit-points-n=4}).

\subsubsection{Structure of the proof}\label{sec:plan-of-proof-n=4}
To prove Theorem~\ref{thm:infinite-orbits-are-dense-n=4}, we start by fixing some point $[\rho]\in\RepDT$ with infinite mapping class group orbit. We denote by $\Orb([\rho])$ its mapping class group orbit and by $\overline{\Orb([\rho])}$ its closure. 
Once we shall have proven that $\overline{\Orb([\rho])}$ is open, it will follow that $\overline{\Orb([\rho])}=\RepDT$ because $\RepDT$ is connected. To show that $\overline{\Orb([\rho])}$ is open, we shall construct an open neighborhood $U_{[\rho']}\subset \overline{\Orb([\rho])}$ around every point $[\rho']$ inside $\overline{\Orb([\rho])}$. We start by constructing $U_{[\rho']}$ around every point $[\rho']\in\Orb([\rho])$ (Section~\ref{sec:open-set-U-n=4}). It will be enough to do so for one specific point in $\Orb([\rho])$ (Section~\ref{sec:choosing-a-point-n=4}) since we can then transport the open neighborhood to every other orbit point by letting $\mcg$ act. We shall then explain how to obtain a similar open set for points in $\overline{\Orb([\rho])}\setminus \Orb([\rho])$ (Section~\ref{sec:limit-points-n=4}).

\subsubsection{Choosing a point $[\rho']$ in $\Orb([\rho])$}\label{sec:choosing-a-point-n=4}
We shall work with the curves $b$, $d$, $e$ and the corresponding angle functions $\beta$, $\delta$, $\varepsilon$ that we introduced in Section~\ref{sec:transervality-n=4}. Let $\rational$ be the finite set of rational rotation angles in the image of any point in $\Orb([\rho])$ (introduced after Corollary~\ref{cor:finitely-many-rational-angles}). Corollary~\ref{cor:finite-intersection-b-orbits-and-d-orbits-n=4} implies that both $\beta^{-1}(\rational)\cap\delta^{-1}(\rational)$ and $\beta^{-1}(\rational)\cap\varepsilon^{-1}(\rational)$ are finite subsets of $\RepDT$. Moreover, since we are working with 4-punctured spheres, the function $\beta$ has exactly two critical points on $\RepDT$. Therefore, since we are assuming that $\Orb([\rho])$ is infinite, there exists a point $[\rho']\in\Orb([\rho])$ such that $(d\beta)_{[\rho']}\neq 0$ and such that $[\rho']$ lies outside of the sets $\beta^{-1}(\rational)\cap\delta^{-1}(\rational)$ and $\beta^{-1}(\rational)\cap\varepsilon^{-1}(\rational)$.

\subsubsection{Constructing the open set $U_{[\rho']}$}\label{sec:open-set-U-n=4}
We now apply Lemma~\ref{lem:poisson-bracket-vanish-iff-gamma-takes-some-values} to the point $[\rho']$. Since $(d\beta)_{[\rho']}\neq 0$, either $\{\beta,\delta\}([\rho'])\neq 0$ or $\{\beta,\varepsilon\}([\rho'])\neq 0$. Let us assume that $\{\beta,\delta\}([\rho'])\neq 0$; the other case can be treated similarly. In particular, both the $b$-orbit and the $d$-orbit of $[\rho']$ are circles of length $\pi$ that meet transversely at $[\rho']$. Because of the way $[\rho']$ was picked, either $\beta([\rho'])$ or $\delta([\rho'])$ is an irrational multiple of $\pi$. Let us assume that $\beta([\rho'])$ is an irrational multiple of $\pi$; again, the other case can be treated by analogous arguments. Fact~\ref{fact:Dehn-twists-irrational-rotations} implies that the $b$-orbit of $[\rho']$ is entirely contained in $\overline{\Orb([\rho'])}=\overline{\Orb([\rho])}$. 

As we explained in the proof of Corollary~\ref{cor:cotangent-space-generated-by-beta-delta-epsilon-n=4}, $\{\beta,\delta\}([\rho'])\neq 0$ implies that the cotangent space to $\RepDT$ at $[\rho']$ is generated by $(d\beta)_{[\rho']}$ and $(d\delta)_{[\rho']}$. By the discussion of Section~\ref{sec:local-parameterization}, there exists an open interval $I\subset \R$ around $0$ such that 
\[
U_{[\rho']}=\{\Phi_d^{t_1}\circ\Phi_b^{t_2}([\rho']):t_1,t_2\in I\}
\]
is an open neighborhood of $[\rho']$ in $\RepDT$. Fact~\ref{fact:densely-many-irrational-points} says that there exists a dense subset $I'\subset I$ such that $\delta(\Phi_b^{t_2}([\rho']))$ is an irrational multiple of $\pi$ for every $t_2\in I'$. Since $\Phi_b^{t_2}([\rho'])\in \overline{\Orb([\rho])}$ for every $t_2\in I'$ by the above, we conclude that the $d$-orbit of $\Phi_b^{t_2}([\rho'])$ is entirely contained in $\overline{\Orb([\rho])}$ for every $t_2\in I'$. By density of $I'$ in $I$, we conclude that $U_{[\rho']}\subset \overline{\Orb([\rho])}$.

\subsubsection{Limit points of $\overline{\Orb([\rho])}$}\label{sec:limit-points-n=4}
We still have to construct an analogous open set around every point in $\overline{\Orb([\rho])}\setminus \Orb([\rho])$. As before, it is enough to do so for one point as we can then translate this open set around every other point using $\mcg$. Let $[\rho_\infty]\in \overline{\Orb([\rho])}\setminus \Orb([\rho])$. By Lemma~\ref{lem:poisson-bracket-vanish-iff-gamma-takes-some-values}, either $\{\beta,\delta\}$ or $\{\beta,\varepsilon\}$ does not vanish at $[\rho_\infty]$. Let us assume that $\{\beta,\delta\}([\rho_\infty])\neq 0$; the other case can be treated similarly. The discussion of Section~\ref{sec:local-parameterization} shows the existence of some open interval $I_\infty\in \R$ around~$0$ such that $V_\infty=\{\Phi_d^{t_1}\circ\Phi_b^{t_2}([\rho_\infty]):t_1,t_2\in I_\infty\}\cap \{\Phi_b^{t_1}\circ\Phi_d^{t_2}([\rho_\infty]):t_1,t_2\in I_\infty\}$ is an open neighborhood of $[\rho_\infty]$ in $\RepDT$.

Since $[\rho_\infty]\in \overline{\Orb([\rho])}\setminus \Orb([\rho])$, there exists an infinite subset $\{[\rho_n]\}\subset \Orb([\rho])$ such that $[\rho_n]$ converges to $[\rho_\infty]$ as $n\to \infty$. We can find some large enough index $n$ such that $[\rho_n]$ lies outside of the finite set $\beta^{-1}(\rational)\cap\delta^{-1}(\rational)$ and such that $[\rho_n]\in V_\infty$. This means that either the $b$-orbit or the $d$-orbit of $[\rho_n]$ is entirely contained in $\overline{\Orb([\rho])}$. Assume it is the $b$-orbit; again, the other case can be treated by similar arguments. Since $[\rho_n]\in V_\infty$, we can find $t_1,t_2\in I_\infty$ such that $[\rho_\infty]=\Phi_d^{-t_2}\circ\Phi_b^{-t_1}([\rho_n])$. Note that $\Phi_b^{-t_1}([\rho_n])\in V_{\infty}$, because $\Phi_b^{-t_1}([\rho_n])=\Phi_d^{t_2}([\rho_\infty])$. So, the Poisson bracket $\{\beta,\delta\}$ does not vanish at $\Phi_b^{-t_1}([\rho_n])$ and Fact~\ref{fact:densely-many-irrational-points} implies that there are densely many $t$ around $t_1$ inside $I_\infty$ such that $\delta(\Phi_b^{-t}([\rho_n]))$ is an irrational multiple of $\pi$. Arguing as in Section~\ref{sec:choosing-a-point-n=4}, we obtain that the $d$-orbit of $\Phi_b^{-t_1}([\rho_n])$---which coincides with the $d$-orbit of $[\rho_\infty]$---is entirely contained in $\overline{\Orb([\rho])}$. Since we are assuming that $\{\beta,\delta\}([\rho_\infty])\neq 0$, we can repeat the argument to find densely many $t$ around $0$ in $I_\infty$ such that $\beta(\Phi_d^t([\rho_\infty]))$ is an irrational multiple of $\pi$. The $b$-orbits of all such $\Phi_d^t([\rho_\infty])$ are entirely contained in $\overline{\Orb([\rho])}$. So, up to shrinking the interval $I_\infty$ to some smaller interval $I_\infty'$ around $0$, we conclude that the open set 
\[
U_{[\rho_\infty]}=\{\Phi_b^{t_1}\circ\Phi_d^{t_2}([\rho_\infty]):t_1,t_2\in I_\infty'\}
\]
around $[\rho_\infty]$ is entirely contained in $\overline{\Orb([\rho])}$. This finishes the proof Theorem~\ref{thm:infinite-orbits-are-dense-n=4}.

\begin{rmk}\label{rem:alternative-end-of-proof}
Alternatively, it is possible to conclude the proof of Theorem~\ref{thm:infinite-orbits-are-dense-n=4} without the construction of the open set from Section~\ref{sec:limit-points-n=4}. We can instead use the ergodicity of the mapping class group action on $\RepDT$ (Theorem~\ref{thm:ergodicity}). Since $\overline{\Orb([\rho])}$ is a closed invariant set with non-empty interior by the construction of Sections~\ref{sec:choosing-a-point-n=4} and~\ref{sec:open-set-U-n=4}, and since the Goldman measure $\nu_{\mathcal{G}}$ is strictly positive, we conclude that $\nu_{\mathcal{G}}(\overline{\Orb([\rho])})=1$ by ergodicity. This immediately implies that $\overline{\Orb([\rho])}=\RepDT$ because $\nu_{\mathcal{G}}$ has full support.
\end{rmk}

\section{Density of infinite orbits: the proof}\label{sec:proof}

\subsection{Overview}
In this section, we generalize Theorem~\ref{thm:infinite-orbits-are-dense-n=4} and prove that infinite mapping class group orbits inside DT components are dense for any number of punctures.

\begin{thm}\label{thm:infinite-orbits-are-dense}
For every punctured sphere $\surface$ and every angle vector $\alpha$ satisfying the angle condition~\eqref{eq:angle-condition}, infinite mapping class group orbits in $\RepDT$ are dense.
\end{thm}

The proof of Theorem~\ref{thm:infinite-orbits-are-dense} is extensive and developed throughout the following sections (Sections~\ref{sec:plan-of-proof}---~\ref{sec:finishing-the-proof}). 

\subsection{Plan of the proof and notation}\label{sec:plan-of-proof}
We shall give a proof of Theorem~\ref{thm:infinite-orbits-are-dense} by induction on the number $n$ of punctures on $\surface$. The base case $n=4$ was treated in Theorem~\ref{thm:infinite-orbits-are-dense-n=4}. To prove the induction step, we fix a number of punctures $n\geq 5$ and assume that Theorem~\ref{thm:infinite-orbits-are-dense} holds for any number of punctures smaller than $n$. This is our induction hypothesis.
\begin{assumption}[Inductive hypothesis]\label{induction-hypotheses}
The conclusion of Theorem~\ref{thm:infinite-orbits-are-dense} holds for every sphere $\overline{\surface}$ with~$\overline{n}$ punctures where $\overline{n}<n$.
\end{assumption}
We also fix a DT component $\RepDT$ and a point $[\rho]\in\RepDT$ whose mapping class group orbit $\Orb([\rho])\subset\RepDT$ is infinite. We follow the same strategy as in the case $n=4$, which we described in Section~\ref{sec:plan-of-proof-n=4}. Briefly, our goal is to first construct an open subset contained in $\overline{\Orb([\rho])}$ around one point of $\Orb([\rho])$ (Sections~\ref{sec:choosing-a-point} and~\ref{sec:open-set-U}) in order to conclude that $\overline{\Orb([\rho])}$ has non-empty interior. We can then conclude that $\overline{\Orb([\rho])}=\RepDT$ by ergodicity as we explained in Remark~\ref{rem:alternative-end-of-proof} (Section~\ref{sec:finishing-the-proof}). There is substantial preliminary work that has to be conducted in order to prove that $\overline{\Orb([\rho])}$ has non-empty interior (Sections~\ref{sec:infinitely-many-orbit-points-regular-fibres} and~\ref{sec:cotangent-space}).

Along the way, we shall use the induction hypothesis (Assumption~\ref{induction-hypotheses}) twice: once in the proof of Lemma~\ref{lem:infinite-orbits-intersect-regular-fibers-at-infinitely-many-points} and once in the proof of Lemma~\ref{lem:getting-rid-of-beta=pi}.

We shall work with a fixed geometric presentation of $\pi_1\surface$ with generators $c_1,\ldots, c_n$ and consider $\mathcal{B}$ the associated standard pants decomposition of $\surface$. Recall that the pants curves of $\mathcal{B}$ are given by the fundamental group elements $b_i=(c_1\cdots c_{i+1})^{-1}$ for $i=1,\ldots,n-3$. The corresponding regular fibers of the moment map $\mu=(\mu_1,\ldots,\mu_{n-3})$ from Section~\ref{sec:moment-map} are, as usual, denoted by
\[
\IntRepDT{\mathcal{B}}(\surface)\subset\RepDT.
\]

\textit{Notation:} Many of our arguments involve trigonometry in the hyperbolic plane and it will be convenient to have a concise notation for hyperbolic lines, rays, and segments. Given two distinct points $A$ and $B$, we adopt the following convention:
\begin{itemize}
    \item $(AB)$ stands for the geodesic line through $A$ and $B$,
    \item $[AB)$ is the geodesic ray starting at $A$ and going through $B$,
    \item $[AB]$ is the geodesic segment between $A$ and $B$.
\end{itemize}

\subsection{Infinitely many orbit points in regular fibers}\label{sec:infinitely-many-orbit-points-regular-fibres}
We start by proving that an infinite orbit necessarily intersects the regular fibers $\IntRepDT{\mathcal{B}}(\surface)$ at infinitely many points. 

\begin{lem}\label{lem:infinite-orbits-intersect-regular-fibers-at-infinitely-many-points}
Infinite orbits intersect the regular fibers at infinitely many points. In particular,
\[
\Orb([\rho])\cap\IntRepDT{\mathcal{B}}(\surface)
\]
is infinite.
\end{lem}
\begin{proof}
We start by introducing some notation. For every subset $I\subset \{0,\ldots,n-3\}$ of cardinality $0\leq |I|\leq n-3$, we consider the subset of $\RepDT$ defined by
\[
A_I=\bigcap_{i\in I}\{\mu_i=0\}\cap\bigcap_{i\notin I}\{\mu_i\neq 0\}.
\]
The functions $\mu_i$ are the components of the moment map introduced in Section~\ref{sec:moment-map}. In other words, the points of $A_I$ are those whose associated $\mathcal{B}$-triangle chain has several degenerate triangles: precisely those indexed by the elements of $I$. We can write $\RepDT$ as the disjoint union of all the $A_I$ and we point out that $\IntRepDT{\mathcal{B}}(\surface)=A_\emptyset$. We shall also write
\[
\Orb_I=\Orb([\rho])\cap A_I.
\]
The sets $\Orb_I$ are of course not invariant sets for the mapping class group action; they are only preserved by the Dehn twists $\tau_{b_1},\ldots,\tau_{b_{n-3}}$. Since we are assuming that $\Orb([\rho])$ is infinite, there exists a subset $I'\subset \{0,\ldots,n-3\}$ of \emph{minimal cardinality} such that $\Orb_{I'}$ is infinite. The conclusion of the lemma is equivalent to $|I'|=0$. 

\begin{assumption}\label{ass:I'-minimal-cardinality}
We shall assume for the sake of contradiction that the set $I'$ of minimal cardinality with the property that $\Orb_{I'}$ is infinite has $|I'|\geq 1$. 
\end{assumption}

Assume that there exists an index $i'\in I'$ such that $i'\leq n-4$ and $i'+1\notin I'$. We shall prove that $\Orb_{I'\cup\{i'+1\}\setminus \{i'\}}$ is also infinite. Note that the $\mathcal{B}$-triangle chain of any point in $A_{I'}$ has $B_{i'}=C_{i'+2}=B_{i'+1}$ because $\mu_{i'}=0$ and it also satisfies $\mu_{i'+1}\neq 0$ because we are assuming $i'+1\notin I'$. Let $\tau$ denote the Dehn twist along the curve $c_{i'+2}c_{i'+3}$. 

\begin{claim}\label{claim:un-degenerating-triangles}
The Dehn twist $\tau$ maps every point of $A_{I'}$ into either $A_{I'\setminus\{i'\}}$ or $A_{I'\cup\{i'+1\}\setminus\{i'\}}$. 
\end{claim}

\begin{proofclaim}
To prove the claim, we shall show that the functions $\mu_j$ with $j\notin\{i',i'+1\}$ are invariant under $\tau$ and that the image of every point in $A_{I'}$ has $\mu_{i'}\neq 0$.

The argument is inspired from a similar argument used in the proof of \cite[Theorem~6.2]{arnaud-sam}. If $[\phi]\in A_{I'}$, then the $\mathcal B$-triangle chain of $\phi$ contains a non-degenerate triangle with vertices $(B_{i'+1}=C_{i'+2}, C_{i'+3}, B_{i'+2})$ because $\mu_{i'+1}\neq 0$. Its interior angles are $(\beta_{i'+1}/2, \pi-\alpha_{i'+3}/2, \pi-\beta_{i'+2}/2)$. Note that since $\mu_{i'}=0$, it holds that $\beta_{i'+1}/2=\pi-\alpha_{i'+2}/2+\beta_{i'}/2$. 
\begin{center}
\begin{tikzpicture}[font=\sffamily]
    
\node[anchor=south west,inner sep=0] at (0,0) {\includegraphics[width=6cm]{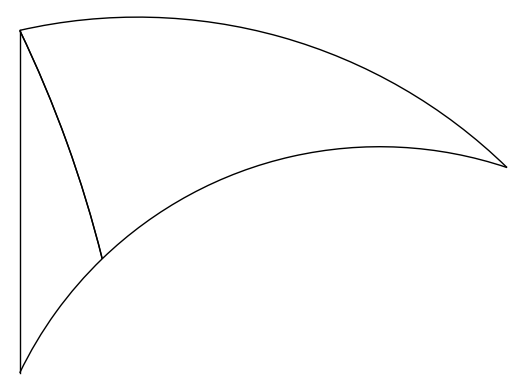}};

\begin{scope}
\fill (0.25,.2) circle (0.07) node[left]{$B_{i'+2}$};
\fill (0.22,4.15) circle (0.07) node[above left]{$B_{i'}=C_{i'+2}=B_{i'+1}$};
\fill (5.8,2.55) circle (0.07) node[right]{$C_{i'+3}$};
\fill (1.15,1.5) circle (0.07) node[right]{$D$};
\end{scope}

\draw[mauve, thick] (5.2,2.7) arc (170:118:.3) node[at start, above left]{\small $\pi-\alpha_{i'+3}/2$};
\draw[apricot, thick] (.5,.7) arc (60:90:.45) node[at end, left]{\small $\pi-\beta_{i'+2}/2$};
\draw[apricot, thick] (0.25,3.6) arc (270:362:.6) node[at start, above left]{\small $\beta_{i'+1}/2$};
\draw[apricot, thick] (0.25,3.4) arc (270:290:.8) node[at start, left]{\small $\beta_{i'}/2$};
\draw[mauve, thick] (0.57,3.4) arc (290:365:.8) node[midway, right]{\small $\pi-\alpha_{i'+2}/2$};
\end{tikzpicture}
\end{center}
The fixed point of $\phi(c_{i'+2}c_{i'+3})$ is the point $D$ on the geodesic segment $[C_{i'+3}B_{i'+2}]$ such that $\angle C_{i'+3}C_{i'+2}D=\pi-\alpha_{i'+2}/2$. The Dehn twist $\tau$ acts on the $\mathcal B$-triangle chain of $[\phi]$ by rotating the vertices $C_{i'+2}$ and $C_{i'+3}$ anti-clockwise around $D$, and it fixes all the other exterior vertices. It also fixes the shared vertices $B_1,\ldots, B_{i'}$ and $B_{i'+2},\ldots,B_{n-3}$ (all except $B_{i'+1}$). In particular, for $j\notin\{i',i'+1\}$, the function $\mu_j$ is invariant under $\tau$. The image by $\tau$ of $[\phi]$ always has $\mu_{i'}\neq 0$ because it rotates the exterior vertex $C_{i'+2}$ and leaves the shared vertex $B_{i'}$ fixed (those two agree before the rotation since we start from a point of $A_{I'}$ and $i'\in I'$). This proves the claim.
\end{proofclaim}

\begin{rmk}\label{rem:un-degenerating-triangles}
We can actually characterize the points of $A_{I'}$ that are mapped inside $A_{I'\cup\{i'+1\}\setminus\{i'\}}$ by $\tau$. Those are the points whose image under $\tau$ has $\mu_{i'+1}=0$. This happens when $C_{i'+3}$ is sent to $B_{i'+2}$ by the rotation induced by $\tau$. This eventuality is only possible when the triangle $(B_{i'+1}=C_{i'+2}, C_{i'+3}, B_{i'+2})$ is isosceles at $B_{i'+1}=C_{i'+2}$ and $D$ is the midpoint of the geodesic segment $[C_{i'+3}B_{i'+2}]$. These two conditions can equivalently be rewritten in terms of angles as $\pi-\alpha_{i'+3}/2=\pi-\beta_{i'+2}/2$ and $\pi-\alpha_{i'+2}/2=\beta_{i'}/2$, or even $\beta_{i'+2}=\alpha_{i'+3}$ and $\beta_{i'}=2\pi-\alpha_{i'+2}$. We conclude that $\tau$ should map most points of $A_{I'}$ into $A_{I'\setminus\{i'\}}$, namely all except those whose action coordinates $\beta_{i'}$ and $\beta_{i'+2}$ have a very specific value.
\end{rmk}

A consequence of Claim~\ref{claim:un-degenerating-triangles} is that $\tau$ maps every point of $\Orb_{I'}$ to either $\Orb_{I'\setminus \{i'\}}$ or to $\Orb_{I'\cup\{i'+1\}\setminus\{i'\}}$. Because we assumed $I'$ to be minimal with the property that $\Orb_{I'}$ is infinite (Assumption~\ref{ass:I'-minimal-cardinality}), the Dehn twist  $\tau$ can only map finitely many points of $\Orb_{I'}$ into $\Orb_{I'\setminus\{i'\}}$. In particular, $\Orb_{I'\cup\{i'+1\}\setminus\{i'\}}$ will contain infinitely many points. In other words, applying $\tau$ to the elements of $\Orb_{I'}$ allowed us to replace the element $i'$ of $I'$ by the consecutive integer $i'+1$ (which we assumed was not an element of $I'$ before). By repeating this procedure, we shall eventually obtain that $\Orb_{\overline{I}}$ is infinite, where
\[
\overline{I}=\{n-2-\vert I'\vert,\ldots,n-3\}.
\]
Note that $\vert\overline{I}\vert=\vert I'\vert$. The reason to replace $\Orb_{I'}$ by $\Orb_{\overline{I}}$ is that $\Orb_{\overline{I}}$ can be seen as a subset of a DT component corresponding to a smaller sphere with only $\overline{n}=n-\vert I'\vert<n$ punctures which will allow us to apply our induction hypothesis (Assumption~\ref{induction-hypotheses}).

More precisely, the curve $b_{n-2-\vert I'\vert}=b_{\overline{n}-2}$ separates $\surface$ into two sub-spheres. Let $\overline{\surface}\subset \surface$ denote the one with peripheral loops $c_1,\ldots, c_{\overline{n}-1}$ and $b_{\overline{n}-2}$.
\begin{center}
\begin{tikzpicture}[scale=1.1, decoration={
    markings,
    mark=at position 0.6 with {\arrow{>}}}]
  \draw (0,-.5) arc(-90:-270: .25 and .5);
   \draw[black!40] (0,.5) arc(90:-90: .25 and .5);
   \draw[apricot] (4,.5) arc(90:270: .25 and .5) node[near end, above left]{$b_{\overline{n}-2}$};
   \draw[lightapricot] (4,.5) arc(90:-90: .25 and .5);
  \draw (8,.5) arc(90:270: .25 and .5);
   \draw (8,.5) arc(90:-90: .25 and .5);
   \draw (8.7, 0) node{$\ldots$};
   \draw (-.8, 0) node{$\ldots$};
  
   \draw (.5,1) arc(180:0: .5 and .25) node[midway, above]{$c_{\overline{n}-2}$};
  \draw[postaction={decorate}] (.5,1) arc(-180:0: .5 and .25);
   \draw (2.5,1) arc(180:0: .5 and .25)node[midway, above]{$c_{\overline{n}-1}$};
   \draw[postaction={decorate}] (2.5,1) arc(-180:0: .5 and .25);
   \draw (4.5,1) arc(180:0: .5 and .25)node[midway, above]{$c_{\overline{n}}$};
   \draw[postaction={decorate}] (4.5,1) arc(-180:0: .5 and .25);
   \draw (6.5,1) arc(180:0: .5 and .25)node[midway, above]{$c_{\overline{n}+1}$};
   \draw[postaction={decorate}] (6.5,1) arc(-180:0: .5 and .25);
   
   \draw (0,.5) to[out=0,in=-90] (.5,1);
   \draw (1.5,1) to[out=-90,in=180] (2,.5);
   \draw (0,-.5) to[out=0,in=180] (2,-.5);
  
   \draw (2,.5) to[out=0,in=-90] (2.5,1);
   \draw (3.5,1) to[out=-90,in=180] (4,.5);
   \draw (2,-.5) to[out=0,in=180] (4,-.5);
  
   \draw (4,.5) to[out=0,in=-90] (4.5,1);
   \draw (5.5,1) to[out=-90,in=180] (6,.5);
   \draw (4,-.5) to[out=0,in=180] (6,-.5);
  
   \draw (6,.5) to[out=0,in=-90] (6.5,1);
   \draw (7.5,1) to[out=-90,in=180] (8,.5);
   \draw (6,-.5) to[out=0,in=180] (8,-.5);

   \draw[anchor=west] (3,-1.5) node{$\overline{\surface}$};
   \draw[->] (3,-1.5) to[out=180, in=-45] (2,.1);
 \end{tikzpicture}
 \end{center}
The natural inclusion $\pi_1\overline{\surface}\hookrightarrow\pi_1\surface$ can be used to restrict representations of $\pi_1\surface$ to representations of $\pi_1\overline{\surface}$. In other words, it defines a restriction map 
\[
\mathcal{R}\colon\Rep (\surface,\psl)\to \Rep(\overline{\surface},\psl)
\]
from the character variety of $\surface$ to the character variety of $\overline{\surface}$. The restriction map is surjective because it is always possible to lift a representation of $\pi_1\overline{\surface}$ to a representation of $\pi_1\surface$ by sending the new generators to the identity. The inclusion $\overline{\surface}\subset \surface$ also defines an injective group homomorphism $\PMod(\overline{\surface})\hookrightarrow \PMod(\surface)$ (see~\cite[Theorem~3.18]{mcg-primer}) whose image is the subgroup of $\PMod(\surface)$ generated by Dehn twists along simple closed curves of $\surface$ that are entirely contained in $\overline{\surface}$. This subgroup acts on $\Rep (\surface,\psl)$ and the restriction map $\mathcal{R}$ is equivariant with respect to this action and the standard action of $\PMod(\overline{\surface})$ on $\Rep(\overline{\surface},\psl)$.

There is a DT component associated to $\overline\surface$ with peripheral angle
\[
\overline{\alpha}=\big(\alpha_1,\ldots, \alpha_{\overline{n}-1}, \alpha_{\overline{n}}+\cdots+\alpha_n-2\pi(n-\overline{n})\big)
\]
which we denote by $\RepDTarg{\overline{\alpha}}(\overline{\surface})$. When we restrict representations from $\RepDT$ to $\pi_1\overline{\surface}$, we observe that those that satisfy $\mu_i=0$ for every $i\in\overline{I}$ get mapped injectively into $\RepDTarg{\overline{\alpha}}(\overline{\surface})$. In other words, $\mathcal R$ restricts to an isomorphism 
\begin{equation}\label{eq:restriction-map-one-subsphere}
\mathcal{R}\colon \RepDT\cap\bigcap_{i\in I}\{\mu_i=0\}\xrightarrow{\cong} \RepDTarg{\overline{\alpha}}(\overline{\surface}).
\end{equation}
The fundamental group elements $b_1,\ldots, b_{\overline{n}-3}\in\pi_1\overline{\surface}$ defines a pants decomposition $\overline{\mathcal{B}}$ of $\overline{\surface}$ that is compatible with the geometric generators $(c_1,\ldots,c_{\overline{n}-1},b_{\overline{n}-2})$ of $\pi_1\overline{\surface}$. The corresponding regular fibers $\IntRepDTarg{\overline{\mathcal{B}}}{\overline{\alpha}}(\overline{\surface}) \subset \RepDTarg{\overline{\alpha}}(\overline{\surface})$ coincide with restrictions of elements of $A_{\overline I}\subset \RepDT$:
\[
\mathcal{R}(A_{\overline I})=\IntRepDTarg{\overline{\mathcal{B}}}{\overline{\alpha}}(\overline{\surface}).
\]
We would like to apply the induction hypothesis to $\mathcal{R}(\Orb_{\overline{I}})\subset \RepDTarg{\overline{\alpha}}(\overline{\surface})$. In principle, even though $\Orb_{\overline{I}}$ is infinite, it is not clear that $\mathcal{R}(\Orb_{\overline{I}})$ contains a point with infinite $\PMod(\overline{\surface})$-orbit. This is what we are going to show.

\begin{rmk}
It was proven in~\cite{arnaud-sam} that a DT component contains at most finitely many finite orbits, which implies that any infinite subset of a DT component necessarily contains a point whose mapping class group orbit is infinite. In order to make the present paper independent from ~\cite{arnaud-sam} we shall not rely on the mentioned result to show Claim~\ref{claim:infinite-sub-orbit}.
\end{rmk}

\begin{claim}\label{claim:infinite-sub-orbit}
At least one point of $\mathcal{R}(\Orb_{\overline{I}})$ has an infinite $\PMod(\overline{\surface})$-orbit inside $\RepDTarg{\overline{\alpha}}(\overline{\surface})$.
\end{claim}
\begin{proofclaim}
Recall that $\mathcal{R}(\Orb_{\overline{I}})$ is actually a subset of $\IntRepDTarg{\overline{\mathcal{B}}}{\overline{\alpha}}(\overline{\surface})$. So, for every point $[\phi]$ of $\mathcal{R}(\Orb_{\overline{I}})$, we can use Corollary~\ref{cor:cotangent-space-generated-by-beta_i-delta_i-epsilon_i} and Remark~\ref{rem:cotangent-space-generated-by-beta_i-delta_i-epsilon_i} to pick a curve $z_i\in\{d_i,e_i\}$ for every $i=1,\ldots, \overline{n}-3$ such the the differentials of the angle functions associated to $b_1,\ldots,b_{\overline{n}-3}, z_1,\ldots, z_{\overline{n}-3}$ generate the cotangent space to $\RepDTarg{\overline{\alpha}}(\overline{\surface})$ at $[\phi]$. Since there are only finitely many ways to pick the curves $z_i$ and $\mathcal{R}(\Orb_{\overline{I}})$ is infinite, we have made the same choice of $z_i$ for infinitely many points of $\mathcal{R}(\Orb_{\overline{I}})$. In other words, there exists an infinite subset $\Orb_{\overline{I}}'\subset \Orb_{\overline{I}}$ for which the cotangent space to $\RepDTarg{\overline{\alpha}}(\overline{\surface})$ at every point of $\mathcal{R}(\Orb_{\overline{I}}')$ is generated by the differentials of the angle functions associated to the same curves $b_1,\ldots,b_{\overline{n}-3}, z_1,\ldots, z_{\overline{n}-3}$. In particular, the corresponding Dehn twists $\tau_{b_i}$ and $\tau_{z_i}$ have no fixed points in $\mathcal{R}(\Orb_{\overline{I}}')$. The respective angle functions associated to the curves $b_i$ and $z_i$ are denoted by $\beta_i$ and $\zeta_i$.

Corollary~\ref{cor:finitely-many-rational-angles} provides us with a finite set $\rational\subset \pi\Q$ of all rational rotation angles in the image of any point inside $\Orb([\rho])$. Since $\rational$ is a finite set, Corollary~\ref{cor:finite-intersection-b-orbits-and-d-orbits} applied to $\RepDTarg{\overline{\alpha}}(\overline{\surface})$ implies that
\begin{equation}\label{eq:rational-points}
\mathcal{T}=\IntRepDTarg{\overline{\mathcal{B}}}{\overline{\alpha}}(\overline{\surface})\cap \bigcap_{i=1}^{\overline{n}-3} \beta_i^{-1}(\rational)\cap \zeta_i^{-1}(\rational)
\end{equation}
is finite. So, we can find a point $[\phi]\in \mathcal{R}(\Orb_{\overline{I}}')\setminus \mathcal{T}$, because $\mathcal{R}(\Orb_{\overline{I}}')$ is infinite and $\mathcal{R}(\Orb_{\overline{I}}')\subset \IntRepDTarg{\overline{\mathcal{B}}}{\overline{\alpha}}(\overline{\surface})$. This point $[\phi]$ has the property that there exists an index $i\in\{1,\ldots,\overline{n}-3\}$ such that either $\beta_i([\phi])$ or $\zeta_i([\phi])$ is an irrational multiple of $[\phi]$. Since none of the two Dehn twists $\tau_{b_i}$ and $\tau_{z_i}$ fixes $[\phi]$ by construction, we conclude that the $\PMod(\overline{\surface})$-orbit of $[\phi]$ is infinite by Fact~\ref{fact:Dehn-twists-irrational-rotations}.
\end{proofclaim}

Let $[\phi]\in\mathcal{R}(\Orb_{\overline{I}})$ be any point with infinite $\PMod(\overline{\surface})$-orbit provided by Claim~\ref{claim:infinite-sub-orbit}. We shall denote by $\Orb_{\overline{\surface}}([\phi])\subset \RepDTarg{\overline{\alpha}}(\overline{\surface})$ its $\PMod(\overline{\surface})$-orbit.  
It is worth mentioning that $\Orb_{\overline{\surface}}([\phi])$ is not necessarily contained inside $\mathcal{R}(\Orb_{\overline{I}})$ because some orbit points in $\Orb_{\overline{\surface}}([\phi])$ may have extra degenerate triangles. However, $\Orb_{\overline{\surface}}([\phi])\cap \IntRepDTarg{\overline{\mathcal{B}}}{\overline{\alpha}}(\overline{\surface})\subset \mathcal{R}(\Orb_{\overline{I}})$.

The equivariance of the restriction map $\mathcal{R}$ (introduced in~\eqref{eq:restriction-map-one-subsphere}) implies
\[
\mathcal{R}^{-1}\left(\Orb_{\overline{\surface}}([\phi])\right)\subset \Orb \big(\mathcal{R}^{-1}([\phi])\big)=\Orb([\rho])
\]
which we use to write the following commutative diagram (in which the vertical arrows are inclusions of subsets).
\begin{equation*}
\begin{tikzcd}
\Orb_{\overline{\surface}}([\phi])\cap \IntRepDTarg{\overline{\mathcal{B}}}{\overline{\alpha}}(\overline{\surface}) \arrow[r, hook, "\mathcal{R}^{-1}"] \arrow[d, hook] & \Orb_{\overline{I}} \arrow[d, hook]\\
\IntRepDTarg{\overline{\mathcal{B}}}{\overline{\alpha}}(\overline{\surface}) \arrow{r}{\mathcal{R}^{-1}}[swap]{\cong} & A_{\overline{I}}
\end{tikzcd}
\end{equation*}
Since $\overline{n}<n$, the induction hypothesis (Assumption~\ref{induction-hypotheses}) applies to $\Orb_{\overline{\surface}}([\phi])$ and implies that it is a dense subset of $\RepDTarg{\overline{\alpha}}(\overline{\surface})$. This also means that $\Orb_{\overline{\surface}}([\phi])\cap \IntRepDTarg{\overline{\mathcal{B}}}{\overline{\alpha}}(\overline{\surface})$ is a dense subset of $\IntRepDTarg{\overline{\mathcal{B}}}{\overline{\alpha}}(\overline{\surface})$. Chasing the diagram above, we conclude that $\Orb_{\overline{I}}$ is a dense subset of $A_{\overline{I}}$. 

The density of $\Orb_{\overline{I}}$ inside $A_{\overline{I}}$ implies that $\Orb_{\overline{I}}$ contains infinitely many points with 
\[
\beta_{\overline{n}-3}\neq 2\pi-\alpha_{\overline{n}-1}.
\]
The image of all such points by the Dehn twist around the curve $c_{\overline{n}-1}c_{\overline{n}}$ belongs to $\Orb_{\overline{I}\setminus \{n-2-\vert I'\vert\}}$ by Claim~\ref{claim:un-degenerating-triangles} and Remark~\ref{rem:un-degenerating-triangles}. So, we conclude that $\Orb_{\overline{I}\setminus \{n-2-\vert I'\vert\}}$ is infinite which contradicts the minimality assumption on $I'$ (Assumption~\ref{ass:I'-minimal-cardinality}). This finishes the proof of Lemma~\ref{lem:infinite-orbits-intersect-regular-fibers-at-infinitely-many-points}.
\end{proof}

\subsection{Cotangent spaces at orbit points}\label{sec:cotangent-space}
We consider the curves $b_i,d_i,e_i$ and the associated angle functions $\beta_i,\delta_i,\varepsilon_i$ that we already used in Corollary~\ref{cor:cotangent-space-generated-by-beta_i-delta_i-epsilon_i}.  Corollary~\ref{cor:cotangent-space-generated-by-beta_i-delta_i-epsilon_i} says that at every point $[\phi]$ of $\IntRepDT{\mathcal{B}}(\surface)$, we can pick $\zeta_i\in\{\delta_i, \varepsilon_i\}$ for each $i=1,\ldots, n-3$ such that $\{(d\beta_i)_{[\phi]}, (d\zeta_i)_{[\phi]}:i=1,\ldots, n-3\}$ is a basis of the cotangent space to $\IntRepDT{\mathcal{B}}(\surface)$ at $[\phi]$. It will be useful for the sequel if we could find a point $[\rho']\in \Orb([\rho])\cap \IntRepDT{\mathcal{B}}(\surface)$ for which we can always pick $\zeta_i=\delta_i$. In other words, we want to find an orbit point $[\rho']$ such that 
\begin{equation}
T^\ast_{[\rho']} \IntRepDT{\mathcal{B}}(\surface)=\left\langle (d\beta_i)_{[\rho']}, (d\delta_i)_{[\rho']}:i=1,\ldots, n-3\right\rangle.
\end{equation}
Recall from Remark~\ref{rem:cotangent-space-generated-by-beta_i-delta_i-epsilon_i} that we can pick $\zeta_i$ to be $\delta_i$ over $\varepsilon_i$ as long as the angle coordinate $\gamma_i$ is not equal to $0$ or $\pi$. This means that we are looking for an orbit point $[\rho']\in \Orb([\rho])\cap\IntRepDT{\mathcal{B}}(\surface)$ satisfying $\gamma_i([\rho'])\notin\{0,\pi\}$ for every $i=1,\ldots, n-3$.

\subsubsection{A first infinite set}
Our next goal is to deal with orbit points having one of their action coordinate $\beta_i$ being equal to~$\pi$.

\begin{lem}\label{lem:getting-rid-of-beta=pi}
The set $\Orb([\rho])\cap\IntRepDT{\mathcal{B}}(\surface)$ contains infinitely many points with the property that if $\beta_i=\pi$, then $\gamma_i\notin\{0,\pi\}$. In other words, the set
\[
\mathcal{X}=\Orb([\rho])\cap\IntRepDT{\mathcal{B}}(\surface)\setminus \bigcup_{i=1}^{n-3}\{\beta_i=\pi\}\cap \left\{\gamma_i\in\{0,\pi\}\right\}
\]
is infinite.
\end{lem}
\begin{proof}
The proof is extensive and organized into multiple claims. For each $i=1,\ldots,n-3$, we introduce the sets
\[
A_i=\IntRepDT{\mathcal{B}}(\surface)\cap\{\beta_i=\pi\}\cap\left\{\gamma_i\in\{0,\pi\}\right\}
\]
and $A=\bigcup_i A_i$. With this notation, $\mathcal{X}= \Orb([\rho])\cap\IntRepDT{\mathcal{B}}\setminus A$. Recall from~\eqref{eq:beta_i-increasing-sequence}, that the action coordinates $\beta_1,\ldots,\beta_{n-3}$ always form an increasing sequence, meaning that for every point of $\RepDT$, there is at most one index $i=1,\ldots,n-3$ for which $\beta_i=\pi$. This means that $A$ is the disjoint union of its subsets $A_i$. We also introduce
\[
\Orb_i=\Orb([\rho])\cap A_i.
\]
Lemma~\ref{lem:infinite-orbits-intersect-regular-fibers-at-infinitely-many-points} says that $\Orb([\rho])\cap\IntRepDT{\mathcal{B}}$ is infinite. If $\Orb([\rho])\cap A$ was a finite set, then we would be done because $\mathcal{X}= \Orb([\rho])\cap\IntRepDT{\mathcal{B}}\setminus \big(\Orb([\rho])\cap A\big)$. From now on, we shall assume that $\Orb([\rho])\cap A$ is infinite, which implies that some $\Orb_j$ is itself infinite. The set $\Orb_j$ decomposes as the disjoint union of the two subsets $\Orb_j^0=\Orb_j\cap \{\gamma_j=0\}$ and $\Orb_j^\pi=\Orb_j\cap \{\gamma_j=\pi\}$.
\begin{claim}\label{claim:Orb_j^pi-infinite}
Since we are working under the assumption that $\Orb_j$ is infinite, $\Orb_j^\pi$ is infinite too.
\end{claim}
\begin{proofclaim}
As we explained before~\eqref{eq:action-of-tau_b_i}, applying the Dehn twist $\tau_{b_i}$ to a point of $\IntRepDT{\mathcal{B}}(\surface)$ changes its $i^{th}$ angle coordinate from $\gamma_i$ to $\gamma_i+\beta_i$. So, in the context of $\Orb_j$, the Dehn twist $\tau_{b_j}$ defines a bijection $\Orb_j^0\longleftrightarrow \Orb_j^\pi$.
\end{proofclaim}

The $\mathcal{B}$-triangle chain of any point in $A_j^\pi=A_j\cap \{\gamma_j=\pi\}$ has the following shape. The two triangles that share the vertex $B_j$ have a right angle at $B_j$ because $\beta_j=\pi$. Since $\gamma_i=\pi$, the point $B_j$ is the intersection of the geodesic lines $(B_{j-1}B_{j+1})$ and $(C_{j+1}C_{j+2})$ (they meet perpendicularly at $B_j$). It is also possible to have a configuration where $B_{j-1}=B_{j+1}$. The two points $C_{j+1}$ and $C_{j+2}$ are always distinct because $\gamma_j\neq 0$. 
\begin{center}
\begin{tikzpicture}[font=\sffamily]
    
\node[anchor=south west,inner sep=0] at (0,0) {\includegraphics[width=8cm]{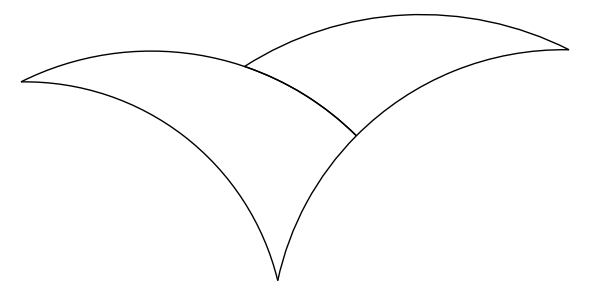}};

\begin{scope}
\fill (0.26,2.9) circle (0.07) node[left]{$B_{j+1}$};
\fill (3.3,3.15) circle (0.07) node[above]{$B_{j-1}$};
\fill (7.65,3.35) circle (0.07) node[right]{$C_{j+1}$};
\fill (4.85,2.18) circle (0.07) node[below right]{$B_j$};
\fill (3.77,.25) circle (0.07) node[right]{$C_{j+2}$};
\fill (2.3,2.3) circle (0.07) node[below]{$D_{j}$};
\end{scope}

\draw[mauve, thick] (7.1,3.3) arc (200:160:.4) node[midway, left]{\small $\pi-\alpha_{j+1}/2$};
\draw[mauve, thick] (3.9,.7) arc (70:110:.4) node[at end, left]{\small $\pi-\alpha_{j+2}/2$};
\draw[thick] (2.7,2.7) arc (50:-45:.5) node[midway, right]{\small $\pi-\delta_{j}/2$};
\draw[apricot, thick] (0.8,2.9) arc (-10:16:.5) node[midway, right]{\small $\pi-\beta_{j+1}/2$};
\draw[apricot, thick] (3.9,2.88) arc (-25:40:.5) node[midway, right]{\small $\beta_{j-1}/2$};

\draw[dashed] (3.3,3.15) to[bend right=5] (2.4,2.4);
\end{tikzpicture}
\end{center}

As in the proof of Lemma~\ref{lem:poisson-bracket-vanish-iff-gamma-takes-some-values-general-case}, we may consider the pants decomposition $\mathcal{D}_j$ of $\surface$ which is the standard pants decomposition associated to the cyclically permuted geometric generators 
\[
(c_{j+1},c_{j+2},\ldots,c_n,c_1,\ldots,c_j)
\]
of $\pi_1\surface$. The first triangle in the $\mathcal{D}_j$-triangle chain of any point in $A_j^\pi$ is always non-degenerate because $C_{j+1}\neq C_{j+2}$. It has vertices $(C_{j+1},C_{j+2},D_j)$ where $D_j$ is the unique point such that $\angle D_jC_{j+1}C_{j+2}=\pi-\alpha_{j+1}/2$, $\angle D_jC_{j+2}C_{j+1}=\pi-\alpha_{j+2}/2$, and such that the triangle $(C_{j+1},C_{j+2},D_j)$ is clockwise oriented. The point $D_j$ with these properties can be easily constructed: it is the intersection of the geodesic lines $(C_{j+1}B_{j-1})$ and $(C_{j+2}B_{j+1})$.

\begin{claim}\label{claim:tau_d_i-does-not-fix-any-point-in-A_jpi}
The Dehn twist $\tau_{d_j}$ along the curve $d_j=(c_{j+1}c_{j+2})^{-1}$ does not fix any point of~$A_j^\pi$.
\end{claim}
\begin{proofclaim}
By Fact~\ref{fact:characterization-fixed-point-Dehn-twist}, in order for $\tau_{d_j}$ to fix a point $[\phi]$ of $A_j^\pi$, its $\mathcal{D}_j$-triangle chain can only contain one non-degenerate triangle: the triangle $(C_{j+1},C_{j+2},D_j)$. This would mean that $D_j=C_1=\cdots=C_j=C_{j+3}=\cdots=C_n$, implying $D_j=B_{j-1}=B_{j+1}$. So, the $\mathcal{B}$-triangle chain of $[\phi]$ would only contain two non-degenerate triangles which is impossible since we are assuming $n\geq 5$ and $[\phi]\in\IntRepDT{\mathcal{B}}(\surface)$ by definition of $A_j^\pi$.
\end{proofclaim}

\begin{claim}\label{claim:apply-tau_d-to-get-rid-of-beta=pi}
The Dehn twist $\tau_{d_j}$ maps every point of $A_j^\pi$ with $\delta_j\neq \pi$ and $\delta_j-\beta_{j+1}+\beta_{j-1}\neq 0$ to a point inside $\IntRepDT{\mathcal{B}}(\surface)$ with $\beta_i\neq \pi$ for every $i=1,\ldots, n-3$. In particular, every point of $\Orb_j^\pi$ with $\delta_j\neq \pi$ and $\delta_j-\beta_{j+1}+\beta_{j-1}\neq 0$ is mapped inside $\mathcal{X}$ by $\tau_{d_j}$.
\end{claim}
\begin{proofclaim}
Let $[\phi]$ be a point of $A_j^\pi$ with $\delta_j\neq \pi$ and $\delta_j-\beta_{j+1}+\beta_{j-1}\neq 0$. 

First, we use the assumption $\delta_j\neq \pi$ to prove that $\tau_{d_j}.[\phi]$ lies inside $\IntRepDT{\mathcal{B}}(\surface)$. Recall that $\tau_{d_j}$ acts on the $\mathcal{B}$-triangle chain of $[\phi]$ by rotating the exterior vertices $C_{j+1}$ and $C_{j+2}$ clockwise around $D_j$ by an angle $\delta_j$ and leaves all the other exterior vertices fixed. In particular, it does not move the shared vertices $B_{j-1}$ and $B_{j+1}$. This means that the $\mathcal{B}$-triangle chain of $\tau_{d_j}.[\phi]$ is degenerate only when the rotation induced by $\tau_{d_j}$ maps $C_{j+1}$ to $B_{j-1}$ or if it maps $C_{j+2}$ to $B_{j+1}$. This is, however, impossible as long as $\delta_j\neq \pi$.

Next, we prove that the condition $\delta_j-\beta_{j+1}+\beta_{j-1}\neq 0$ implies $\beta_i(\tau_{d_j}\!\cdot\![\phi])\neq \pi$ for every $i=1,\ldots,n-3$. Since the curves $d_j$ and $b_i$ are disjoint when $i\neq j$, it holds that $\beta_i(\tau_{d_j}\!\cdot\![\phi])=\beta_i([\phi])\neq \pi$ when $i\neq j$. The tricky part is to prove that $\beta_j(\tau_{d_j}\!\cdot\![\phi])\neq \pi$. Assume for the sake of contradiction that $\beta_j(\tau_{d_j}\!\cdot\![\phi])=\beta_j([\phi])= \pi$. The assumption that $\delta_j-\beta_{j+1}+\beta_{j-1}\neq 0$ implies that $B_{j-1}\neq B_{j+1}$ for otherwise the two points would coincide with $D_j$ and we would have $\pi-\delta_j/2=\beta_{j-1}/2+\pi-\beta_{j+1}/2$. Since $B_{j-1}\neq B_{j+1}$, there are two possible configurations: either $B_{j-1}$ lies strictly between $B_{j+1}$ and $B_j$, or $B_{j+1}$ lies strictly between $B_{j-1}$ and $B_j$. We shall assume that we are in the first configuration; the other configuration can be treated with similar arguments. In that case, $D_j$ lies strictly between $B_{j+1}$ and $C_{j+2}$. Denote by $C_{j+2}^{\mathrm{new}}$ the exterior vertex in the $\mathcal{B}$-triangle chain of $\tau_{d_j}\!\cdot\![\phi]$ obtained from $C_{j+2}$ by a clockwise rotation of angle $\delta_j$ around $D_j$. Since we are assuming that $\beta_j(\tau_{d_j}\!\cdot\![\phi])=\beta_j([\phi])$ and since we know that $\beta_{j+1}$ remains unchanged by $\tau_{d_j}$, the triangle with vertices $(B_j,C_{j+2},B_{j+1})$ in the $\mathcal{B}$-triangle chain of $[\phi]$ is isometric to its counterpart in the $\mathcal{B}$-triangle chain of $\tau_{d_j}\!\cdot\![\phi]$. In particular, $d(B_{j+1},C_{j+2})=d(B_{j+1},C_{j+2}^{\mathrm{new}})$. The triangle inequality gives
\begin{align*}
d(B_{j+1},C_{j+2}^{\mathrm{new}})&\leq d(B_{j+1},D_j)+d(D_j,C_{j+2}^{\mathrm{new}})\\
&=d(B_{j+1},D_j)+d(D_j,C_{j+2})\\
&= d(B_{j+1},C_{j+2}).
\end{align*}
We conclude that $d(B_{j+1},C_{j+2}^{\mathrm{new}})= d(B_{j+1},D_j)+d(D_j,C_{j+2}^{\mathrm{new}})$ which is only possible when $C_{j+2}^{\mathrm{new}}$ lies on the hyperbolic ray $[B_{j+1}D_j)$ (recall that $B_{j+1}\neq D_j$). This is a contradiction as $\delta_j\in (0,2\pi)$, so we conclude that $\beta_j(\tau_{d_j}\!\cdot\![\phi])\neq \pi$ as desired.
\end{proofclaim}

Claim~\ref{claim:apply-tau_d-to-get-rid-of-beta=pi} helps us in the following sense. If $\Orb_j^\pi$ contains finitely many points with $\delta_j= \pi$ or $\delta_j-\beta_{j+1}+\beta_{j-1}= 0$, then it contains infinitely many point with $\delta_j\neq \pi$ and $\delta_j-\beta_{j+1}+\beta_{j-1}\neq 0$ because it is an infinite set by Claim~\ref{claim:Orb_j^pi-infinite}. In that case, Claim~\ref{claim:apply-tau_d-to-get-rid-of-beta=pi} applies and produces infinitely many points in $\mathcal{X}$, proving Lemma~\ref{lem:getting-rid-of-beta=pi}. So, we are reduced to consider the case where $\Orb_j^\pi$ contains infinitely many points with $\delta_j= \pi$ or $\delta_j-\beta_{j+1}+\beta_{j-1}= 0$. To show the existence of infinitely many points in $\mathcal{X}$ in that case, we shall use the induction hypothesis (Assumption~\ref{induction-hypotheses}). We shall proceed as follow.

We follow the same strategy as in the proof of Claim~\ref{claim:infinite-sub-orbit}. We first use Corollary~\ref{cor:cotangent-space-generated-by-beta_i-delta_i-epsilon_i} and Remark~\ref{rem:cotangent-space-generated-by-beta_i-delta_i-epsilon_i} to find an infinite subset 
\begin{equation}\label{eq:definition-Orb'}
(\Orb_j^\pi)'\subset \Orb_j^\pi\cap\big(\{\delta_j=\pi\}\cup\{\delta_j=\beta_{j+1}-\beta_{j-1}\}\big)   
\end{equation}
such that the cotangent space to $\RepDT$ at every point of $(\Orb_j^\pi)'$ is generated by the differentials of the angle functions associated to the \emph{same} curves $b_1,\ldots,b_{n-3},z_1,\ldots,z_{n-3}$ where $z_i\in \{d_i,e_i\}$. As usual, we shall write $\zeta_i$ for the angle function associated to the curve~$z_i$.

We let $\rational\subset\pi\Q$ denote the finite set of all rational rotation angles in the image of any point inside $\Orb([\rho])$ (Corollary~\ref{cor:finitely-many-rational-angles}). Corollary~\ref{cor:finite-intersection-b-orbits-and-d-orbits} implies that the set
\begin{equation}\label{eq:rational-points2}
\mathcal{T}=\IntRepDT{\overline{\mathcal B}}\cap\left(\beta_j^{-1}(\rational)\cap\delta_j^{-1}(\rational)\right)\cap\bigcap_{i\neq j} \beta_i^{-1}(\rational)\cap \zeta_i^{-1}(\rational)
\end{equation}
is finite. So, there exists a point $[\phi]\in (\Orb_j^\pi)' \setminus \mathcal{T}$ because we are assuming that $(\Orb_j^\pi)'$ is infinite. 

\begin{claim}\label{claim:index-k}
There exists an index $k\neq j$ such that either $\beta_k([\phi])$ or $\zeta_k([\phi])$ is an irrational multiple of $\pi$.
\end{claim}
\begin{proofclaim}
Since $[\phi]\in (\Orb_j^\pi)'\subset A_j^\pi$, it holds that $\beta_j([\phi])=\pi\in \pi\Q$. Recall that $[\phi]\notin \mathcal{T}$ by construction, so either there exists an index $k\neq j$ such that $\beta_k([\phi])$ or $\zeta_k([\phi])$ is an irrational multiple of $\pi$ (and the claim is proven), or $\delta_j([\phi])\notin\pi\Q$. When the latter occurs, then $\delta_j([\phi])$ is equal to $\beta_{j+1}([\phi])-\beta_j([\phi])$ because it cannot be equal to $\pi$. If $\delta_j([\phi])=\beta_{j+1}([\phi])-\beta_{j-1}([\phi])$ and $\delta_j([\phi])\notin \pi\Q$, then one of $\beta_{j+1}([\phi])$ and $\beta_{j-1}([\phi])$ must be an irrational multiple of $\pi$ which also proves the claim.
\end{proofclaim}

The curve $b_j$ separates $\surface$ into two sub-spheres $\surface_1$ and $\surface_2$, with respectively $j+2$ and $n-j$ punctures.
\begin{center}
\begin{tikzpicture}[scale=1.1, decoration={
    markings,
    mark=at position 0.6 with {\arrow{>}}}]
  \draw (0,-.5) arc(-90:-270: .25 and .5);
   \draw[black!40] (0,.5) arc(90:-90: .25 and .5);
   \draw[apricot] (4,.5) arc(90:270: .25 and .5) node[near end, above left]{$b_j$};
   \draw[lightapricot] (4,.5) arc(90:-90: .25 and .5);
  \draw (8,.5) arc(90:270: .25 and .5);
   \draw (8,.5) arc(90:-90: .25 and .5);
   \draw (8.7, 0) node{$\ldots$};
   \draw (-.8, 0) node{$\ldots$};
  
   \draw (.5,1) arc(180:0: .5 and .25) node[midway, above]{$c_{j}$};
  \draw[postaction={decorate}] (.5,1) arc(-180:0: .5 and .25);
   \draw (2.5,1) arc(180:0: .5 and .25)node[midway, above]{$c_{j+1}$};
   \draw[postaction={decorate}] (2.5,1) arc(-180:0: .5 and .25);
   \draw (4.5,1) arc(180:0: .5 and .25)node[midway, above]{$c_{j+2}$};
   \draw[postaction={decorate}] (4.5,1) arc(-180:0: .5 and .25);
   \draw (6.5,1) arc(180:0: .5 and .25)node[midway, above]{$c_{j+3}$};
   \draw[postaction={decorate}] (6.5,1) arc(-180:0: .5 and .25);
   
   \draw (0,.5) to[out=0,in=-90] (.5,1);
   \draw (1.5,1) to[out=-90,in=180] (2,.5);
   \draw (0,-.5) to[out=0,in=180] (2,-.5);
  
   \draw (2,.5) to[out=0,in=-90] (2.5,1);
   \draw (3.5,1) to[out=-90,in=180] (4,.5);
   \draw (2,-.5) to[out=0,in=180] (4,-.5);
  
   \draw (4,.5) to[out=0,in=-90] (4.5,1);
   \draw (5.5,1) to[out=-90,in=180] (6,.5);
   \draw (4,-.5) to[out=0,in=180] (6,-.5);
  
   \draw (6,.5) to[out=0,in=-90] (6.5,1);
   \draw (7.5,1) to[out=-90,in=180] (8,.5);
   \draw (6,-.5) to[out=0,in=180] (8,-.5);

   \draw[anchor=west] (3,-1.5) node{$\surface_1$};
   \draw[->] (3,-1.5) to[out=180, in=-45] (2,.1);
   \draw[anchor=east] (5,-1.5) node{$\surface_2$};
   \draw[->] (5,-1.5) to[out=0, in=-135] (6,.1);
 \end{tikzpicture}
 \end{center}
By choosing a system of geometric generators for both fundamental groups $\pi_1\surface_1$ and $\pi_1\surface_2$ and by mapping them to $(c_1,\ldots,c_{j+1},b_j)$ and $(b_j^{-1},c_{j+2},\ldots, c_n)$ respectively, we define two inclusions $\pi_1\surface_1\hookrightarrow\pi_1\surface$ and $\pi_1\surface_2\hookrightarrow\pi_1\surface$. 
The restrictions of $[\phi]$, say $[\phi\vert_{\surface_1}]$ and $[\phi\vert_{\surface_2}]$ both live in DT components, say $\RepDTarg{\theta_1}(\surface_1)$ and $\RepDTarg{\theta_2}(\surface_2)$, where $\theta_1=(\alpha_1,\ldots,\alpha_{j+1},\pi)$ and $\theta_2=(\pi,\alpha_{j+2},\ldots,\alpha_n)$ (recall that $\beta_j([\phi])=\pi$). More generally, there is a restriction map 
\begin{equation}\label{eq:restriction-map-two-subspheres}
\mathcal{R}\colon \RepDT\cap\{\beta_j=\pi\}\rightarrow \RepDTarg{\theta_1}(\surface_1)\times \RepDTarg{\theta_2}(\surface_2)
\end{equation}
which is surjective but it fails to be injective.\footnote{For instance, restricting the points in the regular fibers of the moment map $\IntRepDT{\mathcal{B}}\cap\{\beta_j=\pi\}\to \RepDTarg{\theta_1}(\surface_1)\times \RepDTarg{\theta_2}(\surface_2)$ defines a circle bundle where each fiber is parametrized by $\gamma_j$.}  
In this notation, $\mathcal{R}([\phi])=([\phi\vert_{\surface_1}],[\phi\vert_{\surface_2}])$. 

From the inclusions $\surface_1\subset \surface$ and $\surface_2\subset \surface$, we obtain two injective group homomorphisms $\PMod(\surface_1)\hookrightarrow\PMod(\surface)$ and $\PMod(\surface_2)\hookrightarrow\PMod(\surface)$ (see~\cite[Theorem~3.18]{mcg-primer}) which glue along $b_j$ to produce a combined injective group homomorphism
\begin{equation}\label{eq:inclusion-mcg-two-subspheres}
\PMod(\surface_1)\times \PMod(\surface_2)\hookrightarrow\PMod(\surface).
\end{equation}
Its image is the subgroup of $\PMod(\surface)$ generated by Dehn twists along simple closed curves that are entirely contained in $\surface_1$ or $\surface_2$. It preserves the locus $\RepDT\cap\{\beta_j=\pi\}$ inside $\RepDT$ and makes the restriction map $\mathcal{R}$~\eqref{eq:restriction-map-two-subspheres} equivariant with respect to the standard action of $\PMod(\surface_1)\times \PMod(\surface_2)$ on $\RepDTarg{\theta_1}(\surface_1)\times \RepDTarg{\theta_2}(\surface_2)$.

The orbits of the restrictions $[\phi\vert_{\surface_1}]$ and $[\phi\vert_{\surface_2}]$ under the action of $\PMod(\surface_1)$ and $\PMod(\surface_2)$ are written 
\[
\Orb([\phi\vert_{\surface_1}])\subset \RepDTarg{\theta_1}(\surface_1) \text{ and } \Orb([\phi\vert_{\surface_2}])\subset \RepDTarg{\theta_2}(\surface_2).
\]
In order to apply the induction hypothesis (Assumption~\ref{induction-hypotheses}), we need the following claim.

\begin{claim}\label{claim:infinitely-many-orbit-points-on-one-side}
If $k<j$, then $\Orb([\phi\vert_{\surface_1}])$ is infinite. If $k>j$, then $\Orb([\phi\vert_{\surface_2}])$ is infinite. Here $k$ is the integer from Claim~\ref{claim:index-k}.
\end{claim}
\begin{proofclaim}
Recall from Claim~\ref{claim:index-k} that $k$ was defined with the property that either $\beta_k([\phi])$ or $\zeta_k([\phi])$ is an irrational multiple of $\pi$. Let us first assume that $k<j$. In that case, the two curves $b_k$ and $z_k$ are contained in $\surface_1$. So, if $\beta_k([\phi])\in\R\setminus \pi\Q$, then iterating $\tau_{b_k}$ on $[\phi\vert_{\Sigma_1}]$ will produce infinitely many orbit points in $\Orb([\phi\vert_{\surface_1}])$ by Fact~\ref{fact:Dehn-twists-irrational-rotations}. In the same fashion, when $\zeta_k([\phi])\in\R\setminus \pi\Q$ we obtain infinitely many orbit points in $\Orb([\phi\vert_{\Sigma_1}])$ by iterating $\tau_{z_k}$ on $[\phi\vert_{\Sigma_1}]$. Note that both $\tau_{b_k}$ and $\tau_{z_k}$ do not fix $[\phi\vert_{\Sigma_1}]$ because we are assuming that the differentials $(d\beta_k)_{[\phi]}$ and $(d\zeta_k)_{[\phi]}$ are part of a basis of the cotangent space to $\RepDT$ at $[\phi]$ (this was the defining property of $(\Orb_j^\pi)'$~\eqref{eq:definition-Orb'}), so they do not vanish (Fact~\ref{fact:characterization-fixed-point-Dehn-twist}).

If instead $k>j$, then the curves $b_k$ and $z_k$ are contained in $\surface_2$ and the same argument applies \textit{mutatis mutandis}.
\end{proofclaim}

We are now ready to conclude the proof of Lemma~\ref{lem:getting-rid-of-beta=pi}. Thanks to Claim~\ref{claim:infinitely-many-orbit-points-on-one-side} and since $n>\max(n-j,j+2)$, we can apply the induction hypothesis (Assumption~\ref{induction-hypotheses}) to one of the restrictions $[\phi\vert_{\surface_1}]$ or $[\phi\vert_{\surface_2}]$ depending on the value of $k$. It shows that either $\Orb([\phi\vert_{\surface_1}])$ or $\Orb([\phi\vert_{\surface_2}])$ is dense in their respective DT component. We shall write 
\[
\Orb_{\surface_1}([\phi]) \text{ and } \Orb_{\surface_2}([\phi])
\]
for the orbits of $[\phi]$ under the action of $\PMod(\surface_1)$ and $\PMod(\surface_2)$ seen as subgroups of $\PMod(\surface)$. Note that both are sub-orbits of $\Orb([\phi])=\Orb([\rho])$ and satisfy 
\[
\mathcal{R}(\Orb_{\surface_1}([\phi]))=\Orb([\phi\vert_{{\surface_1}}])\times \{[\phi\vert_{\surface_2}]\}
\]
and similarly for $\surface_2$. It is worth pointing out that, even though every point in $\Orb_{\surface_1}([\phi])$ and $\Orb_{\surface_2}([\phi])$ has $\beta_j=\pi$, the value of $\gamma_j$ may vary; in fact, it is not preserved under the actions of $\PMod(\surface_1)$ and $\PMod(\surface_2)$.

We shall consider the case where $k<j$; the same arguments apply \textit{mutatis mutandis} when $k>j$. In that case, $\Orb([\phi\vert_{{\surface_1}}])$ is a dense subset of $\RepDTarg{\theta_1}(\surface_1)$ as we have just explained. In particular, it implies that $\Orb_{\surface_1}([\phi])\cap \IntRepDT{\mathcal{B}}(\surface)$ is infinite. Recall that by construction
\[
\Orb_{\surface_1}([\phi])\cap \IntRepDT{\mathcal{B}}(\surface)\subset \Orb([\rho])\cap \IntRepDT{\mathcal{B}}(\surface)\cap \{\beta_j=\pi\}.
\]
There are two possibilities: either $\Orb_{\surface_1}([\phi])\cap \IntRepDT{\mathcal{B}}(\surface)$ contains infinitely many points with $\gamma_j\notin\{0,\pi\}$. Since these points lie in $\mathcal{X}$ by definition then Lemma \ref{lem:getting-rid-of-beta=pi} follows. Otherwise, all but finitely many points of $\Orb_{\surface_1}([\phi])\cap \IntRepDT{\mathcal{B}}(\surface)$ satisfy $\gamma_j\in\{0,\pi\}$. Let us assume to be in the second case.

Now, observe the following. Every point in $A_j^\pi\cap \{\delta_j=\pi\}$ satisfies $\beta_j=\pi$, $\delta_j=\pi$, and $\gamma_j=\pi$. A simple trigonometric computation shows that, under these conditions, it is possible to relate the values of $\beta_{j-1}$ and $\beta_{j+1}$. In particular, for a fixed value of $\beta_{j-1}$, there are only finitely many values of $\beta_{j+1}$ that will guarantee $\beta_j=\pi$, $\delta_j=\pi$, and $\gamma_j=\pi$, and vice versa. Similarly, every point in $A_j^\pi\cap \{\delta_j=\beta_{j+1}-\beta_{j-1}\}$ has $B_{j-1}=B_{j+1}=D_j$. In particular, it holds that $d(B_j,B_{j-1})=d(B_j,B_{j+1})$ and both distances are directly related by trigonometry to the angles $\beta_{j-1}$ and $\beta_{j+1}$. So, again, for every value of $\beta_{j-1}$, there are only finitely many possible values of $\beta_{j+1}$ that will guarantee $B_{j-1}=B_{j+1}$, and vice versa. Let us say that $\beta_{j-1}$ and $\beta_{j+1}$ are \emph{compatible} if they are the coordinates of the same point in $A_j^\pi\cap(\{\delta_j=\pi\}\cup  \{\delta_j=\beta_{j+1}-\beta_{j-1}\})$. We have just explained that for every value of $\beta_{j-1}$, there are only finitely many compatible values of~$\beta_{j+1}$, and vice versa.

\begin{claim}\label{claim:incompatibility}
Recall that we are working under the assumption that all but finitely many points of $\Orb_{\surface_1}([\phi])\cap \IntRepDT{\mathcal{B}}(\surface)$ satisfy $\gamma_j\in\{0,\pi\}$. We claim that $\Orb_j^\pi$ contains infinitely many points for which $\beta_{j-1}$ is incompatible with $\beta_{j+1}$.
\end{claim}
\begin{proofclaim}
First, note that $\beta_{j+1}$ takes the same value at every point of $\Orb_{\surface_1}([\phi])\cap \IntRepDT{\mathcal{B}}(\surface)$. Since we are assuming that $\Orb([\phi\vert_{{\surface_1}}])$ is a dense subset of $\RepDTarg{\theta_1}(\surface_1)$, there are infinitely many points in $\Orb_{\surface_1}([\phi])\cap \IntRepDT{\mathcal{B}}(\surface)$ for which the value of $\beta_{j-1}$ is incompatible with $\beta_{j+1}$ (which is the same at every point of $\Orb_{\surface_1}([\phi])$). By using that all but finitely many points of $\Orb_{\surface_1}([\phi])\cap \IntRepDT{\mathcal{B}}(\surface)$ satisfy $\gamma_j\in\{0,\pi\}$, we conclude that $\Orb_{\surface_1}([\phi])\cap \IntRepDT{\mathcal{B}}(\surface)$ contains infinitely many points with $\gamma_j\in\{0,\pi\}$ and for which $\beta_{j-1}$ is incompatible with $\beta_{j+1}$. We can apply the Dehn twist $\tau_{b_j}$ to all those points with $\gamma_j=0$ to get $\gamma_j=\pi$ instead. When we apply $\tau_{b_j}$ to a point of $\Orb_{\surface_1}([\phi])\cap \IntRepDT{\mathcal{B}}(\surface)$, we might land outside that set. However, the image by $\tau_{b_j}$ of a point of $\Orb_{\surface_1}([\phi])\cap \IntRepDT{\mathcal{B}}(\surface)$ with $\gamma_j=0$ belongs to $\Orb_j^\pi$. Since $\tau_{b_j}$ does not affect the values of $\beta_{j-1}$ and $\beta_{j+1}$, incompatible values remain incompatible and we therefore produced infinitely many points in $\Orb_j^\pi$ for which $\beta_{j-1}$ is incompatible with $\beta_{j+1}$.
\end{proofclaim}

All the infinitely many points of $\Orb_j^\pi$ for which $\beta_{j-1}$ is incompatible with $\beta_{j+1}$ provided by Claim~\ref{claim:incompatibility} either have $\delta_j\neq \pi$ or $\delta_j\neq \beta_{j+1}-\beta_j$ by the definition of compatibility. Claim~\ref{claim:apply-tau_d-to-get-rid-of-beta=pi} applies and produce infinitely many points in $\mathcal{X}$ in that case too. This completes the proof of the Lemma~\ref{lem:getting-rid-of-beta=pi}.
\end{proof}

\subsubsection{A second infinite set}
We now refine Lemma~\ref{lem:getting-rid-of-beta=pi} as follows.

\begin{lem}\label{lem:getting-rid-of-gamma_i=0-or-pi}
The set
\[
\mathcal{X}'=\Orb([\rho])\cap\IntRepDT{\mathcal{B}}(\surface)\setminus \bigcup_{i=1}^{n-3}\left\{\gamma_i\in\{0,\pi\}\right\}
\]
is infinite.
\end{lem}
\begin{proof}
Lemma~\ref{lem:getting-rid-of-beta=pi}  tells us that the set
\[
\mathcal{X}=\Orb([\rho])\cap\IntRepDT{\mathcal{B}}(\surface)\setminus \bigcup_{i=1}^{n-3}\{\beta_i=\pi\}\cap \left\{\gamma_i\in\{0,\pi\}\right\}
\]
is infinite. Pick an element $[\phi]$ of $\mathcal{X}$. For each $i=1,\ldots,n-3$ with $\gamma_i([\phi])\in\{0,\pi\}$, apply $\tau_{b_i}$ to $[\phi]$. Note that by definition of $\mathcal{X}$, if $\gamma_i([\phi])\in\{0,\pi\}$, then $\beta_i([\phi])\neq \pi$. This ensures us that whenever $\gamma_i([\phi])\in\{0,\pi\}$, then $\gamma_i(\tau_{b_i}\!\cdot\![\phi])\notin\{0,\pi\}$ since $\gamma_i(\tau_{b_i}\!\cdot\![\phi])=\gamma_i([\phi])+\beta_i([\phi])$.

We just explained how to find, for every element $[\phi]$ of $\mathcal{X}$, a binary vector $\omega\in\{0,1\}^{n-3}$ such that $(\prod_i \tau_{b_i}^{\omega_i})\!\cdot\![\phi]\in\mathcal{X}'$ (recall that the Dehn twists $\tau_{b_i}$ commute because the curves $b_i$ are disjoint). This defines a map $\Omega\colon\mathcal{X}\to \mathcal{X'}$. Since the number of different binary vectors $\omega$ is $2^{n-3}$, any point in the image of $\Omega$ has at most $2^{n-3}$ pre-images. Since $\mathcal{X}$ is infinite, this shows that the image of $\Omega$ is infinite, and so $\mathcal{X}'$ is infinite too. 
\end{proof}

\subsection{Choosing a point $[\rho']$ in $\Orb([\rho])$}\label{sec:choosing-a-point}
We are now ready to pick our preferred orbit point $[\rho']\in \Orb([\rho])$. By construction (and that was the whole reason for proving Lemma~\ref{lem:getting-rid-of-gamma_i=0-or-pi}), the cotangent space to $\RepDT$ at every point of $\mathcal{X}'$ is generated by $d\beta_i$ and $d\delta_i$ with $i=1,\ldots,n-3$. This is a consequence of Corollary~\ref{cor:cotangent-space-generated-by-beta_i-delta_i-epsilon_i} and Remark~\ref{rem:cotangent-space-generated-by-beta_i-delta_i-epsilon_i}. We can play the usual trick and use Corollaries~\ref{cor:finitely-many-rational-angles} and~\ref{cor:finite-intersection-b-orbits-and-d-orbits} to identify a finite set
\[
\mathcal{T}=\IntRepDTarg{\mathcal{B}}{\alpha}(\surface)\cap \bigcap_{i=1}^{n-3} \beta_i^{-1}(\rational)\cap \delta_i^{-1}(\rational).
\]
We choose $[\rho']$ to be any point inside $\mathcal{X}'\setminus \mathcal{T}$. Such a point always exists because $\mathcal{X}'$ is infinite by Lemma~\ref{lem:getting-rid-of-gamma_i=0-or-pi} and $\mathcal{T}$ is a finite set. The chosen point $[\rho']$ (as any other element of $\mathcal{X}'\setminus \mathcal{T}$) enjoys the following properties:
\begin{enumerate}
    \item $[\rho']\in \IntRepDT{\mathcal{B}}(\surface)$ and $\gamma_i([\rho'])\notin\{0,\pi\}$ for every $i=1,\ldots, n-3$ by definition of $\mathcal{X}'$,\label{property:1}
    \item $\{\beta_i,\delta_i\}([\rho'])\neq 0$ for every $i=1,\ldots,n-3$ by Lemma~\ref{lem:poisson-bracket-vanish-iff-gamma-takes-some-values-general-case} because $\gamma_i([\rho'])\notin\{0,\pi\}$,\label{property:2}
    \item $T^\ast_{[\rho']}\RepDT=\langle (d\beta_i)_{[\rho']}, (d\delta_i)_{[\rho']}: i=1,\ldots, n-3\rangle$ because of Remark~\ref{rem:cotangent-space-generated-by-beta_i-delta_i-epsilon_i} and Property~\eqref{property:2},\label{property:3}
    \item There exists an index $j\in\{1,\ldots,n-3\}$ such that $\beta_j([\rho'])\in \R\setminus \pi\Q$ or $\delta_j([\rho'])\in \R\setminus \pi\Q$ by definition of $\mathcal{T}$.\label{property:4}
\end{enumerate}

\subsection{Constructing the open set $U_{[\rho']}$}\label{sec:open-set-U}
The goal of this section is to construct an open set around our preferred orbit point $[\rho']$ constructed in Section~\ref{sec:choosing-a-point} that is entirely contained in $\overline{\Orb([\rho])}$. To simplify the notation, for any two real numbers $(t,s)$ close to $0$ and $i\in\{1,\ldots,n-3\}$, we will write
\[
\Phi_i^{(t,s)}=\Phi_{d_i}^{t}\circ\Phi_{b_i}^{s}
\]
for the composition of the Hamiltonian flows associated with the curves $d_i$ and $b_i$. Recall from Property~\eqref{property:4} that there exists an index $j\in\{1,\ldots,n-3\}$ such that either $\beta_j([\rho'])$ or $\delta_j([\rho'])$ is an irrational multiple of $\pi$.
As we explained in Section~\ref{sec:local-parameterization}, since the cotangent space at $[\rho']$ is generated by the differentials $(d\beta_i)_{[\rho']}$ and $(d\delta_i)_{[\rho']}$ by Property~\eqref{property:3}, there exists an open interval $I\subset \R$ containing~$0$ such that the map $\Psi\colon I^{2(n-3)}\to \RepDT$ defined by sending $(t_1,\ldots, t_{n-3}, s_1, \ldots, s_{n-3})$ to
\[
\Phi_{j-1}^{(t_{j-1},s_{j-1})}\circ\cdots\circ\Phi_{1}^{(t_{1},s_1)}\circ \Phi_{n-3}^{(t_{n-3},s_{n-3})}\circ\cdots\circ \Phi_{j}^{(t_j,s_j)}([\rho'])
\]
is a diffeomorphism onto its image $\Im(\Psi)$. The desired open set $U_{[\rho']}$ will ultimately be the image by $\Psi$ of a smaller cube $(I')^{2(n-3)}$, where $I'\subset I$ is an open sub-interval that contains~$0$. In the definition of $\Psi$, we chose to compose the Hamiltonian flows in this specific order for convenience with respect to the upcoming arguments. It is important to start with the two Hamiltonian flows corresponding to the curves $b_j$ and $d_j$ for which either $\beta_j([\rho'])$ or $\delta_j([\rho'])$ is an irrational multiple of $\pi$.

Up to shrinking the interval $I$ a first time (by \emph{shrinking $I$} we mean replacing $I$ by an open sub-interval that still contains $0$), we may assume that every point in $\Im(\Psi)$ satisfies Properties~\eqref{property:1}---\eqref{property:3}. This is because they are open properties and, as we explained in the proof of Lemma~\ref{lem:permuting-Hamiltonian-flows}, we can always shrink $I$ so that $\Im(\Psi)$ is contained in any open neighborhood of $[\rho']$.

\subsubsection{A first slice}
Let $U_j\subset \Im(\Psi)$ be the set of images by $\Psi$ of points in $I^{2(n-3)}$ with $t_i=s_i=0$ when $i\neq j$. In other words,
\begin{equation}\label{eq:first-slice}
U_j=\left\{\Phi_{j}^{(t_j,s_j)}([\rho']):t_j,s_j\in I\right\}.
\end{equation}
Properties~\eqref{property:4} and~\eqref{property:2} tell us that either $\beta_j([\rho'])$ or $\delta_j([\rho'])$ is an irrational multiple of $\pi$ and $\{\beta_j,\delta_j\}([\rho'])\neq 0$. So, using similar arguments as in Section~\ref{sec:open-set-U-n=4} and shrinking $I$ further if necessary, we conclude that $U_j\subset \overline{\Orb([\rho])}$. 

\subsubsection{Extending the slice by increasing the index}\label{sec:increase}
Of course, $U_j$ is not an open subset of $\RepDT$; it is a 2-dimensional slice of $\Im(\Psi)$. In order to construct the desired open set $U_{[\rho']}$, we shall ``enlarge'' $U_j$ until we get an open neighborhood of $[\rho']$. We shall do that by working with the fundamental group elements $d_1,\ldots,d_{n-3}$ and the associated pants decompositions $\mathcal{D}_1,\ldots, \mathcal{D}_{n-3}$ of $\Sigma$ which we already considered in Lemma~\ref{lem:poisson-bracket-vanish-iff-gamma-takes-some-values-general-case}. Recall that the first triangle in the $\mathcal{D}_i$-triangle chain of any point in $\RepDT$ has vertices $(C_{i+1},C_{i+2},D_i)$.

\begin{claim}\label{claim:D_i-neq-C_i+3}
If $[\phi]$ is any point of $\IntRepDT{\mathcal{B}}(\surface)$ with $\gamma_i([\phi])\neq 0$ for some index $i$ (typically a point in $\Im(\Psi)$), then only finitely many points along the $b_i$-orbit of $[\phi]$ have $D_i=C_{i+3}$.
\end{claim} 
\begin{proofclaim}
Since we are assuming that $\gamma_i([\phi])\neq 0$, the first triangle in the $\mathcal{D}_i$-triangle chain of $[\phi]$ is non-degenerate. It has vertices $(C_{i+1},C_{i+2},D_i)$ and interior angles $(\pi-\alpha_{i+1}/2,\pi-\alpha_{i+2}/2, \pi-\delta_i/2)$. Recall that the points on the $b_i$-orbit of $[\phi]$ can be represented by triangle chains where the first $i$ triangles are rotated around the shared vertex $B_i$ and the other triangles remain fixed. In particular, the points $C_{i+2}$ and $C_{i+3}$ do not move along the $b_i$-orbit of $[\phi]$. It is possible that $C_{i+2}=C_{i+3}$, but in that case $D_i\neq C_{i+3}$ because $D_i\neq C_{i+2}$. If $C_{i+2}\neq C_{i+3}$, then the geodesic line $(C_{i+2}C_{i+3})$ is well-defined and invariant along the $b_i$-orbit of $[\phi]$. 

If $[\phi']$ is a point on the $b_i$-orbit of $[\phi]$, we shall denote the vertices the associated triangle chains with a prime. Recall that $B_i'=B_i$, $C_{i+2}'=C_{i+2}$, and $C_{i+3}'=C_{i+3}$. So, in order for $[\phi']$ to satisfy $D_i'=C_{i+3}'$, the point $C_{i+1}'$ must belong to both:
\begin{itemize}
    \item The circle of radius $d(B_i,C_{i+1})$ centered at $B_i$ because $d(B_i,C_{i+1}')=d(B_i,C_{i+1})$ since $[\phi']$ belongs to the $b_i$-orbit of $[\phi]$.
    \item The geodesic line through $C_{i+2}$ that makes an angle $\pi-\alpha_{i+2}/2$ with the geodesic line through $C_{i+2}$ and $C_{i+3}$ because we are assuming $D_i'=C_{i+3}'$, $C_{i+2}'=C_{i+2}$, and $C_{i+3}'=C_{i+3}$. 
\end{itemize}
So there are at most two possibilities for the point $C_{i+1}'$. This proves the claim.
\end{proofclaim}

If we study the $\mathcal{D}_i$-triangle chain of a point $[\phi]\in\IntRepDT{\mathcal{B}}(\surface)$ with $\gamma_i([\phi])\neq 0$ and $D_i\neq C_{i+3}$, then not only the first triangle is non-degenerate (because $\gamma_i([\phi])\neq 0$), but also the second triangle because $D_i\neq C_{i+3}$. So, $(d\delta_i)_{[\phi]}\neq 0$ by Fact~\ref{fact:characterization-fixed-point-Dehn-twist} and there is a well-defined angle coordinate $\gamma_1^{\mathcal{D}_i}$ paired with $\delta_i$ which is defined as the angle between the geodesic rays $[D_iC_{i+3})$ and $[D_i,C_{i+2})$.
\begin{center}
\begin{tikzpicture}[font=\sffamily]
    
\node[anchor=south west,inner sep=0] at (0,0) {\includegraphics[width=6cm]{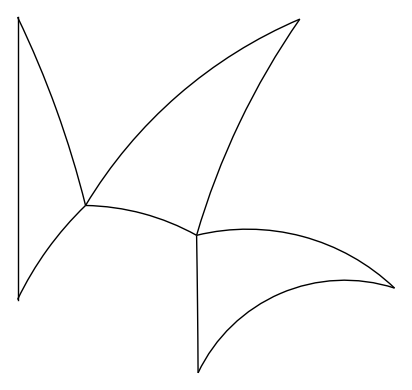}};

\begin{scope}
\fill (0.25,1.35) circle (0.07) node[left]{$B_{i-1}$};
\fill (1.25,2.7) circle (0.07) node[below]{$B_{i}$};
\fill (2.85,2.25) circle (0.07) node[left]{$B_{i+1}$};
\fill (2.87,0.3) circle (0.07) node[left]{$B_{i+2}$};
\fill (0.25,5.45) circle (0.07) node[left]{$C_{i+1}$};
\fill (4.35,5.4) circle (0.07) node[right]{$C_{i+2}$};
\fill (5.7,1.5) circle (0.07) node[right]{$C_{i+3}$};
\fill[mauve] (2.3,3.1) circle (0.07) node[left]{$D_{i}$};
\end{scope}

\draw[mauve] (0.25,5.45) to[bend left=15] (2.3,3.1);
\draw[mauve] (0.25,5.45) to[bend left=10] (4.35,5.4);
\draw[mauve] (2.3,3.1) to[bend left=10] (4.35,5.4);
\draw[mauve] (2.3,3.1) to[bend left=25] (5.7,1.5);

\draw[sky, thick] (1.6,3.2) arc (50:110:.5) node[midway, above]{\small $\gamma_i$};
\draw[mauve, thick] (0.8,5.05) arc (-40:10:.6) node[at end, above]{\small $\pi-\alpha_{i+1}/2$};
\draw[mauve, thick] (3.88,5) arc (230:172:.6) node[at end, above]{\small $\pi-\alpha_{i+2}/2$};
\draw[mauve, thick] (2.9,4) arc (60:125:1);
\draw[mauve] (2.18,4.6) node{\small $\pi-\delta_{i}/2$};
\draw[plum, thick] (2.9,3.05) arc (-5:58:.6) node[at end, right]{\small $\gamma_1^{\mathcal{D}_i}$};
\end{tikzpicture}
\end{center}

\begin{claim}\label{claim:Poisson-bracket-d_i-d_i+1}
If $[\phi]$ is a point with $C_{i+1}\neq C_{i+2}$ and $D_i\neq C_{i+3}$ for some index $i$, then 
\[
\{\delta_i,\delta_{i+1}\}([\phi])=0\quad\Leftrightarrow\quad \gamma_1^{\mathcal{D}_i}\in\{0,\pi\}.
\]
\end{claim}
\begin{proofclaim}
Similar arguments as in the proof of Lemma~\ref{lem:poisson-bracket-vanish-iff-gamma-takes-some-values-general-case} can be used to prove the claim. If we compare the $\mathcal{D}_i$-triangle chain and the $\mathcal{D}_{i+1}$-triangle chain of $[\rho]$, then we shall obtain a relation of the kind
\[
\cos(\delta_{i+1}/2)=\cos(\gamma_1^{\mathcal{D}_i})\cdot k_1+k_2,
\]
where $k_1$ and $k_2$ are constant along the $d_i$-orbit of $[\phi]$. Note that since we are assuming $n\geq 5$, the differential $(d\delta_{i+1})_{[\phi]}$ vanishes if and only if the first triangle in the $\mathcal{D}_{i+1}$-triangle chain of $[\phi]$ is degenerate by Fact~\ref{fact:characterization-fixed-point-Dehn-twist} and the same argument as in the proof of Claim~\ref{claim:tau_d_i-does-not-fix-any-point-in-A_jpi}. This can happen only when $\gamma_1^{\mathcal{D}_i}=0$.
\end{proofclaim}

We can now combine the previous two claims to prove the following useful statement.

\begin{lem}\label{lem:d_i->d_i+1}
If $[\phi]$ is a point of $\IntRepDT{\mathcal{B}}(\surface)$ with $\gamma_i([\phi])\notin \{0,\pi\}$ for some index $i\in\{1,\ldots,n-3\}$, then for every small enough open interval $J\subset \R$ around $0$, there exists densely many $(t_i,s_i)\in J^2$ for which
\[
\delta_{i+1}\left(\Phi_{i}^{(t_i,s_i)}([\phi])\right)\in\R\setminus\pi\Q.
\]
\end{lem}
\begin{proof}
Since we are assuming that $\gamma_i([\phi])\notin \{0,\pi\}$ and $[\phi]\in\IntRepDT{\mathcal{B}}(\surface)$, Lemma~\ref{lem:poisson-bracket-vanish-iff-gamma-takes-some-values-general-case} implies that $\{\beta_i,\delta_i\}([\phi])\neq 0$. Arguing as in Section~\ref{sec:local-parameterization} with the Inverse Function Theorem, we conclude that there exists a small enough interval 
$J\subset \R$ around $0$ such that 
\[
\Phi_i^{(t_i,s_i)}\colon J^2\to\RepDT
\]
is a diffeomorphism onto its image. Up to shrinking $J$, we may assume that for every $s_i\in J$, the point $\Phi_{b_i}^{s_i}([\phi])$ still satisfies $\gamma_i\neq 0$, and thus also $C_{i+1}\neq C_{i+2}$. Claim~\ref{claim:D_i-neq-C_i+3} says that $\Phi_{b_i}^{s_i}([\phi])$ also satisfies $D_i\neq C_{i+3}$ for all but finitely many $s_i\in J$. Along the $d_i$-orbits of every $\Phi_{b_i}^{s_i}([\phi])$ with $C_{i+1}\neq C_{i+2}$ and $D_i\neq C_{i+3}$ (which are all circles by assumption on $J$), the Poisson bracket $\{\delta_i,\delta_{i+1}\}$ only vanishes at finitely many points by Claim~\ref{claim:Poisson-bracket-d_i-d_i+1}. In particular, this means that the function $\delta_{i+1}$ takes values in $\R\setminus \pi\Q$ at densely many points along those $d_i$-orbits. Since this happens along the $d_i$-orbits of $\Phi_{b_i}^{s_i}([\phi])$ for all but finitely many $s_i\in J$, the proof of the lemma is complete.
\end{proof}

Recall that we already identified in~\eqref{eq:first-slice} a 2-dimensional slice $U_j$ that passes through $[\rho']$ and is contained in $\overline{\Orb([\rho])}$. By definition, $U_j$ lies in $\Im(\Psi)$ and so every point of $U_j$ satisfies the hypotheses of Lemma~\ref{lem:d_i->d_i+1}. This means that, up to further shrinking $I$ if necessary, the function $\delta_{j+1}$ takes values in $\R\setminus \pi\Q$ at densely many points of $U_j$. By using Fact~\ref{fact:Dehn-twists-irrational-rotations}, we conclude that 
\[
V_{j+1}=\left\{\Phi_{d_{j+1}}^{t_{j+1}}\circ\Phi_{j}^{(t_j,s_j)}([\rho']):t_j,t_{j+1},s_j\in I\right\}
\]
is entirely contained in $\overline{\Orb([\rho])}$. It is also contained in $\Im(\Psi)$ (take $s_{j+1}=0$). We have just extended $U_j$ by one dimension.
\begin{claim}\label{claim:d_i->b_i}
The function $\beta_{j+1}$ takes values in $\R\setminus \pi\Q$ at densely many points of $V_{j+1}$.
\end{claim}
\begin{proofclaim}
Recall that the relation $\{\beta_{j+1},\delta_{j+1}\}\neq 0$ holds at every point of $\Im(\Psi)$ by assumption. In particular, it also holds at every point of $U_j$. Now, if we apply Fact~\ref{fact:densely-many-irrational-points}, we shall conclude that $\beta_{j+1}$ is an irrational multiple of $\pi$ at densely many points along the $d_{j+1}$-orbit of every point in $U_j$. This proves the claim.
\end{proofclaim}
Claim~\ref{claim:d_i->b_i} implies that 
\[
U_{j+1}'=\left\{\Phi_{b_{j+1}}^{s_{j+1}}\circ\Phi_{d_{j+1}}^{t_{j+1}}\circ\Phi_{j}^{(t_j,s_j)}([\rho']):t_j,t_{j+1},s_j,s_{j+1}\in I\right\}
\]
is contained in $\overline{\Orb([\rho])}$. Using Lemma~\ref{lem:permuting-Hamiltonian-flows}, and after shrinking $I$ further if necessary, we deduce that the set 
\[
U_{j+1}=\left\{\Phi_{j+1}^{(t_{j+1},s_{j+1})}\circ\Phi_{j}^{(t_j,s_j)}([\rho']):t_j,t_{j+1},s_j,s_{j+1}\in I\right\}
\]
is also entirely contained in $\overline{\Orb([\rho])}$. We have now extended the slice $U_j$ by two extra dimensions. This procedure can be iterated until we shall obtain that
\[
U_{n-3}=\left\{\Phi_{n-3}^{(t_{n-3},s_{n-3})}\circ\cdots\circ\Phi_{j}^{(t_j,s_j)}([\rho']):t_j,s_j,\ldots,t_{n-3},s_{n-3}\in I\right\}
\]
is contained in $\overline{\Orb([\rho])}$. The set $U_{n-3}$ is an $(n-j-2)$-dimensional neighborhood of $[\rho']$ that is contained in $\overline{\Orb([\rho])}$. It is still not an open neighborhood of $[\rho]$ (except when $j=1$). The purpose of the next section is to enlarge $U_{n-3}$ to an open neighborhood of~$[\rho']$.

\subsubsection{Extending the slice by decreasing the index}\label{sec:decrease}
In order to enlarge $U_{n-3}$, we need an analogue of Lemma~\ref{lem:d_i->d_i+1} which shows that
\[
\delta_{i-1}\left(\Phi_{i}^{(t_i,s_i)}([\phi])\right)\in\R\setminus\pi\Q
\]
for densely many $(t_i,s_i)$ allowing us to extend $U_{n-3}$ along the direction of $\Phi_{d_{i-1}}$. We can obtain such a result by looking at different pants decompositions than the ones we considered in Section~\ref{sec:increase}. Consider the system of geometric generators of $\pi_1\Sigma$  given by
\[
(c_{i+1}, c_{i+2}, c_i, c_i^{-1}c_{i+3}c_i,\ldots, c_i^{-1}c_{i-1}c_i).
\]
We denote the associated standard pants decomposition by $\mathcal{D}'_i$. The first triangle in the $\mathcal{D}_i'$-triangle chain of any point in $\RepDT$ coincides with the first triangle in its $\mathcal{D}_i$-triangle chain because the first two generators coincide. They both have vertices $(C_{i+1},C_{i+2},D_i)$. The two triangle chains differ from the second triangle onward. For instance, the second triangle in a $\mathcal{D}_i'$-triangle chain has vertices $D_i$, $C_{i}$, and a third vertex that is the fixed point of $\phi((c_{i+1}c_{i+2}c_i)^{-1})$.
\begin{center}
\begin{tikzpicture}[font=\sffamily]
    
\node[anchor=south west,inner sep=0] at (0,0) {\includegraphics[width=6cm]{fig/fig-triangle-chain-3}};

\begin{scope}
\fill (0.25,1.35) circle (0.07) node[left]{$B_{i-2}$};
\fill (1.25,2.7) circle (0.07) node[left]{$B_{i-1}$};
\fill (2.85,2.25) circle (0.07) node[below left]{$B_{i}$};
\fill (2.87,0.3) circle (0.07) node[left]{$B_{i+1}$};
\fill (0.25,5.45) circle (0.07) node[left]{$C_{i}$};
\fill (4.35,5.4) circle (0.07) node[right]{$C_{i+1}$};
\fill (5.7,1.5) circle (0.07) node[right]{$C_{i+2}$};
\fill[mauve] (2.3,3.1) circle (0.07) node[left]{$D_{i}$};
\end{scope}

\draw[mauve] (0.25,5.45) to[bend left=15] (2.3,3.1);
\draw[mauve] (4.35,5.4) to[bend left=5] (5.7,1.5);
\draw[mauve] (2.3,3.1) to[bend left=10] (4.35,5.4);
\draw[mauve] (2.3,3.1) to[bend left=25] (5.7,1.5);

\draw[mauve, thick] (5,2.2) arc (140:98:.7) node[at end, right]{\small $\pi-\alpha_{i+2}/2$};
\draw[mauve, thick] (3.9,5) arc (220:294:.6) node[at end, right]{\small $\pi-\alpha_{i+1}/2$};
\draw[plum, thick] (3.1,3.05) arc (-5:120:.8);
\draw[plum] (2.1,4.4) node{\small $\gamma_1^{\mathcal{D}'_i}$};
\draw[mauve, thick] (3.3,3) arc (-5:58:.95) node[near start, right]{\small $\pi-\delta_{i}/2$};
\end{tikzpicture}
\end{center}

An analogous statement to Claim~\ref{claim:D_i-neq-C_i+3} holds. It can be proved using similar arguments.

\begin{claim}\label{claim:D_i-neq-C_i}
If $[\phi]$ is any point of $\IntRepDT{\mathcal{B}}(\surface)$ with $\gamma_i([\phi])\neq 0$ for some index $i$, then only finitely many points along the $b_i$-orbit of $[\phi]$ have $D_i=C_{i}$.
\end{claim}

If the $\mathcal{D}_i'$-triangle chain of a point $[\phi]$ satisfies $C_{i+1}\neq C_{i+2}$ and $D_i\neq C_i$, then its first two triangle are non-degenerate. So, there is a well-defined angle coordinate $\gamma_i^{\mathcal{D}_i'}$ which the angle between the geodesic rays $[D_iC_{i+2})$ and $[D_i,C_i)$. We can reformulate Claim~\ref{claim:Poisson-bracket-d_i-d_i+1} in those terms.

\begin{claim}\label{claim:Poisson-bracket-d_i-d_i-1}
If $[\phi]$ is a point with $C_{i+1}\neq C_{i+2}$ and $D_i\neq C_{i}$ for some index $i$, then 
\[
\{\delta_{i-1},\delta_{i}\}([\phi])=0\quad\Leftrightarrow\quad \gamma_1^{\mathcal{D}_i'}\in\{\pi-\delta_i/2,2\pi-\delta_i/2\}.
\]
\end{claim}

When we combine Claims~\ref{claim:D_i-neq-C_i} and~\ref{claim:Poisson-bracket-d_i-d_i-1}, we obtain the desired variant of Lemma~\ref{lem:d_i->d_i+1} that will allow us to enlarge $U_{n-3}$.

\begin{lem}\label{lem:d_i->d_i-1}
If $[\phi]$ is a point of $\IntRepDT{\mathcal{B}}(\surface)$ with $\gamma_i([\phi])\notin \{0,\pi\}$ for some index $i\in\{1,\ldots,n-3\}$, then for every small enough open interval $J\subset \R$ around $0$, there exist densely many $(t_i,s_i)\in J^2$ for which
\[
\delta_{i-1}\left(\Phi_{i}^{(t_i,s_i)}([\phi])\right)\in\R\setminus\pi\Q.
\]
\end{lem}

Here is how we can keep enlarging $U_{n-3}$. First, we use Lemma~\ref{lem:permuting-Hamiltonian-flows} to modify $U_{n-3}$ slightly and affirm that 
\[
W_{n-3}=\left\{\Phi_{j}^{(t_j,s_j)}\circ \Phi_{n-3}^{(t_{n-3},s_{n-3})}\circ\cdots\circ\Phi_{j+1}^{(t_{j+1},s_{j+1})}([\rho']):t_j,s_j,\ldots,t_{n-3},s_{n-3}\in I\right\}
\]
is also contained in $\overline{\Orb([\rho'])}$ (we might have to shrink $I$ if necessary). Now, we can apply Lemma~\ref{lem:d_i->d_i-1} and conclude that, after maybe shrinking $I$ further,
\[
V_{j-1}=\bigcup_{t_{j-1}\in I}\Phi_{d_{j-1}}^{t_{j-1}}(W_{n-3})\subset \overline{\Orb([\rho'])}.
\]
If we apply Claim~\ref{claim:d_i->b_i}, we shall further obtain that 
\[
U_{j-1}'=\bigcup_{t_{j-1},s_{j-1}\in I}\Phi_{b_{j-1}}^{s_{j-1}}\circ \Phi_{d_{j-1}}^{t_{j-1}}(W_{n-3}) \subset \overline{\Orb([\rho'])}.
\]
Lemma~\ref{lem:permuting-Hamiltonian-flows} allows us, up to shrinking $I$ further, to permute Hamiltonian flows, so that we have 
\[
U_{j-1}=\bigcup_{t_{j-1},s_{j-1}\in I}\Phi_{j-1}^{(t_{j-1},s_{j-1})}(W_{n-3})\subset \overline{\Orb([\rho'])}.
\]
Repeating this process, we shall eventually obtain that 
\[
U_{1}=\bigcup_{t_1,\ldots,s_{j-1}\in I}\Phi_{1}^{(t_{1},s_{1})}\circ\cdots\circ\Phi_{j-1}^{(t_{j-1},s_{j-1})}(W_{n-3})\subset \overline{\Orb([\rho'])}.
\]
After applying Lemma~\ref{lem:permuting-Hamiltonian-flows} one last time and shrinking $I$ if necessary, we conclude that
\[
U_{[\rho']}=\left\{\Phi_{1}^{(t_{1},s_{1})}\circ\cdots\circ \Phi_{n-3}^{(t_{n-3},s_{n-3})}([\rho']):t_i,s_i\in I\right\}
\]
is entirely contained in $\overline{\Orb([\rho'])}$. The set $U_{[\rho']}$ is the desired an open neighborhood of~$[\rho']$. 

\subsection{Finishing the proof}\label{sec:finishing-the-proof}
So far, we have proven that when $\Orb([\rho])$ is infinite, then it contains an orbit point $[\rho']$ for which we were able to construct an open neighborhood $U_{[\rho']}$ that is contained in $\overline{\Orb([\rho])}$. This shows that $\overline{\Orb([\rho])}$ is a closed invariant set with non-empty interior. We conclude that $\overline{\Orb([\rho])}=\RepDT$ because of the ergodicity of the mapping class group action on $\RepDT$ (Theorem~\ref{thm:ergodicity}), as we explained in Remark~\ref{rem:alternative-end-of-proof}. This completes the proof of Theorem~\ref{thm:infinite-orbits-are-dense}.

\bibliographystyle{amsalpha}
\bibliography{bibliography.bib}
\end{document}